\documentclass[10pt]{article}
\usepackage{enumitem}
\usepackage{amsfonts}
\usepackage{amsmath}
\usepackage{graphicx}
\usepackage{epstopdf}
\usepackage{fancyhdr}
\usepackage{amsthm}
\usepackage{geometry}
\usepackage{subfigure}
\usepackage[figuresright]{rotating}
\usepackage{caption}
\usepackage{bm}
\usepackage{amssymb}
\usepackage{amscd}
\usepackage{float}
\usepackage{color}
\usepackage{tabularx}
\usepackage{diagbox}
\usepackage{tabularx}
\usepackage{mathtools}
\usepackage{multicol}
\usepackage{multirow}
\usepackage{makecell}
\usepackage[justification=centering]{caption}
\usepackage{rotfloat}
\usepackage{longtable}
\usepackage{arydshln}
\usepackage{algorithm}
\usepackage{algpseudocode}
\allowdisplaybreaks[4]

\newtheorem{theorem}{Theorem}[section]
\newtheorem{corollary}{Corollary}[section]

\newtheorem{remark}{Remark}[section]
\newtheorem{proposition}{Proposition}[section]
\newtheorem{definition}{Definition}[section]

\newcounter{nextauthor}
\def\mathrm{\mbox}

\numberwithin{remark}{section}

\begin{document}
\title{{\bf  Equilibrium reinsurance and investment strategies for insurers with random risk aversion under Heston's SV model }\thanks{This work was supported by the National Natural Science Foundation of China (12171339), the Scientific and Technological Research Program of Chongqing Municipal Education Commission (KJQN202400819), the grant from Chongqing Technology and Business University (2356004) and the Fundamental Research Funds for the Central Universities (2682023CX071).}}
\author{Jian-hao Kang$^a$, Zhun Gou$^b$ and Nan-jing Huang$^c$ \thanks{Corresponding author: nanjinghuang@hotmail.com; njhuang@scu.edu.cn}\\
{\small a. School of Mathematics, Southwest Jiaotong University, Chengdu, Sichuan 610031, P.R. China}\\
{\small b. College of Mathematics and Statistics, Chongqing Technology and Business University,}\\
{\small Chongqing 400067, P.R. China}\\
{\small c. Department of Mathematics, Sichuan University, Chengdu, Sichuan 610064, P.R. China}}
\date{}
\maketitle \vspace*{-9mm}
\begin{abstract}
\noindent
This study employs expected certainty equivalents to explore the reinsurance and investment issue pertaining to an insurer that aims to maximize the expected utility while being subject to random risk aversion. The insurer's surplus process is modeled approximately by a drifted Brownian motion, and the financial market is comprised of a risk-free asset and a risky asset with its price depicted by Heston's stochastic volatility (SV) model. Within a game theory framework, a strict verification theorem is formulated to delineate the equilibrium reinsurance and investment strategies as well as the corresponding value function. Furthermore, through solving the pseudo Hamilton-Jacobi-Bellman (HJB) system, semi-analytical formulations for the equilibrium reinsurance and investment strategies and the associated value function are obtained under the exponential utility. Additionally, several numerical experiments are carried out to demonstrate the characteristics of the equilibrium reinsurance and investment strategies.
\end{abstract}

\section{Introduction}
Since the seminal work by Merton \cite{Merton1969, Merton1971}, the investment problem within the framework of expected utility maximization in a dynamic setting has emerged as one of the highly favored research avenues in the realm of financial economics. Based on the work \cite{Merton1969, Merton1971}, a large strand of literature has been devoted to expanding this research. For example, Kim and Omberg \cite{Kim1996} investigated the dynamic portfolio problem with stochastic risk premium for hyperbolic absolute risk aversion utility functions; Liu \cite{Liu2007} explicitly solved dynamic portfolio choice problems with quadratic asset returns under the power utility; Liu and Loewenstein \cite{Liu2002} studied the portfolio selection with transaction costs under the power utility. On the other hand, reinsurance serves as an efficient approach for insurers to manage and mitigate risks. Richard \cite{Richard1975} firstly incorporated life insurance in the investment problem under Merton's framework. Since then, many scholars have paid much attention to the dynamic reinsurance and investment optimization problems within the expected utility maximization criterion. For instance, Bai and Guo \cite{Bai2010} explored the issue of optimal dynamic excess-of-loss reinsurance as well as the multidimensional portfolio selection, with the focus being on the context of the exponential utility function; Yi et al. \cite{Yi2013} investigated the investment and reinsurance optimization problem with model uncertainty under the exponential utility; Deng et al. \cite{Deng2018} studied the consequences of strategic interplay between two insurers with regard to their reinsurance-investment strategies within the framework of the exponential utility function. For more related studies about reinsurance or/and investment problems, we refer the reader to the literature by, e.g., Li et al. \cite{Li2014}, Liang and Bayraktar \cite{Liang2014}, Yang and Zhang \cite{Yang2005}, Yuen et al. \cite{Yuen2015}, just mention a few.

For the reinsurance and investment problems within the expected utility maximization criterion, the time-inconsistency emerges under many cases such as the problems with non-exponential discounting factor, in which an optimal reinsurance-investment strategy obtained now may no longer be optimal in the future. It is well known that the main methods of handling time-inconsistency are to look for time-consistent equilibrium strategies under a game-theoretic framework instead of optimal strategies. Several treatments are available in accordance with different methodologies. For example, Bj{\"{o}}rk et al. \cite{Bjork2017} and Bj{\"{o}}rk and Murgoci \cite{Bjork2014} followed the classical dynamic programming framework and derived an extended HJB equation to describe the equilibrium; Hu et al. \cite{Hu2012} characterized the game under the framework of open-loop equilibrium control and obtained the equilibrium by the stochastic maximum principle; Yong \cite{Yong2012} introduced the so-called equilibrium HJB equation to formulate the equilibrium strategy  in the context of a multi-person differential game featuring a hierarchical structure. More studies and results concerning the time-inconsistent problems can be found in the literature (see, for instance, Bensoussan et al. \cite{Bensoussan2019}, Bj{\"{o}}rk et al. \cite{Bjork2021}, Ekeland and Pirvu \cite{Ekeland2008}, He and Jiang \cite{He2021}, Hern\'{a}ndez and Possama\"{\i} \cite{H2023}, Li et al. \cite{Li2012} and the references therein).

One common point of the above literature is that the risk aversion coefficient of the investor/insurer is assumed to be constant. However, in real life, it is not feasible for the investor/insurer to simply determine the precise value of his/her risk aversion coefficient. As a result, the estimation of the precise value of risk aversion coefficient has been extensively studied through different methods, for example, by exploiting data regarding labor income in the market \cite{Chetty2006}, by using market data concerning implied and realized volatilities \cite{Bollerslev2011}, by conducting questionnaire survey on representative groups \cite{Burgaard2020}, just mention a few. Recently, Desmettre and Steffensen \cite{Desmettre2023} pointed out that the constant risk aversion should be replaced by the random variable and they used expected certainty equivalents to establish the optimal investment problem in a traditional Black-Scholes financial market setting, which involves random risk aversion and is time-inconsistent. They also derived equilibrium investment strategies in the scenarios of power utility and exponential utility. It should be noted that the price of the risky asset in \cite{Desmettre2023} was supposed to adhere to geometric Brownian motion, indicating that the volatility of said risky asset's price is of a deterministic nature. Nevertheless, lots of empirical evidences indicate that the constant volatility assumption contradicts to the phenomenon of the volatility smile and find that this phenomenon can be well explained by the SV models like the Hull-White model \cite{hull1987} and the Heston model \cite{heston1993closed} as well as the 3/2 SV model \cite{Drimus2012}. Especially, besides the advantage of deriving formulas in a closed-form for problems of pricing, the Heston model is advantageous in explaining many phenomena in empirical research, such as volatility clustering and heavy-tailed return distributions. It has been recognized as a traditional and significant SV model, finding extensive applications in option pricing \cite{Cui2017, He2024} as well as portfolio selection scenarios \cite{Bergen2018, Kraft2013}. Therefore, one natural question is: can we consider the reinsurance and investment optimization problem involving random risk aversion when Heston's SV model is utilized to describe the process of the risky asset's price?

Inspired and motivated by the above researches, the present paper is thus devoted to answering this question by investigating a general reinsurance-investment optimization problem involving random risk aversion within a financial market in which the price process of the risky asset is regulated by Heston's SV model. To be specific, we adopt a drifted Brownian motion to approximate the surplus process and use the approach of expected certainty equivalents to capture the general reinsurance-investment optimization problem with random risk aversion that can be seen as a sum over nonlinear functions of expectations, which is time-inconsistent since Bellman's dynamic programming does not hold. In this paper, we follow the framework of \cite{Bjork2017, Bjork2014, Desmettre2023} to handle the time-inconsistency issue for the general reinsurance-investment optimization problem with random risk aversion. In particular, we derive the pseudo HJB system to capture the equilibrium reinsurance and investment strategies. However, the consideration of Heston's SV makes it difficult and complex to solve the pseudo HJB system because the technique of separating variables is hard to be applied, such as for the power utility. Through complicated calculations and derivations, we solve the pseudo HJB system under the exponential utility and obtain semi-closed form representations regarding the equilibrium reinsurance and investment strategies. On the other hand, in the case where the insurer's utility is either the exponential utility or the power utility, Heston's SV model could give rise to an infinite value function. Therefore, the methods used in \cite{Desmettre2023} to make sure the well-defined value function do not apply here in a straightforward manner. Instead, we follow \cite{ZengX2013} to set certain technical prerequisites for model parameters and then use the martingale technique to ensure that the value function is properly defined.

The key contributions within this paper are recapitulated as follows: (i) apart from the fact that the risky asset's price is captured by Heston's SV model, we pioneer to incorporate random risk aversion to the reinsurance-investment optimization problem in a dynamic environment, which extends the study in \cite{Desmettre2023} to some extent and enriches the research on time-inconsistent reinsurance-investment problems; (ii) a rigorous verification theorem is provided for the general reinsurance and investment optimization problem with random risk aversion; (iii) under the exponential utility, semi-closed form formulations for the equilibrium reinsurance and investment strategies, along with the corresponding value function, are derived, and the admissibility prerequisites and relevant suppositions of the verification theorem are rigorously proven under mild conditions, which shows that the calculations and proofs in this paper are more complex than those in \cite{Desmettre2023}; (iv) numerical experiments are conducted to analyze the effects of model parameters with respect to the equilibrium reinsurance and investment strategies and some interesting findings are derived.

The subsequent part of the paper is arranged as follows. Section 2 furnishes the model setup and assumptions. Section 3 formulates the reinsurance and investment optimization problem involving random risk aversion and presents a verification theorem to capture the equilibrium reinsurance and investment strategy. Section 4 derives the (semi-)closed form formulations of the equilibrium reinsurance and investment strategies for exponential utility under the cases of $n$ possible risk aversions and one risk aversion, respectively. Section 5 provides some numerical illustrations to demonstrate our main outcomes. Section 6 draws the conclusion of this paper and appendices contain the proofs.

\section{Model setting and assumptions}
In this paper, we hypothesize that the insurance and financial markets are devoid of transaction costs and taxes, and that trading can be continuous. Let $(\Omega,\mathcal{F},\mathbb{F},\mathbb{P})$ be a complete filtered probability space equipped with a finite time horizon $[0,T]$, where $\mathbb{F}:=\{\mathcal{F}_{t}\}_{t\in[0,T]}$ is a right-continuous and $\mathbb{P}$-complete filtration. Moreover, we suppose that any decisions taken at moment $t$ are founded on $\mathcal{F}_{t}$, and every stochastic process introduced hereafter is both well-defined and $\mathcal{F}_{t}$-adapted.

\subsection{Surplus process}
Without reinsurance and investment, we take it as an assumption that the surplus process of an insurer is captured through the classical Crem$\acute{e}$r-Lundberg model as follows
\begin{equation}\label{eq1}
  R(t)=\tilde{u}_{0}+c t-\sum\limits_{i=1}^{N(t)}Z_{i},
\end{equation}
where $\tilde{u}_{0}\geq0$ represents the initial surplus, $c$ stands for the premium rate, $\sum\limits_{i=1}^{N(t)}Z_{i}$ denotes a compound Poisson process which characterizes the aggregate claims up to time $t$, $\{N(t)\}_{t\in[0,T]}$ is a homogeneous Poisson process with a positive intensity  $\lambda_{1}$ that represents the count of claims taking place within the time span $[0,t]$ and the claim size $\{Z_{i}\}_{i\in\mathbb{N}}$ is independent of $\{N(t)\}_{t\in[0,T]}$, which is hypothesized to constitute a series of random variables that are positive, independent and identically distributed, and have a finite first moment $\mathbb{E}[Z_{i}]=\mu_{1}>0$ and a second moment $\mathbb{E}[Z^{2}_{i}]=\mu_{2}>0$. Furthermore, the premium rate $c$ is postulated to be ascertained based on the expected value premium principle, that is, $c=(1+\eta_{1})\lambda_{1}\mu_{1}$, where $\eta_{1}$ denotes the insurer's relative safety loading.

We further hypothesize that the insurer can manage her/his insurance risk via the purchase of proportional reinsurance or the acquisition of new business. Specially, we denote by $q(t)$ the insurer's proportional reinsurance/new business degree at time $t$. If $q(t)\in[0,1]$, then it represents the case of purchasing proportional reinsurance. It means that the reinsurer will make compensation to the insurer for $100(1-q(t))\%$ of the claims occurring at time $t$ and the insurer pays $100q(t)\%$ of the claims at the same time. Postulate that the rate of reinsurance premium is also ascertained based on the expected value premium rule. Therefore, the reinsurance premium should be paid at a rate of $(1+\eta_{2})(1-q(t))\lambda_{1}\mu_{1}$, where $\eta_{2}(\geq\eta_{1})$ denotes the safety loading of the reinsurer and the condition $\eta_{2}\geq\eta_{1}$ is necessary to prevent the insurer from engaging in arbitrage. If $q(t)>1$, then it relates to the instance of procuring new business, suggesting that the insurer can play the role of a reinsurer for other insurance providers. For simplicity, $\{q(t)\}_{t\in[0,T]}$ is called a reinsurance strategy in the sequel. Hence, under the condition of such a reinsurance strategy, the insurer's surplus process dynamics can be illustrated by
\begin{align}\label{eq2}
   R(t)&=\tilde{u}_{0}+\int^{t}_{0}\left[c-(1+\eta_{2})(1-q(s))\lambda_{1}\mu_{1}\right]\mathrm{d}s-\sum\limits_{i=1}^{N(t)}q(T_{i})Z_{i}\nonumber\\
   &=\tilde{u}_{0}+\int^{t}_{0}\left[\eta+(1+\eta_{2})q(s)\right]\lambda_{1}\mu_{1}\mathrm{d}s-\sum\limits_{i=1}^{N(t)}q(T_{i})Z_{i}.
\end{align}
Here, $\eta=\eta_{1}-\eta_{2}\leq0$ and $T_{i}$ represents the $i$-th claim's occurring time. In line with the research in \cite{Grandell1990}, the dynamics of $R(t)$ in \eqref{eq2} can be reformulated by the subsequent approximation diffusion model
\begin{align}\label{eq3}
   \mathrm{d}R(t)=a[\eta+\eta_{2}q(t)]\mathrm{d}t+bq(t)\mathrm{d}W_{0}(t),
\end{align}
where $a=\lambda_{1}\mu_{1}$, $b=\sqrt{\lambda_{1}\mu_{2}}$ and $W_{0}:=\{W_{0}(t)\}_{t\in[0,T]}$ denotes a standard $\mathbb{P}$-Brownian motion.

\subsection{Financial market}
We consider a financial market that consists of a risk-free asset (e.g., banking deposit) and a risky asset (e.g., stock). Specifically, the price $S_{0}(t)$ of the risk-free asset follows
\begin{equation}\label{eq4}
  \mathrm{d}S_{0}(t)=rS_{0}(t)\mathrm{d}t, \quad S_{0}(0)=s_{0}>0,
\end{equation}
where the positive constant $r$ is the risk-free interest rate. Moreover, we presume that the price $S_{1}(t)$ of the risky asset adheres to Heston's SV model
\begin{equation}
   \left\{ \begin{aligned}\label{eq5}
   &\mathrm{d}S_{1}(t)=S_{1}(t)\left[(r+\xi V(t))\mathrm{d}t+\sqrt{V(t)}\mathrm{d}W_{1}(t)\right], \quad S_{1}(0)=s_{1}>0,\\
   &\mathrm{d}V(t)=\kappa(\theta-V(t))\mathrm{d}t+\sigma\sqrt{V(t)}\mathrm{d}W_{2}(t), \quad V(0)=v_{0}>0,
  \end{aligned}\right.
\end{equation}
where $\xi$, $\kappa$, $\theta$ and $\sigma$ are positive constants respectively standing for the premium for volatility, the mean-reversion rate, the long-run mean and the volatility of volatility. In addition, $W_{1}:=\{W_{1}(t)\}_{t\in[0,T]}$ and $W_{2}:=\{W_{2}(t)\}_{t\in[0,T]}$ are standard $\mathbb{P}$-Brownian motions with $Cov(W_{1}(t),W_{2}(t))=\rho t,\rho\in[-1,1]$. We hypothesize that $W_{0}$ is independent of $W_{1}$ and $W_{2}$. We also presume that $2\kappa \theta>\sigma^{2}$ to guarantee $V(t)$ is almost surely non-negative.

\subsection{Wealth process}
At the initial moment $t=0$ with a positive capital $x_{0}$, the insurer is permitted to either engage in proportional reinsurance procurement or new business acquisition in the insurance market and to allocate funds for investment in the financial market. Then the trading strategy is represented by a pair of stochastic processes $u(t):=(q(t),\pi(t))$, where $\pi(t)$ signifies the dollar quantities placed into the risky asset at time $t$. We denote the wealth process with the strategy $u(t)$ by $X^{u}(t)$. Therefore, the dollar amounts $X^{u}(t)-\pi(t)$ is invested in the risk-free asset. Furthermore, with $X^{u}(0)=x_{0}$, the variation of the insurer's wealth process can be expressed through the subsequent stochastic differential equation (SDE)
\begin{align}\label{eq6}
   \mathrm{d}X^{u}(t)&=\mathrm{d}R(t)+\frac{X^{u}(t)-\pi(t)}{S_{0}(t)}\mathrm{d}S_{0}(t)+\frac{\pi(t)}{S_{1}(t)}\mathrm{d}S_{1}(t)\nonumber\\
   &=[rX^{u}(t)+a\eta+a\eta_{2}q(t)+\xi V(t)\pi(t)]\mathrm{d}t+bq(t)\mathrm{d}W_{0}(t)+\pi(t)\sqrt{V(t)}\mathrm{d}W_{1}(t).
\end{align}

\section{Problem formulation and verification theorem}
\setcounter{equation}{0}
In this section, we will first characterize the reinsurance and investment optimization problem involving random risk aversion and then provide a verification theorem to capture the equilibrium reinsurance and investment strategy under the framework of game theory.

As we all know, a class of classical reinsurance and investment optimization problems aims to find a reinsurance and investment strategy maximizing the expected utility from the insurer's terminal wealth, that is,
\begin{equation}\label{eq7}
  \widehat{V}(t,x,v):=\sup\limits_{u} \; \mathbb{E}_{t}\left[\varphi(X^{u}(T))\right].
\end{equation}
Here, $X^{u}(t)=x$, $V(t)=v$, $\mathbb{E}_{t}$ represents the $\mathcal{F}_{t}$-conditional expectation in regard to the measure $\mathbb{P}$, $\varphi:(-\infty,\infty)\rightarrow \mathbb{R}$ is a utility function and $X^{u}(T)$ denotes the terminal wealth of the insurer. A reinsurance and investment strategy is called optimal if it can achieve the supremum in \eqref{eq7}. Equivalently, the insurer's objective is to optimize the certainty equivalent of her/his terminal wealth, which amounts to maximizing the reward functional
\begin{equation}\label{eq8}
  \widehat{J}^{u}(t,x,v):=\varphi^{-1}\left(\mathbb{E}_{t}\left[\varphi(X^{u}(T))\right]\right),
\end{equation}
where $\varphi^{-1}(\cdot)$ denotes the inverse function of $\varphi(\cdot)$.

For the reward functional \eqref{eq8}, we follow the study of \cite{Desmettre2023} to assume that the risk aversion coefficient $\gamma$ of the utility function $\varphi$ is a real-valued random variable. Moreover, on account of the idea of maximizing the certainty equivalent of terminal wealth with respect to the random risk aversion, we aim to maximize the subsequent reward functional
\begin{equation}\label{eq9}
  J^{u}(t,x,v):=\int(\varphi^{\gamma})^{-1}\left(\mathbb{E}_{t}\left[\varphi^{\gamma}(X^{u}(T))\right]\right)\mathrm{d}\Gamma(\gamma),
\end{equation}
where $\Gamma$ denotes the cumulative distribution function corresponding to $\gamma$ and the integral is taken over the support of $\Gamma$. Note that we adopt the notation $\varphi^{\gamma}$ to highlight the dependence of $\varphi$ on the risk aversion coefficient $\gamma$. For the ease of notations, in what follows, we adopt the following symbols:
\begin{equation}\label{eq10}
  y^{u,\gamma}(t,x,v):=\mathbb{E}_{t}\left[\varphi^{\gamma}(X^{u}(T))\right]
\end{equation}
and
\begin{equation}\label{eq11}
  \iota^{\gamma}(y):=\frac{\mathrm{d}}{\mathrm{d} y}(\varphi^{\gamma})^{-1}(y).
\end{equation}
Therefore, the reward functional \eqref{eq9} can be rewritten as
\begin{equation}\label{eq12}
  J^{u}(t,x,v)=\int(\varphi^{\gamma})^{-1}\left(y^{u,\gamma}(t,x,v)\right)\mathrm{d}\Gamma(\gamma)
\end{equation}
and then the insurer's goal is to maximize the reward functional \eqref{eq12}.

It follows from \eqref{eq10} and \eqref{eq12} that the reward functional \eqref{eq12} can be seen as a sum over nonlinear functions of expectations, which results in the issue of time-inconsistency. In order to deal with time-inconsistency and to find a time-consistent reinsurance and investment strategy, we follow the studies of \cite{Bjork2017, Bjork2021, Desmettre2023, Li2012} to seek an equilibrium reinsurance and investment strategy by game theory. Now, we first give the definition of admissible reinsurance and investment strategy as follows.
\begin{definition}\label{definition3.1}(Admissible reinsurance and investment strategy)
Let $\mathcal{O}:=\mathbb{R}\times\mathbb{R}^{+}$ and $\mathcal{Q}:=[0,T]\times\mathcal{O}$. For any fixed $t\in[0,T]$, a reinsurance and investment strategy $u(t)=(q(t),\pi(t))$ is said to be admissible if
\begin{itemize}
\item[(i)] for each initial point $(t,x,v)\in\mathcal{Q}$, the SDE \eqref{eq6} admits a unique strong solution $X^{u}$;
\item[(ii)] $\{u(t)\}_{t\in[0,T]}$ is $\mathcal{F}_{t}$-progressively measurable and continuous, $q(t)\geq0$ and $\mathbb{E}\left[\int^{T}_{0}(q^{2}(t)+\pi^{2}(t))\mathrm{d}t\right]<\infty$;
\item[(iii)] $\int\left|(\varphi^{\gamma})^{-1}\left(\mathbb{E}_{t}\left[\varphi^{\gamma}(X^{u}(T))\right]\right)\right|\mathrm{d}\Gamma(\gamma)<\infty$.
\end{itemize}
Let $\Pi$ be the set of all admissible reinsurance and investment strategies.
\end{definition}

Next, in the context of game theory, we characterize the following definition of equilibrium reinsurance and investment strategy.
\begin{definition}\label{definition3.2}(Equilibrium reinsurance and investment strategy)
For any arbitrarily chosen initial state $(t,x,v)\in\mathcal{Q}$, consider an admissible reinsurance and investment strategy $\hat{u}(t,x,v)=(\hat{q}(t,x,v),\hat{\pi}(t,x,v))$. Choose three real numbers $q>0$, $\pi$ and $h>0$ and define the following reinsurance and investment strategy
\begin{align}\label{eq13}
\left(q_{h}(s,\widetilde{x},\widetilde{v}),\pi_{h}(s,\widetilde{x},\widetilde{v})\right)=
\begin{cases}
(q,\pi), \quad\quad\quad\quad\quad\;\quad\quad\;\text{for}\; \quad\, t\leq s < t+h,(\widetilde{x},\widetilde{v})\in \mathbb{R} \times \mathbb{R}^{+},\\
(\hat{q}(s,\widetilde{x},\widetilde{v}),\hat{\pi}(s,\widetilde{x},\widetilde{v})), \quad\;\text{for}\;\quad  t+h\leq s < T,(\widetilde{x},\widetilde{v})\in \mathbb{R} \times \mathbb{R}^{+}.
\end{cases}
\end{align}
Let $u_{h}(s,\widetilde{x},\widetilde{v}):=(q_{h}(s,\widetilde{x},\widetilde{v}),\pi_{h}(s,\widetilde{x},\widetilde{v}))$. If $$\underset{h\rightarrow0}\liminf\frac{J^{\hat{u}}(t,x,v)-J^{u_{h}}(t,x,v)}{h}
\geq0
$$
for all $(q,\pi)\in\Pi$, then $\hat{u}(t,x,v)$ is called an equilibrium reinsurance and investment strategy. Moreover, the equilibrium value function can be defined by $V(t,x,v):=J^{\hat{u}}(t,x,v)$.
\end{definition}

\begin{remark}
We would like to mention that we consider the closed-loop equilibrium reinsurance and investment strategy in this paper. The idea in Definition \ref{definition3.2} is similar to the ones in \cite{Bjork2017, Bjork2021, Desmettre2023, Li2012}. The reinsurance and investment problem within $[0,T]$ is considered to be a non-cooperative game participated in by countless insurers indexed by $t\in[0,T]$, who can merely manipulate the wealth process $X^{u}$ at time $t$ through the selection of $u(t,X^{u}(t),V(t))$. The time-consistent reinsurance and investment strategy $\hat{u}$ is a subgame perfect Nash equilibrium. That is to say, for any $t\in[0,T)$, supposing insurers $s\in(t,T]$ individually decide on the reinsurance and investment strategy $\{\hat{u}(s,X^{u}(s),V(s))\}_{s\in(t,T]}$, it will be optimal for insurer $t$ to opt for $\hat{u}(t,X^{u}(t),V(t))$. Since each insurer $t$ conceptually solves the reinsurance and investment problem at the specific moment of time $t$, that is, on a time set having a Lebesgue measure of $0$, the reinsurance and investment strategy exerts no impact. Thus, for every time instant $t$ and a small positive $h$ representing the minimal time passage, insurer $t$ considers the reinsurance and investment problem within $[t,t+h)$ on the condition that insurers $s\in[t+h,T]$ have opted for the optimal reinsurance and investment strategy.
\end{remark}

Based on Definition \ref{definition3.2}, the insurer's aim is to maximize the reward functional \eqref{eq12} by finding the equilibrium reinsurance and investment strategy. For convenience, we introduce some notations. Denote
\begin{align*}
C^{1,2,2}(\mathcal{Q})=&\left\{\phi(t,x,v)|\phi(t,\cdot,\cdot)\; \text{is once continuously differentiable on} \;[0,T]\; \text{and}\right. \\
 &\;\left.\phi(\cdot,x,v)\;\text{is twice continuously differentiable on}\; \mathcal{O} \right\}
\end{align*}
and for any $\phi(t,x,v)\in C^{1,2,2}(\mathcal{Q})$ and $u=(q,\pi)$, let
\begin{align}\label{eq14}
  \mathcal{A}^{u}\phi(t,x,v)=&\phi_{t}(t,x,v)+(rx+a\eta+a\eta_{2}q+\xi v\pi)\phi_{x}(t,x,v)+0.5(b^{2}q^{2}+\pi^{2}v)\phi_{xx}(t,x,v)\nonumber\\
&+\kappa(\theta-v)\phi_{v}(t,x,v)+0.5\sigma^{2}v\phi_{vv}(t,x,v)+\rho\pi\sigma v\phi_{xv}(t,x,v).
\end{align}
In addition, for $i\in \mathbb{N}$, we define the set $\mathbb{L}^{i}_{\mathcal{F}}(0,T;\mathbb{R})$ as the set of $\mathcal{F}_{t}$-adapted stochastic processes $Z(\cdot)$ with $\mathbb{E}\left[\int^{T}_{0}\parallel Z(t)\parallel_{i}\mathrm{d}t\right]<\infty$, where $\parallel \cdot\parallel_{i}$ denotes the $i$-norm. Then, for any $\phi(t,x,v)\in C^{1,2,2}(\mathcal{Q})$ and $u=(q,\pi)$, we assume that
\begin{itemize}
  \item[(A)] $\phi_{t},\;(rx+a\eta+a\eta_{2}q+\xi v\pi)\phi_{x},\;\kappa(\theta-v)\phi_{v},\;(b^{2}q^{2}+\pi^{2}v)\phi_{xx},\;\sigma^{2}v\phi_{vv},\;\rho\pi\sigma v\phi_{xv}\in\mathbb{L}^{1}_{\mathcal{F}}(0,T;\mathbb{R})$ and $bq\phi_{x},\;\pi\sqrt{v}\phi_{x},\;\sigma\sqrt{v}\phi_{v}\in\mathbb{L}^{2}_{\mathcal{F}}(0,T;\mathbb{R})$.
\end{itemize}

Subsequently, we are going to offer a verification theorem to capture the equilibrium reinsurance and investment strategy.
\begin{theorem}\label{theorem3.1}(Verification theorem)
Assume that there exist functions $U, Y^{\gamma}, H\in C^{1,2,2}(\mathcal{Q})$ such that
\begin{align}\label{eq15}
\sup\limits_{u} \;\left\{\mathcal{A}^{u}U(t,x,v)-\mathcal{A}^{u}H(t,x,v)
+\int\iota^{\gamma}(Y^{\gamma}(t,x,v))\mathcal{A}^{u}Y^{\gamma}(t,x,v)\mathrm{d}\Gamma(\gamma)\right\}=0
\end{align}
and
\begin{equation}\label{eq16}
  \mathcal{A}^{\hat{u}}Y^{\gamma}(t,x,v)=0
\end{equation}
with boundary conditions $U(T,x,v)=x$ and $Y^{\gamma}(T,x,v)=\varphi^{\gamma}(x)$ for all $\gamma$, where
\begin{equation}\label{eq17}
H(t,x,v)=\int(\varphi^{\gamma})^{-1}\left(Y^{\gamma}(t,x,v)\right)\mathrm{d}\Gamma(\gamma)
\end{equation}
and
\begin{align}\label{eq18}
\hat{u}=\arg\sup\limits_{u} \;\left\{\mathcal{A}^{u}U(t,x,v)-\mathcal{A}^{u}H(t,x,v)
+\int\iota^{\gamma}(Y^{\gamma}(t,x,v))\mathcal{A}^{u}Y^{\gamma}(t,x,v)\mathrm{d}\Gamma(\gamma)\right\}.
\end{align}
Moreover, assume that $U$, $Y^{\gamma}$ and $H$ satisfy the condition (A) for all $\gamma$. Then $\hat{u}=(\hat{q},\hat{\pi})$ is an equilibrium
reinsurance and investment strategy and $V(t,x,v)=U(t,x,v)$ as well as $y^{\hat{u},\gamma}(t,x,v)=Y^{\gamma}(t,x,v)$.
\end{theorem}
\begin{proof}
See Appendix A.
\end{proof}

\begin{corollary}\label{cor1}
It follows from \eqref{eq17} that
\begin{eqnarray}\label{eq30}
\left\{
\begin{aligned}
&H_{t}(t,x,v)=\int\iota^{\gamma}(Y^{\gamma}(t,x,v))Y_{t}^{\gamma}(t,x,v)\mathrm{d}\Gamma(\gamma),\\
&H_{x}(t,x,v)=\int\iota^{\gamma}(Y^{\gamma}(t,x,v))Y_{x}^{\gamma}(t,x,v)\mathrm{d}\Gamma(\gamma),\\
&H_{v}(t,x,v)=\int\iota^{\gamma}(Y^{\gamma}(t,x,v))Y_{v}^{\gamma}(t,x,v)\mathrm{d}\Gamma(\gamma),\\
&H_{xx}(t,x,v)=\int\iota^{\gamma}(Y^{\gamma}(t,x,v))Y_{xx}^{\gamma}(t,x,v)\mathrm{d}\Gamma(\gamma)+
\int(\iota^{\gamma})^{'}(Y^{\gamma}(t,x,v))(Y_{x}^{\gamma}(t,x,v))^{2}\mathrm{d}\Gamma(\gamma),\\
&H_{vv}(t,x,v)=\int\iota^{\gamma}(Y^{\gamma}(t,x,v))Y_{vv}^{\gamma}(t,x,v)\mathrm{d}\Gamma(\gamma)+
\int(\iota^{\gamma})^{'}(Y^{\gamma}(t,x,v))(Y_{v}^{\gamma}(t,x,v))^{2}\mathrm{d}\Gamma(\gamma),\\
&H_{xv}(t,x,v)=\int\iota^{\gamma}(Y^{\gamma}(t,x,v))Y_{xv}^{\gamma}(t,x,v)\mathrm{d}\Gamma(\gamma)+
\int(\iota^{\gamma})^{'}(Y^{\gamma}(t,x,v))Y_{x}^{\gamma}(t,x,v)Y_{v}^{\gamma}(t,x,v)\mathrm{d}\Gamma(\gamma),
\end{aligned}
\right.
\end{eqnarray}
where $(\iota^{\gamma})^{'}$ denotes the derivative of $\iota^{\gamma}$, then the pseudo HJB \eqref{eq15} can be rewritten as
\begin{align}\label{eq31}
  \sup\limits_{u} \;&\left\{\mathcal{A}^{u}U(t,x,v)-0.5(b^{2}q^{2}+\pi^{2}v)\int(\iota^{\gamma})^{'}(Y^{\gamma}(t,x,v))(Y_{x}^{\gamma}(t,x,v))^{2}\mathrm{d}\Gamma(\gamma)-\rho\pi\sigma v\right.\nonumber\\
&\left.\times\int(\iota^{\gamma})^{'}(Y^{\gamma}(t,x,v))Y_{x}^{\gamma}(t,x,v)Y_{v}^{\gamma}(t,x,v)\mathrm{d}\Gamma(\gamma)
-0.5\sigma^{2}v\int(\iota^{\gamma})^{'}(Y^{\gamma}(t,x,v))(Y_{v}^{\gamma}(t,x,v))^{2}\mathrm{d}\Gamma(\gamma)\right\}=0.
\end{align}
\end{corollary}

\section{Equilibrium strategies to the optimization problem}
\setcounter{equation}{0}
In this section, our first step is to strive to solve the reinsurance and investment optimization problem with random risk aversion for a general utility function to the greatest extent possible. Then we will solve the problem for the well-known exponential utility function.

Since we have obtained the pseudo HJB \eqref{eq31} to capture the equilibrium reinsurance and investment strategy, it is possible for us to differentiate the expression inside the bracket of \eqref{eq31} in relation to $q$ and $\pi$, respectively, so as to obtain
\begin{equation}\label{eq32}
  \hat{q}(t)=-\frac{a \eta_{2}U_{x}(t,x,v)}{b^{2}\left[U_{xx}(t,x,v)-\int(\iota^{\gamma})^{'}(Y^{\gamma}(t,x,v))(Y_{x}^{\gamma}(t,x,v))^{2}\mathrm{d}\Gamma(\gamma)\right]}
\end{equation}
and
\begin{equation}\label{eq33}
\hat{\pi}(t)=-\frac{\xi U_{x}(t,x,v)+\rho\sigma\left[U_{xv}(t,x,v)-\int(\iota^{\gamma})^{'}(Y^{\gamma}(t,x,v))Y_{x}^{\gamma}(t,x,v)Y_{v}^{\gamma}(t,x,v)\mathrm{d}\Gamma(\gamma)\right]}
{U_{xx}(t,x,v)-\int(\iota^{\gamma})^{'}(Y^{\gamma}(t,x,v))(Y_{x}^{\gamma}(t,x,v))^{2}\mathrm{d}\Gamma(\gamma)},
\end{equation}
where we use $(\hat{q}(t),\hat{\pi}(t)):=(\hat{q}(t,x,v),\hat{\pi}(t,x,v))$. By the equation \eqref{eq23}, we know that
$$
U(t,x,v)=H(t,x,v)=\int(\varphi^{\gamma})^{-1}\left(Y^{\gamma}(t,x,v)\right)\mathrm{d}\Gamma(\gamma).
$$
By using the system of equations \eqref{eq30}, the candidate equilibrium reinsurance and investment strategy can be rewritten as
\begin{equation}\label{eq34}
  \hat{q}(t)=-\frac{a \eta_{2}\int\iota^{\gamma}(Y^{\gamma}(t,x,v))Y_{x}^{\gamma}(t,x,v)\mathrm{d}\Gamma(\gamma)}{b^{2}\int\iota^{\gamma}(Y^{\gamma}(t,x,v))Y_{xx}^{\gamma}(t,x,v)\mathrm{d}\Gamma(\gamma)}
\end{equation}
and
\begin{equation}\label{eq35}
\hat{\pi}(t)=-\frac{\xi \int\iota^{\gamma}(Y^{\gamma}(t,x,v))Y_{x}^{\gamma}(t,x,v)\mathrm{d}\Gamma(\gamma)+\rho\sigma\int\iota^{\gamma}(Y^{\gamma}(t,x,v))Y_{xv}^{\gamma}(t,x,v)\mathrm{d}\Gamma(\gamma)}
{\int\iota^{\gamma}(Y^{\gamma}(t,x,v))Y_{xx}^{\gamma}(t,x,v)\mathrm{d}\Gamma(\gamma)}.
\end{equation}

If the utility function $\varphi^{\gamma}$ takes the form of exponential utility
$$
\varphi^{\gamma}(\cdot)=-\frac{1}{\gamma}e^{-\gamma (\cdot)},
$$
where $\gamma>0$ a.s., then its inverse and the corresponding derivative can be given as follows
$$
(\varphi^{\gamma})^{-1}(\cdot)=-\frac{1}{\gamma}\ln(-\gamma (\cdot)),\quad
\frac{\mathrm{d}}{\mathrm{d} \nu}(\varphi^{\gamma})^{-1}(\nu)=-\frac{1}{\gamma \nu}.
$$
Thus, the candidate equilibrium reinsurance and investment strategy becomes
\begin{equation*}
   \left\{ \begin{aligned}
   &\hat{q}(t)=-\frac{a \eta_{2} \int \frac{Y_{x}^{\gamma}(t,x,v)}{\gamma Y^{\gamma}(t,x,v)}\mathrm{d}\Gamma(\gamma)}
{b^{2}\int \frac{Y_{xx}^{\gamma}(t,x,v)}{\gamma Y^{\gamma}(t,x,v)}\mathrm{d}\Gamma(\gamma)},\\
   &\hat{\pi}(t)=-\frac{\xi\int \frac{Y_{x}^{\gamma}(t,x,v)}{\gamma Y^{\gamma}(t,x,v)}\mathrm{d}\Gamma(\gamma)
+\rho\sigma\int \frac{Y_{xv}^{\gamma}(t,x,v)}{\gamma Y^{\gamma}(t,x,v)}\mathrm{d}\Gamma(\gamma)}
{\int \frac{Y_{xx}^{\gamma}(t,x,v)}{\gamma Y^{\gamma}(t,x,v)}\mathrm{d}\Gamma(\gamma)}.
  \end{aligned}\right.
\end{equation*}
We make an ansatz that $Y^{\gamma}(t,x,v)=-\frac{1}{\gamma}e^{g_{1}^{\gamma}(t)x+g_{2}^{\gamma}(t)v+g_{3}^{\gamma}(t)}$ with deterministic functions $g_{1}^{\gamma}(t)$, $g_{2}^{\gamma}(t)$ and $g_{3}^{\gamma}(t)$ of $t\in[0,T]$. We would like to point out that we adopt the top-script $\gamma$ on functions $g_{1}^{\gamma}(t)$, $g_{2}^{\gamma}(t)$ and $g_{3}^{\gamma}(t)$ to denote the dependence on $\gamma$ rather than taking the $\gamma$-th power of functions $g_{1}(t)$, $g_{2}(t)$ and $g_{3}(t)$.
Then, the partial derivatives of the ansatz read
\begin{eqnarray*}
\left\{
\begin{aligned}
   &Y^{\gamma}_{t}(t,x,v)=-\frac{1}{\gamma}e^{g_{1}^{\gamma}(t)x+g_{2}^{\gamma}(t)v+g_{3}^{\gamma}(t)}\left[\frac{\partial g_{1}^{\gamma}(t)}{\partial t}x+\frac{\partial g_{2}^{\gamma}(t)}{\partial t}v+\frac{\partial g_{3}^{\gamma}(t)}{\partial t}\right],\nonumber\\
   &Y^{\gamma}_{x}(t,x,v)=-\frac{1}{\gamma}g_{1}^{\gamma}(t)e^{g_{1}^{\gamma}(t)x+g_{2}^{\gamma}(t)v+g_{3}^{\gamma}(t)},
\quad Y^{\gamma}_{xx}(t,x,v)=-\frac{1}{\gamma}(g_{1}^{\gamma}(t))^{2}e^{g_{1}^{\gamma}(t)x+g_{2}^{\gamma}(t)v+g_{3}^{\gamma}(t)},\nonumber\\
   &Y^{\gamma}_{v}(t,x,v)=-\frac{1}{\gamma}g_{2}^{\gamma}(t)e^{g_{1}^{\gamma}(t)x+g_{2}^{\gamma}(t)v+g_{3}^{\gamma}(t)},
\quad Y^{\gamma}_{vv}(t,x,v)=-\frac{1}{\gamma}(g_{2}^{\gamma}(t))^{2}e^{g_{1}^{\gamma}(t)x+g_{2}^{\gamma}(t)v+g_{3}^{\gamma}(t)},\nonumber\\
&Y^{\gamma}_{xv}(t,x,v)=-\frac{1}{\gamma}g_{1}^{\gamma}(t)g_{2}^{\gamma}(t)e^{g_{1}^{\gamma}(t)x+g_{2}^{\gamma}(t)v+g_{3}^{\gamma}(t)}
\end{aligned}
\right.
\end{eqnarray*}
and the candidate equilibrium reinsurance and investment strategy follows
\begin{equation}\label{eq36}
  \hat{q}(t)=-\frac{a \eta_{2} \int \frac{g_{1}^{\gamma}(t)}{\gamma }\mathrm{d}\Gamma(\gamma)}
{b^{2}\int \frac{(g_{1}^{\gamma}(t))^{2}}{\gamma }\mathrm{d}\Gamma(\gamma)}
\end{equation}
and
\begin{equation}\label{eq37}
 \hat{\pi}(t)=-\frac{\xi\int \frac{g_{1}^{\gamma}(t)}{\gamma }\mathrm{d}\Gamma(\gamma)
+\rho\sigma\int \frac{g_{1}^{\gamma}(t)g_{2}^{\gamma}(t)}{\gamma }\mathrm{d}\Gamma(\gamma)}
{\int \frac{(g_{1}^{\gamma}(t))^{2}}{\gamma }\mathrm{d}\Gamma(\gamma)}.
\end{equation}
Furthermore, substituting the partial derivatives of $Y^{\gamma}(t,x,v)$ into \eqref{eq16} yields
\begin{align}\label{eq38}
  0=&e^{g_{1}^{\gamma}(t)x+g_{2}^{\gamma}(t)v+g_{3}^{\gamma}(t)}\left[\frac{\partial g_{1}^{\gamma}(t)}{\partial t}x+\frac{\partial g_{2}^{\gamma}(t)}{\partial t}v+\frac{\partial g_{3}^{\gamma}(t)}{\partial t}\right]+(rx+a\eta+a\eta_{2}\hat{q}+\xi v\hat{\pi})g_{1}^{\gamma}(t)e^{g_{1}^{\gamma}(t)x+g_{2}^{\gamma}(t)v+g_{3}^{\gamma}(t)}\nonumber\\
&+0.5(b^{2}\hat{q}^{2}+\hat{\pi}^{2}v)(g_{1}^{\gamma}(t))^{2}e^{g_{1}^{\gamma}(t)x+g_{2}^{\gamma}(t)v+g_{3}^{\gamma}(t)}
+\kappa(\theta-v)g_{2}^{\gamma}(t)e^{g_{1}^{\gamma}(t)x+g_{2}^{\gamma}(t)v+g_{3}^{\gamma}(t)}\nonumber\\
&+0.5\sigma^{2}v(g_{2}^{\gamma}(t))^{2}e^{g_{1}^{\gamma}(t)x+g_{2}^{\gamma}(t)v+g_{3}^{\gamma}(t)}+\rho\hat{\pi}\sigma vg_{1}^{\gamma}(t)g_{2}^{\gamma}(t)e^{g_{1}^{\gamma}(t)x+g_{2}^{\gamma}(t)v+g_{3}^{\gamma}(t)},
\end{align}
where $\hat{q}$ and $\hat{\pi}$ are given by \eqref{eq36} and \eqref{eq37}. Since $g_{1}^{\gamma}(t)$ and $g_{2}^{\gamma}(t)$ are assumed to be deterministic functions of $t$, it follows from \eqref{eq36} and \eqref{eq37} that $\hat{q}$ and $\hat{\pi}$ are independent of $x$ and $v$. Therefore, one can compare the coefficients in front of $x$ and $v$ in \eqref{eq38} to obtain
\begin{align}
 & \frac{\partial g_{1}^{\gamma}(t)}{\partial t}+r g_{1}^{\gamma}(t)=0, \label{eq39}\\
 & \frac{\partial g_{2}^{\gamma}(t)}{\partial t}+\xi\hat{\pi}g_{1}^{\gamma}(t)+0.5\hat{\pi}^{2}(g_{1}^{\gamma}(t))^{2}-\kappa g_{2}^{\gamma}(t)
+0.5\sigma^{2}(g_{2}^{\gamma}(t))^{2}+\rho\hat{\pi}\sigma g_{1}^{\gamma}(t)g_{2}^{\gamma}(t)=0, \label{eq40} \\
&\frac{\partial g_{3}^{\gamma}(t)}{\partial t}+(a\eta+a\eta_{2}\hat{q})g_{1}^{\gamma}(t)+0.5b^{2}\hat{q}^{2}(g_{1}^{\gamma}(t))^{2}+\kappa\theta g_{2}^{\gamma}(t)=0. \label{eq41}
\end{align}
By the boundary condition $g_{1}^{\gamma}(T)=-\gamma$, it can be inferred from \eqref{eq39} that $g_{1}^{\gamma}(t)=-\gamma e^{r(T-t)}$. Moreover, we are able to deduce the candidate equilibrium reinsurance and investment strategy as shown below
\begin{equation}\label{eq42}
\hat{q}(t)=\frac{a \eta_{2}}{b^{2}\mathbb{E}[\gamma]}e^{-r(T-t)},\quad
   \hat{\pi}(t)=\frac{\xi+\rho\sigma \int g_{2}^{\gamma}(t)\mathrm{d}\Gamma(\gamma) }{\mathbb{E}[\gamma]}e^{-r(T-t)}.
\end{equation}
Substituting \eqref{eq42} into \eqref{eq40}, we know that the ordinary differential equation (ODE) \eqref{eq40} is infinite-dimensional due to the distribution of $\gamma$, which is generally difficult to solve. In the sequel, we will carry out our investigation for special scenarios of an $n$-point distribution and a one-point distribution with respect to $\gamma$, respectively.

\subsection{Special case with $n$ possible risk aversions}
If the random risk aversion is assumed to follow an $n$-point distribution with possible outcomes $\gamma_{i}>0\;(i=1,\cdot\cdot\cdot,n)$ a.s. that are realized with the probabilities $p_{i}=p(\gamma_{i})$, where $\sum\nolimits_{i=1}^{n}p_{i}=1$, then we can obtain the following candidate equilibrium reinsurance and investment strategy
\begin{equation}\label{eq43}
\hat{q}(t)=\frac{a \eta_{2}}{b^{2}\sum\limits_{i=1}^{n}\gamma_{i} p_{i}}e^{-r(T-t)},\quad
\hat{\pi}(t)=\frac{\xi+\rho\sigma \sum\limits_{i=1}^{n}g_{2}^{\gamma_{i}}(t) p_{i}}{\sum\limits_{i=1}^{n}\gamma_{i} p_{i}}e^{-r(T-t)}.
\end{equation}
Moreover, for $i=1,\cdot\cdot\cdot,n$, we have
\begin{equation}\label{eq44}
  -\frac{\partial g_{2}^{\gamma_{i}}(t)}{\partial t}=\xi\hat{\pi}g_{1}^{\gamma_{i}}(t)+0.5\hat{\pi}^{2}(g_{1}^{\gamma_{i}}(t))^{2}-\kappa g_{2}^{\gamma_{i}}(t)
+0.5\sigma^{2}(g_{2}^{\gamma_{i}}(t))^{2}+\rho\hat{\pi}\sigma g_{1}^{\gamma_{i}}(t)g_{2}^{\gamma_{i}}(t)
\end{equation}
and
\begin{equation}\label{eq45}
  \frac{\partial g_{3}^{\gamma_{i}}(t)}{\partial t}+(a\eta+a\eta_{2}\hat{q})g_{1}^{\gamma_{i}}(t)+0.5b^{2}\hat{q}^{2}(g_{1}^{\gamma_{i}}(t))^{2}+\kappa\theta g_{2}^{\gamma_{i}}(t)=0
\end{equation}
with $g_{2}^{\gamma_{i}}(T)=g_{3}^{\gamma_{i}}(T)=0$, where $\hat{q}$ and $\hat{\pi}$ are given by \eqref{eq43} and $g_{1}^{\gamma_{i}}(t)=-\gamma_{i} e^{r(T-t)}$. Since $\hat{\pi}$ contains the term $\sum\nolimits_{i=1}^{n}g_{2}^{\gamma_{i}}(t) p_{i}$, the equation \eqref{eq44} constitutes a system consisting of $n$ coupled equations for $g_{2}^{\gamma_{i}}(t)$ with $i=1,\cdot\cdot\cdot,n$. Therefore, exploring whether there exists a solution to the general infinite-dimensional ODE \eqref{eq40} is reduced to the question of the solvability of the coupled equations. It is not hard to see that the right-hand side of \eqref{eq44} is continuous on its domain. Thus, relying on the existence theorem of Peano, there exists no less than one local solution for \eqref{eq44}. Then, the existence of solutions to \eqref{eq45} can be ensured. We note that due to $\gamma_{i}>0$ and $p_{i}>0$, $\sum\nolimits_{i=1}^{n}\gamma_{i} p_{i}$ cannot reach zero and thus the solutions of \eqref{eq44} and \eqref{eq45} cannot explode.

\begin{remark}
From the equations \eqref{eq43} and \eqref{eq44},
it can be noticed that the candidate equilibrium reinsurance strategy is not reliant on the parameters of the risky asset, and the parameters of the insurance market do not affect the candidate equilibrium investment strategy. The reason for this is that the insurance market and the financial market are assumed to be not correlated. Moreover, it can be found that the candidate equilibrium reinsurance and investment strategy is deterministic and state-independent.
\end{remark}

Under some mild conditions, we will prove that for the candidate equilibrium reinsurance and investment strategy \eqref{eq43}, the
admissibility prerequisites and relevant suppositions in Theorem \ref{theorem3.1} are satisfied, which implies that \eqref{eq43} is indeed an equilibrium reinsurance and investment strategy.

\begin{proposition}\label{proposition4.1}
For the case of $n$ possible risk aversions, if there exists a non-positive solution $g_{2}^{\gamma_{i}}(t)$ to the equation \eqref{eq44} and the parameters satisfy $-8\gamma_{i}\xi\bar{\pi}(t)+32\gamma_{i}^{2}\bar{\pi}^{2}(t)\leq\frac{\kappa^{2}}{2\sigma^{2}}$, where $\bar{\pi}$ is given by the equation \eqref{eq58}, then \eqref{eq43} is an equilibrium reinsurance and investment strategy.
\end{proposition}
\begin{proof}
See Appendix B.
\end{proof}

\begin{remark}
Proposition \ref{proposition4.1} shows that the equilibrium reinsurance and investment strategy and the relevant value function are well-defined under some mild conditions. Although it may seem difficult to directly verify the conditions, we will verify these conditions by numerical experiments in Section 5.
\end{remark}

\begin{remark}
The equilibrium reinsurance strategy given by \eqref{eq43} shows that the insurer is prone to purchasing reinsurance if the first moment of random risk aversion $\gamma$ is higher. Moreover, the ratio $\frac{a\eta_{2}}{b^{2}}$ exerts an influence on the insurer's prospective profit when it functions as a reinsurer. Furthermore, the time horizon is another crucial factor that influences the insurer's equilibrium reinsurance strategy. Denote $\mathbb{E}[\gamma]=\sum\limits_{i=1}^{n}\gamma_{i} p_{i}$. If $\frac{a \eta_{2}}{b^{2}\mathbb{E}[\gamma]}\geq1$, when $T\in\left(\frac{1}{r}\ln\left(\frac{a \eta_{2}}{b^{2}\mathbb{E}[\gamma]}\right),\infty\right)$, then the equilibrium reinsurance strategy is acquiring new business only on the condition that $T-t<\frac{1}{r}\ln\left(\frac{a \eta_{2}}{b^{2}\mathbb{E}[\gamma]}\right)$; when $T\in\left(0,\frac{1}{r}\ln\left(\frac{a \eta_{2}}{b^{2}\mathbb{E}[\gamma]}\right)\right)$, then the equilibrium reinsurance strategy is acquiring new business throughout the whole period of insurance-investment. On the other hand, if $\frac{a \eta_{2}}{b^{2}\mathbb{E}[\gamma]}<1$, the hazard of obtaining new business would be extremely high, making it difficult for the insurer to accept. As a result, the insurer would purchase reinsurance to transfer risk over the entire insurance-investment horizon.
\end{remark}

\subsection{Special case with one possible risk aversion}
If the distribution of the risk aversion degenerates into one outcome $\gamma$, then it follows from \eqref{eq40} and \eqref{eq42} that
the candidate equilibrium reinsurance and investment strategy can be given by
\begin{equation}\label{eq46}
   \left\{ \begin{aligned}
   &\hat{q}(t)=\frac{a \eta_{2}}{\gamma b^{2}}e^{-r(T-t)},\\
   &\hat{\pi}(t)=\frac{\xi+\rho\sigma  g_{2}^{\gamma}(t) }{\gamma}e^{-r(T-t)},
  \end{aligned}\right.
\end{equation}
where $g_{2}^{\gamma}(t)$ is captured by
\begin{equation}\label{eq47}
  \frac{\partial g_{2}^{\gamma}(t)}{\partial t}-0.5\xi^{2}-(\kappa+\rho\sigma\xi)g_{2}^{\gamma}(t)+0.5\sigma^{2}(1-\rho^{2})(g_{2}^{\gamma}(t))^{2}=0.
\end{equation}
Let $\tau=T-t$. Then \eqref{eq47} can be rewritten as
\begin{equation*}
  \frac{\partial g_{2}^{\gamma}(\tau)}{\partial \tau}=0.5k_{1}-k_{2}g_{2}^{\gamma}(\tau)+0.5k_{3}(g_{2}^{\gamma}(\tau))^{2},
\end{equation*}
where $k_{1}=-\xi^{2}$, $k_{2}=\kappa+\rho\sigma\xi$ and $k_{3}=\sigma^{2}(1-\rho^{2})$. By Lemma 1.1 of \cite{Bergen2018}, we can obtain that $g_{2}^{\gamma}(t)$ satisfies
\begin{equation}\label{eq48}
 g_{2}^{\gamma}(t)=\frac{k_{1}\left(e^{k_{4}(T-t)}-1\right)}{2k_{4}+(k_{2}+k_{4})\left(e^{k_{4}(T-t)}-1\right)}
\end{equation}
with $k_{4}=\sqrt{k_{2}^{2}-k_{1}k_{3}}$. Moreover, we have
$$
g_{3}^{\gamma}(t)=-0.5\frac{a^{2}\eta_{2}^{2}}{b^{2}}(T-t)+\frac{a\eta\gamma}{r}\left(1-e^{r(T-t)}\right)
+\frac{2\kappa\theta}{k_{3}}\ln\left[\frac{2k_{4}e^{\frac{1}{2}(k_{2}+k_{4})(T-t)}}{2k_{4}+(k_{2}+k_{4})\left(e^{k_{4}(T-t)}-1\right)}\right].
$$

Note that $g_{2}^{\gamma}(t)<0$ in the equation \eqref{eq48}. Since $g_{1}^{\gamma}(t)$, $g_{2}^{\gamma}(t)$ and $g_{3}^{\gamma}(t)$ have unique global solutions, one can follow completely along the lines of $n$ possible risk aversions to verify the admissibility prerequisites and the suppositions in Theorem \ref{theorem3.1}, but in a more direct way. Thus, the following results are available.

\begin{proposition}\label{proposition4.2}
For the case of one possible risk aversion, if the parameters satisfy $-8\gamma_{i}\xi\bar{\pi}(t)+32\gamma_{i}^{2}\bar{\pi}^{2}(t)\leq\frac{\kappa^{2}}{2\sigma^{2}}$, where $\bar{\pi}(t)=\frac{\xi+\rho\sigma  g_{2}^{\gamma}(t) }{\gamma}$ and $g_{2}^{\gamma}(t)$ is given by the equation \eqref{eq48}, then \eqref{eq46} is an equilibrium reinsurance and investment strategy.
\end{proposition}

\begin{remark}
When there is just one possible risk aversion, the equilibrium reinsurance and investment strategy \eqref{eq46} is identical to the optimal reinsurance and investment strategy (40) with $\beta=0$ in \cite{Yi2013}. This confirms that for the situation where there is just one possible risk aversion, the equilibrium reinsurance and investment strategy solving the problem \eqref{eq8} coincides with the optimal reinsurance and investment strategy solving the problem \eqref{eq7}.
\end{remark}

\section{Numerical illustrations}
\setcounter{equation}{0}
This section is devoted to carrying out some numerical experiments for demonstrating the impacts of model parameters upon the equilibrium reinsurance and investment strategy \eqref{eq43}. During numerical illustrations, except when otherwise specified, the fundamental insurance and financial market parameters are obtained from \cite{Bi2019} and \cite{Kraft2013}, respectively, which are given in Table \ref{tab1}. Moreover, following \cite{Desmettre2023}, we consider the two point-distribution for risk aversion, where the parameters are specified within two particular cases: (I) $\gamma_{1}=0.5$, $\gamma_{2}=4$ and $p_{1}=p_{2}=0.5$; (II) $\gamma_{1}=0.5$, $\gamma_{2}=4$, $p_{1}=0.8$ and $p_{2}=0.2$. In addition, we set $T=10(years)$ and $T=100(years)$.

\begin{table}[H]
 \caption{Values of insurance and financial market parameters.}
 \label{tab1}%\label{tab:tb1}
 \small
  \centering
   \begin{tabular}{>{\hfil}p{1cm}<{\hfil} >{\hfil}p{1.1cm}<{\hfil} >{\hfil}p{1cm}<{\hfil} >{\hfil}p{1cm}<{\hfil} >{\hfil}p{1.2cm}<{\hfil}
   >{\hfil}p{1cm}<{\hfil} >{\hfil}p{1cm}<{\hfil} >{\hfil}p{1.1cm}<{\hfil} >{\hfil}p{1.1cm}<{\hfil} >{\hfil}p{1cm}<{\hfil}>{\hfil}p{0.6cm}<{\hfil}  }
     \hline
{$\eta_{1}$}& {$\eta_{2}$}& 	{$\lambda_{1}$} & {$\mu_{1}$} & {$\mu_{2}$} & {$r$} & {$\xi$}&{$\kappa$} &{$\theta$}&{$\sigma$} & {$\rho$}   \\\hline
0.3 & 0.5 &   1& 0.1 &0.2 & 0.05 & $\frac{7}{15}$ & 5 &$(0.15)^{2}$& 0.25 & -0.5 \\
     \hline
   \end{tabular}
\end{table}

\subsection{Impacts of model parameters on the equilibrium investment strategy}
The equation \eqref{eq43} shows that the equilibrium investment strategy depends on $g_{2}^{\gamma_{i}}$ with $i=1,2$. It can be derived from the equation \eqref{eq44} that $g_{2}^{\gamma_{i}}\;(i=1,2)$ are captured by two coupled ODEs, whose coefficients depend on time $t$. This results in that there is no analytical solution to the equation \eqref{eq44} in general. On the other hand, in algorithms for numerically solving the forward ODEs, the Predictor-Corrector or Modified-Euler method \cite{Duffy2006} not only improves the accuracy of the solution, but also enhances the stability of the numerical solution. Thus, we attempt to adopt the Predictor-Corrector or Modified-Euler method to numerically calculate the equilibrium investment strategy. To begin with, we need to convert the backward ODE \eqref{eq44} into an equivalent forward ODE. To be specific, we let $h_{2}^{\gamma_{i}}(T-t)=g_{2}^{\gamma_{i}}(t)$ and rewrite the equation \eqref{eq44} as follows
\begin{align*}
  -\frac{\partial g_{2}^{\gamma_{i}}(T-s)}{\partial (T-s)}=&\xi\hat{\pi}(T-s)g_{1}^{\gamma_{i}}(T-s)+0.5(\hat{\pi}(T-s))^{2}(g_{1}^{\gamma_{i}}(T-s))^{2}-\kappa g_{2}^{\gamma_{i}}(T-s)\\
&+0.5\sigma^{2}(g_{2}^{\gamma_{i}}(T-s))^{2}+\rho\hat{\pi}(T-s)\sigma g_{1}^{\gamma_{i}}(T-s)g_{2}^{\gamma_{i}}(T-s),
\end{align*}
which is equivalent to the following forward ODE
\begin{align}\label{eq64}
  \frac{\partial h_{2}^{\gamma_{i}}(s)}{\partial s}=&\xi\hat{\pi}(T-s)g_{1}^{\gamma_{i}}(T-s)+0.5(\hat{\pi}(T-s))^{2}(g_{1}^{\gamma_{i}}(T-s))^{2}-\kappa h_{2}^{\gamma_{i}}(s)\nonumber\\
&+0.5\sigma^{2}(h_{2}^{\gamma_{i}}(s))^{2}+\rho\hat{\pi}(T-s)\sigma g_{1}^{\gamma_{i}}(T-s)h_{2}^{\gamma_{i}}(s)
\end{align}
with $ h_{2}^{\gamma_{i}}(0)=g_{2}^{\gamma_{i}}(T)=0$ and $s\in(0,T]$, where
$$\hat{\pi}(T-s)=\frac{\xi+\rho\sigma \sum\limits_{i=1}^{2}g_{2}^{\gamma_{i}}(T-s) p_{i}}{\sum\limits_{i=1}^{2}\gamma_{i} p_{i}}e^{-r(T-(T-s))}=\frac{\xi+\rho\sigma \sum\limits_{i=1}^{2}h_{2}^{\gamma_{i}}(s) p_{i}}{\sum\limits_{i=1}^{2}\gamma_{i} p_{i}}e^{-rs}.$$

Next, we divide the time interval $[0,T]$ into $M$ parts and denote by $l=\frac{T}{M}$. Let $t_{m}=ml$ and
\begin{align*}
F_{i}(t_{m},g_{1}^{\gamma_{i}}(T-t_{m}),h_{2}^{\gamma_{1}}(t_{m}),h_{2}^{\gamma_{2}}(t_{m}))
=&\frac{\xi^{2}+\xi\rho\sigma\sum\limits_{i=1}^{2}h_{2}^{\gamma_{i}}(t_{m}) p_{i}}{\sum\limits_{i=1}^{2}\gamma_{i} p_{i}}e^{-rt_{m}}g_{1}^{\gamma_{i}}(T-t_{m})-\kappa h_{2}^{\gamma_{i}}(t_{m})\\
&+0.5[\frac{\xi+\rho\sigma \sum\limits_{i=1}^{2}h_{2}^{\gamma_{i}}(t_{m}) p_{i}}{\sum\limits_{i=1}^{2}\gamma_{i} p_{i}}e^{-rt_{m}}]^{2}(g_{1}^{\gamma_{i}}(T-t_{m}))^{2}+0.5\sigma^{2}(h_{2}^{\gamma_{i}}(t_{m}))^{2}\\
&+\rho\frac{\xi+\rho\sigma \sum\limits_{i=1}^{2}h_{2}^{\gamma_{i}}(t_{m}) p_{i}}{\sum\limits_{i=1}^{2}\gamma_{i} p_{i}}e^{-rt_{m}}\sigma g_{1}^{\gamma_{i}}(T-t_{m})h_{2}^{\gamma_{i}}(t_{m}),
\end{align*}
where $m=0,\cdot\cdot\cdot,M$ and $i=1,2$. Then, we give Algorithm 1 to explore the equilibrium investment strategy.

\begin{algorithm}
\caption{Exploring equilibrium investment strategy.}
\label{alg::conjugateGradient}
  \begin{algorithmic}[1]
    \Require
      the expressions of $g_{1}^{\gamma_{1}}$, $g_{1}^{\gamma_{2}}$, $\hat{\pi}$, $\bar{\pi}$, $F_{1}$ and $F_{2}$;
      the values of parameters $T$, $M$, $r$, $\xi$, $\kappa$, $\sigma$, $\rho$, $\gamma_{1}$, $\gamma_{2}$, $p_{1}$ and $p_{2}$;
    \Ensure
      the value of the equilibrium investment strategy $\hat{\pi}$;
    \State initial $h_{2}^{\gamma_{1}}(0)=h_{2}^{\gamma_{2}}(0)=0$;
     \For{$m=0,\;m< M$}
      \State compute $\bar{h}_{2}^{\gamma_{1}}(t_{m+1})=h_{2}^{\gamma_{1}}(t_{m})
      +lF_{1}(t_{m},g_{1}^{\gamma_{1}}(T-t_{m}),h_{2}^{\gamma_{1}}(t_{m}),h_{2}^{\gamma_{2}}(t_{m}))$;
      \State compute $\bar{h}_{2}^{\gamma_{2}}(t_{m+1})=h_{2}^{\gamma_{2}}(t_{m})
      +lF_{2}(t_{m},g_{1}^{\gamma_{2}}(T-t_{m}),h_{2}^{\gamma_{1}}(t_{m}),h_{2}^{\gamma_{2}}(t_{m}))$;
      \State compute
      \begin{align*}
      h_{2}^{\gamma_{1}}(t_{m+1})=&h_{2}^{\gamma_{1}}(t_{m})
      +\frac{l}{2}\left[F_{1}(t_{m},g_{1}^{\gamma_{1}}(T-t_{m}),h_{2}^{\gamma_{1}}(t_{m}),h_{2}^{\gamma_{2}}(t_{m}))\right.\\
      &\left.+
      F_{1}(t_{m+1},g_{1}^{\gamma_{1}}(T-t_{m+1}),\bar{h}_{2}^{\gamma_{1}}(t_{m+1}),\bar{h}_{2}^{\gamma_{2}}(t_{m+1}))\right];
      \end{align*}
       \State compute
      \begin{align*}
      h_{2}^{\gamma_{2}}(t_{m+1})=&h_{2}^{\gamma_{2}}(t_{m})
      +\frac{l}{2}\left[F_{2}(t_{m},g_{1}^{\gamma_{2}}(T-t_{m}),h_{2}^{\gamma_{1}}(t_{m}),h_{2}^{\gamma_{2}}(t_{m}))\right.\\
      &\left.+
      F_{2}(t_{m+1},g_{1}^{\gamma_{2}}(T-t_{m+1}),\bar{h}_{2}^{\gamma_{1}}(t_{m+1}),\bar{h}_{2}^{\gamma_{2}}(t_{m+1}))\right];
      \end{align*}
      \State compute $-8\gamma_{_{i}}\xi\bar{\pi}(T-t_{m+1})+32\gamma_{_{i}}^{^{2}}\bar{\pi}^{2}(T-t_{m+1})$;
    \While {$(h_{2}^{\gamma_{i}}(t_{m+1})\leq0)\;and\;(-8\gamma_{_{i}}\xi\bar{\pi}(T-t_{m+1})+32\gamma_{_{i}}^{^{2}}\bar{\pi}^{2}(T-t_{m+1})\leq\frac{\kappa^{2}}{2\sigma^{^{2}}})\; for\;i=1,2$}
      \State compute $\hat{\pi}(T-t_{m+1})=\frac{\xi+\rho\sigma \sum\limits_{i=1}^{2}h_{2}^{\gamma_{i}}(t_{m+1}) p_{i}}{\sum\limits_{i=1}^{2}\gamma_{i} p_{i}}e^{-rt_{m+1}}$;
      \State $m=m+1$;
    \EndWhile
    \EndFor
    \State reverse $\hat{\pi}(T-t_{m+1})$ with $m=0,\cdot\cdot\cdot,M-1$.
  \end{algorithmic}
\end{algorithm}

Figure \ref{fig1} shows the impacts of the risk-free interest rate $r$ on the equilibrium investment strategy. Looking at Figure \ref{fig1}, we observe that the equilibrium investment strategy decreases with $r$. In other words, a higher risk-free interest rate leads to a lower equilibrium investment strategy. This accords with our common sense. When the risk-free interest rate becomes higher, the insurer is more inclined to invest profits in the risk-free asset, thereby leading to a reduction in the investment directed towards the risky asset. Moreover, we can see that when the risk-free interest rate remains fixed, an increase in the expected value of risk aversion results in a smaller investment by the insurer in the risky asset. In addition, another phenomenon is reflected in Figure \ref{fig1}, that is, with the given risk-free interest rate parameter, the equilibrium investment strategy gets smaller when time $t$ grows. Notably, as time $t$ approaches $100$, the equilibrium investment strategy is inclined towards 0. This substantiates the theory of decreasing life-cycle investment profiles in the context of the exponential utility, which is consistent with the case of the power utility in \cite{Desmettre2023}. It is fascinating to note that a similar phenomenon occurs in the subsequent Figure \ref{fig2}.

What is depicted in Figure \ref{fig2} is the impacts of the premium related to volatility $\xi$ on the equilibrium investment strategy. From Figure \ref{fig2}, we notice that the equilibrium amount invested in the risky asset goes up when $\xi$ is on the rise. One plausible interpretation for such a situation lies in the fact that with a larger $\xi$, the risky asset has a greater appreciation rate. This benefits the insurer in attaining more expected returns, thereby inducing she/he to invest a greater amount of money in the risky asset. Moreover, Figure \ref{fig2} demonstrates that the insurer whose expected value of risk aversion is higher tends to cut down on the investment in the risky asset.

\begin{figure}
    \centering
    \begin{minipage}{5cm}
      \includegraphics[width=5cm]{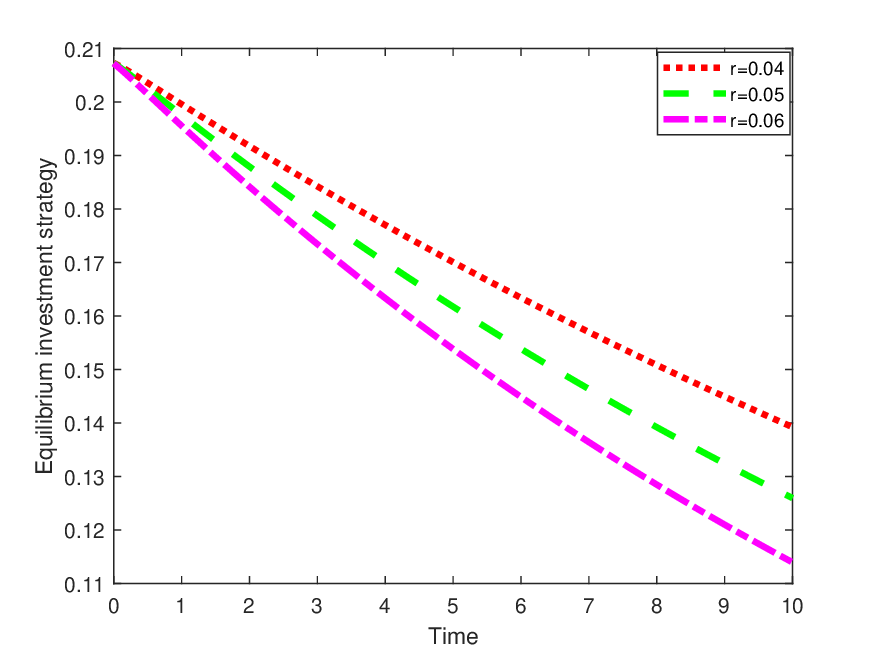}
      (a) Case (I) when $T=10$
    \end{minipage}
     \begin{minipage}{5cm}
      \includegraphics[width=5cm]{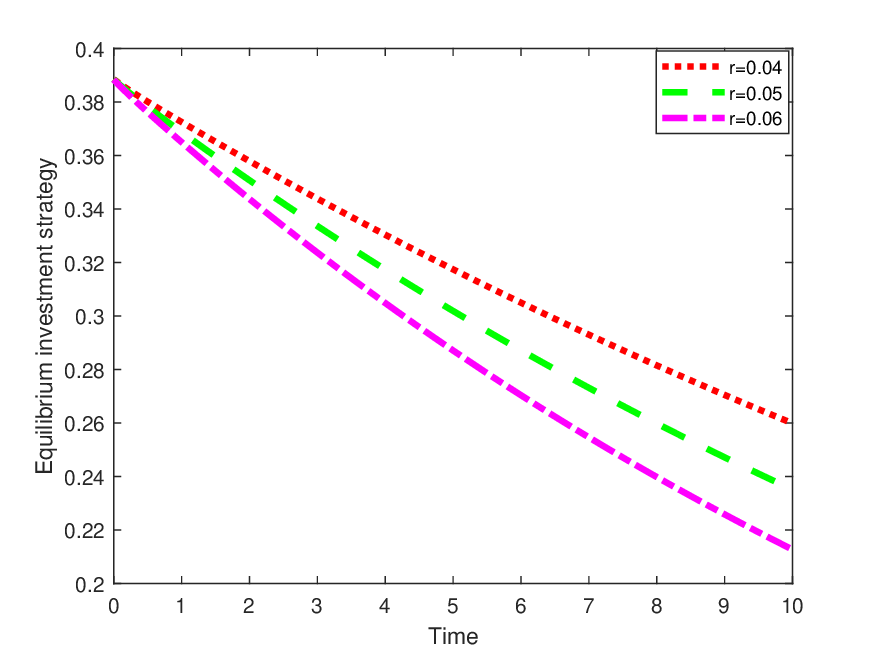}
      (b) Case (II) when $T=10$
    \end{minipage}

    \begin{minipage}{5cm}
      \includegraphics[width=5cm]{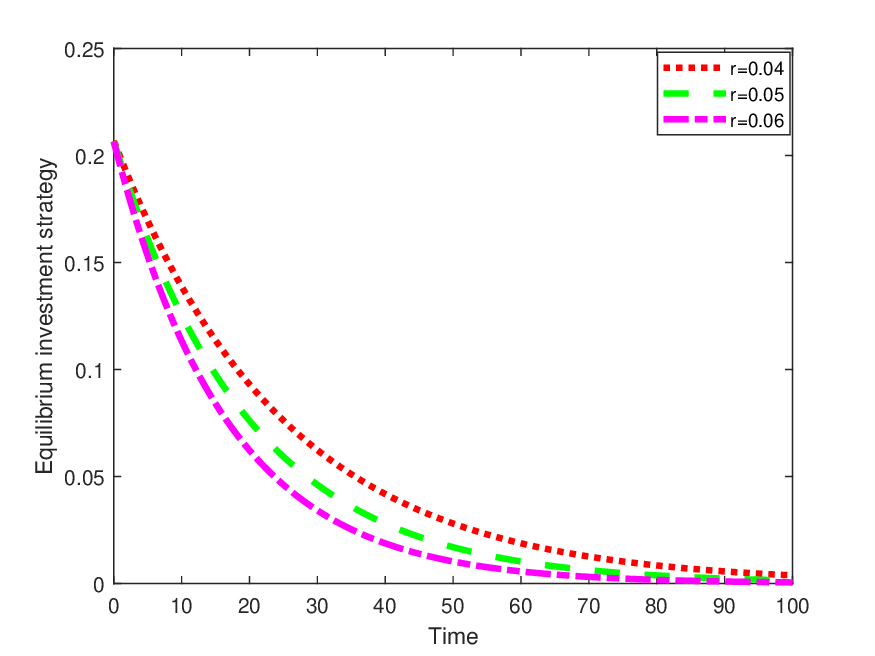}
      (c) Case (I) when $T=100$
    \end{minipage}
    \begin{minipage}{5cm}
      \includegraphics[width=5cm]{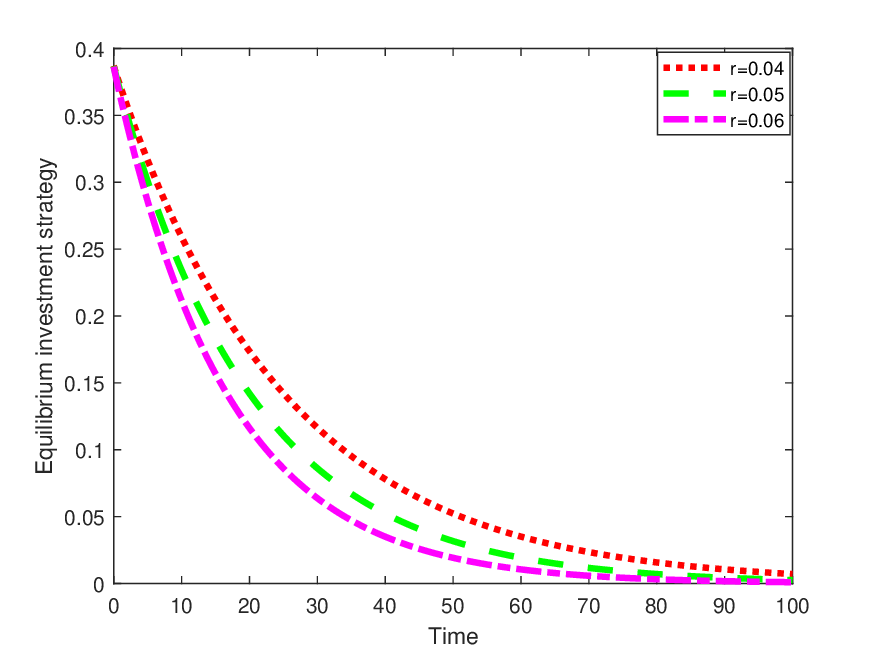}
      \caption*{(d) Case (II) when $T=100$}
    \end{minipage}
    \caption{The impacts of $r$ on the equilibrium investment strategy.}
\label{fig1}
\end{figure}

\begin{figure}
    \centering
    \begin{minipage}{5cm}
      \includegraphics[width=5cm]{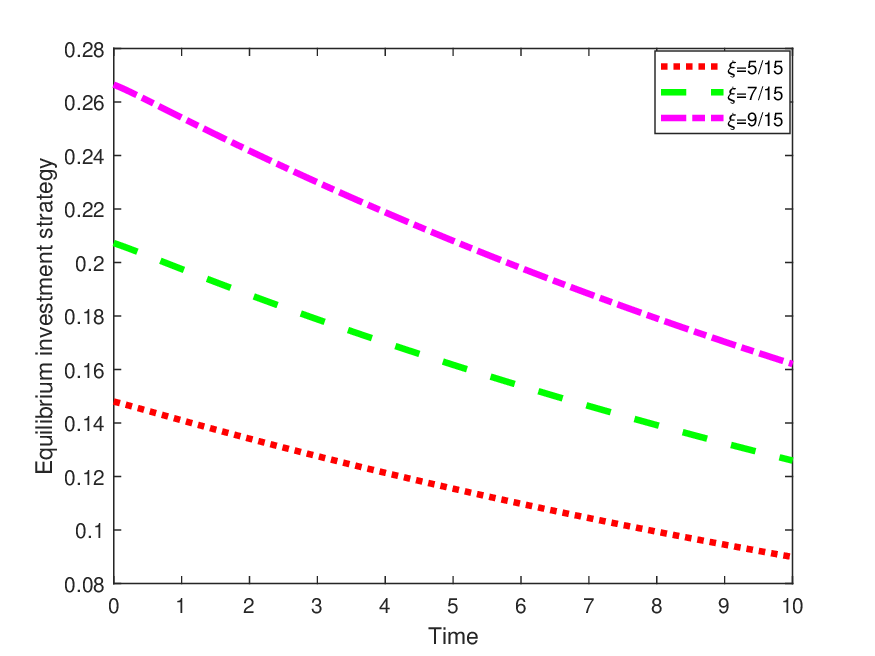}
      (a) Case (I) when $T=10$
    \end{minipage}
     \begin{minipage}{5cm}
      \includegraphics[width=5cm]{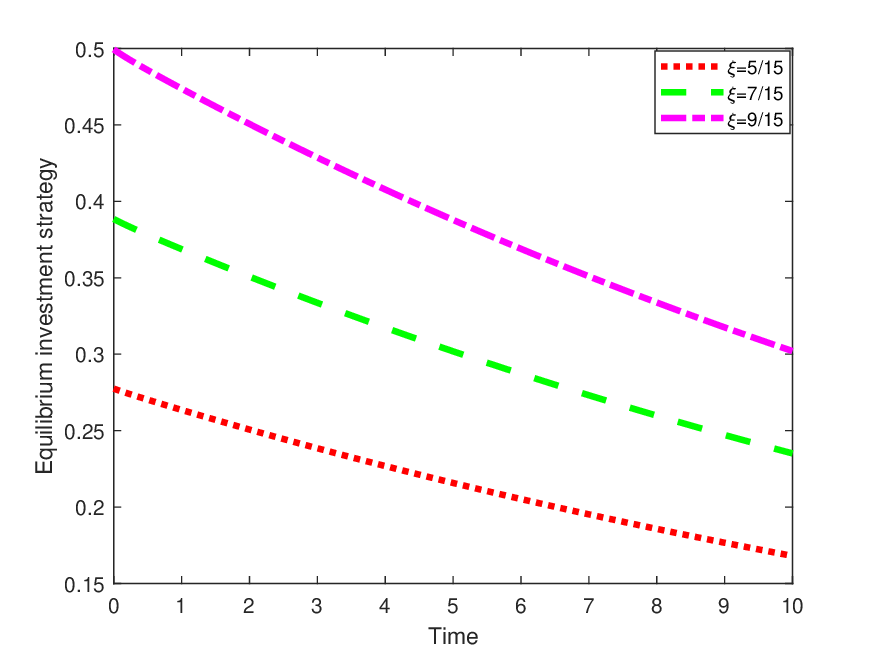}
      (b) Case (II) when $T=10$
    \end{minipage}

    \begin{minipage}{5cm}
      \includegraphics[width=5cm]{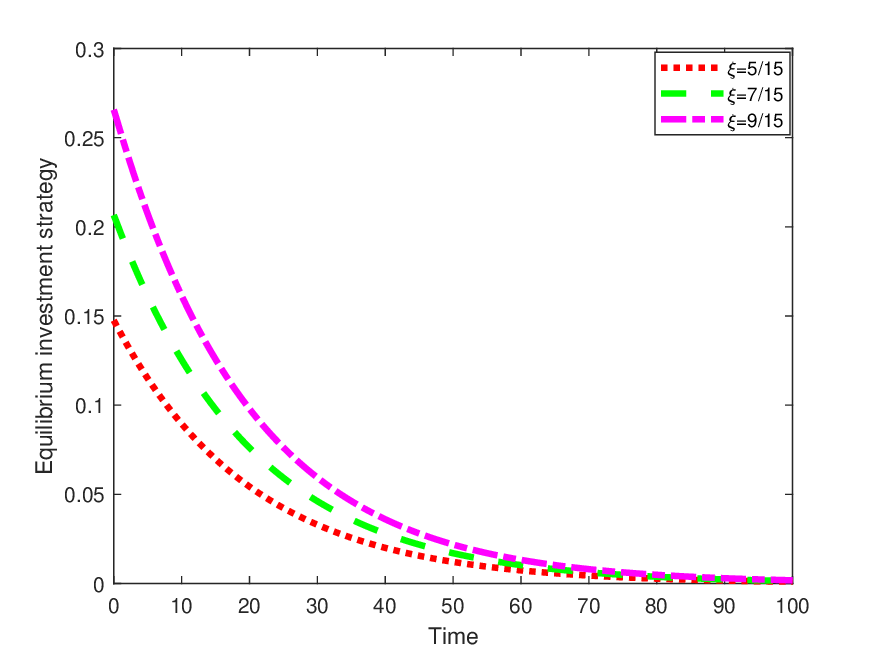}
      (c) Case (I) when $T=100$
    \end{minipage}
    \begin{minipage}{5cm}
      \includegraphics[width=5cm]{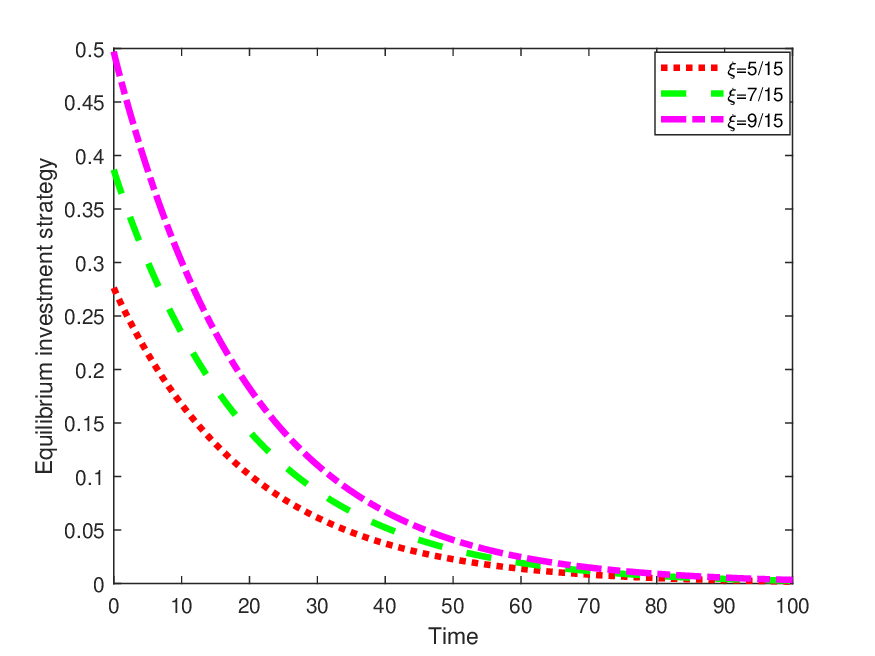}
      \caption*{(d) Case (II) when $T=100$}
    \end{minipage}
    \caption{The impacts of $\xi$ on the equilibrium investment strategy.}
\label{fig2}
\end{figure}

%\begin{figure}
%    \centering
%    \begin{minipage}{5cm}
%      \includegraphics[width=5cm]{a31}
%     (a) Case (I) when $T=10$
%   \end{minipage}
%     \begin{minipage}{5cm}
%      \includegraphics[width=5cm]{a32}
%      (b) Case (II) when $T=10$
%    \end{minipage}
%
%    \begin{minipage}{5cm}
%      \includegraphics[width=5cm]{a33}
%     (c) Case (I) when $T=100$
%   \end{minipage}
%    \begin{minipage}{5cm}
%      \includegraphics[width=5cm]{a34}
%      \caption*{(d) Case (II) when $T=100$}
%    \end{minipage}
%    \caption{The impacts of $\kappa$ on the equilibrium investment strategy when $\rho=-0.5$.}
%\label{fig3}
%\end{figure}

\begin{figure}
    \centering
    \begin{minipage}{5cm}
      \includegraphics[width=5cm]{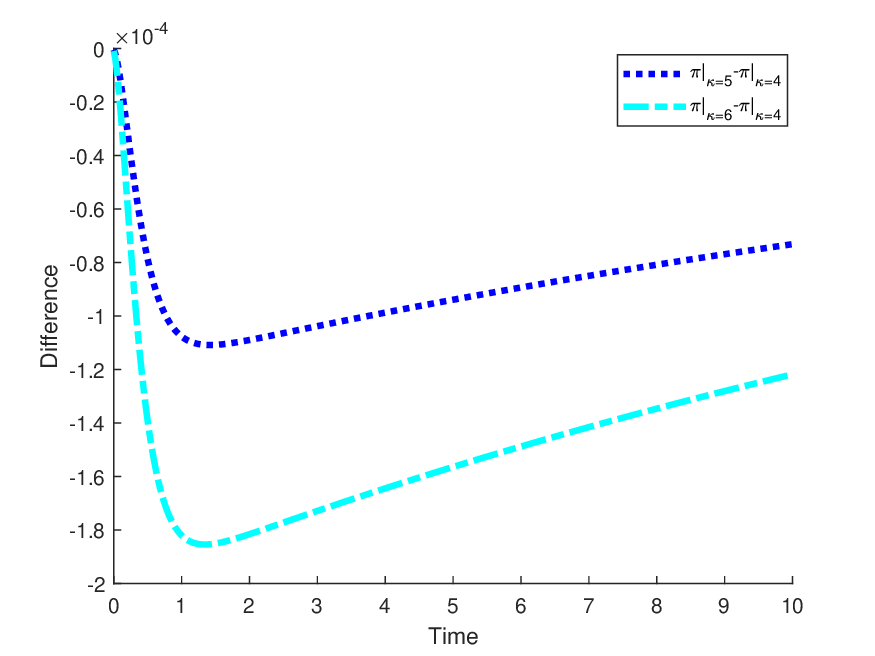}
      (a) Case (I) when $T=10$
    \end{minipage}
     \begin{minipage}{5cm}
      \includegraphics[width=5cm]{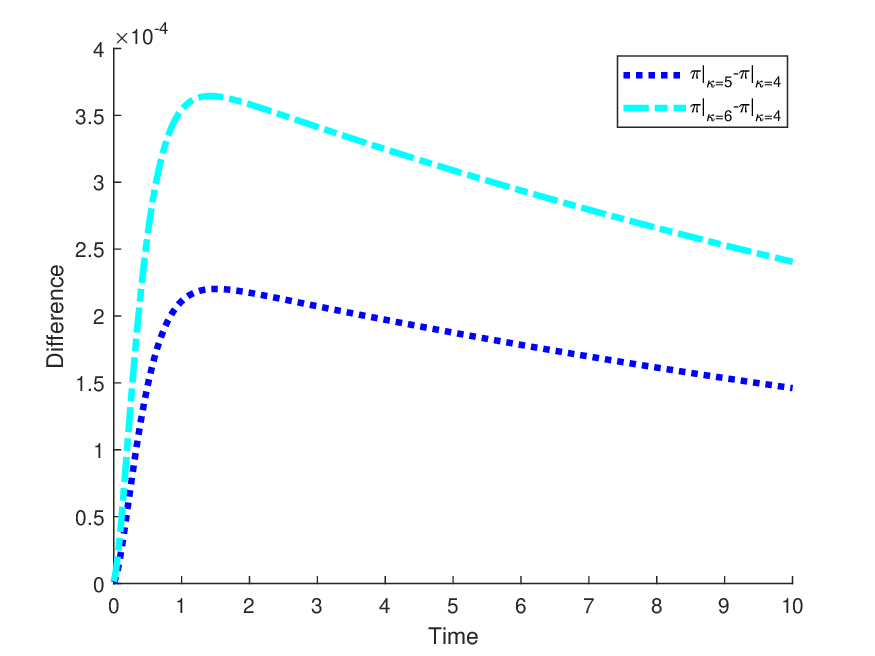}
      (b) Case (II) when $T=10$
    \end{minipage}

    \begin{minipage}{5cm}
      \includegraphics[width=5cm]{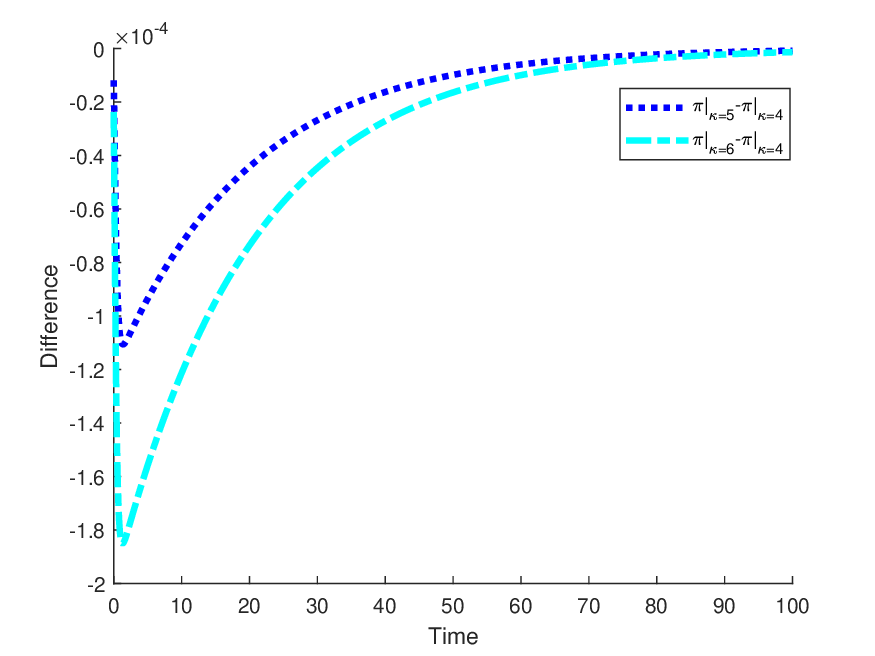}
      (c) Case (I) when $T=100$
    \end{minipage}
    \begin{minipage}{5cm}
      \includegraphics[width=5cm]{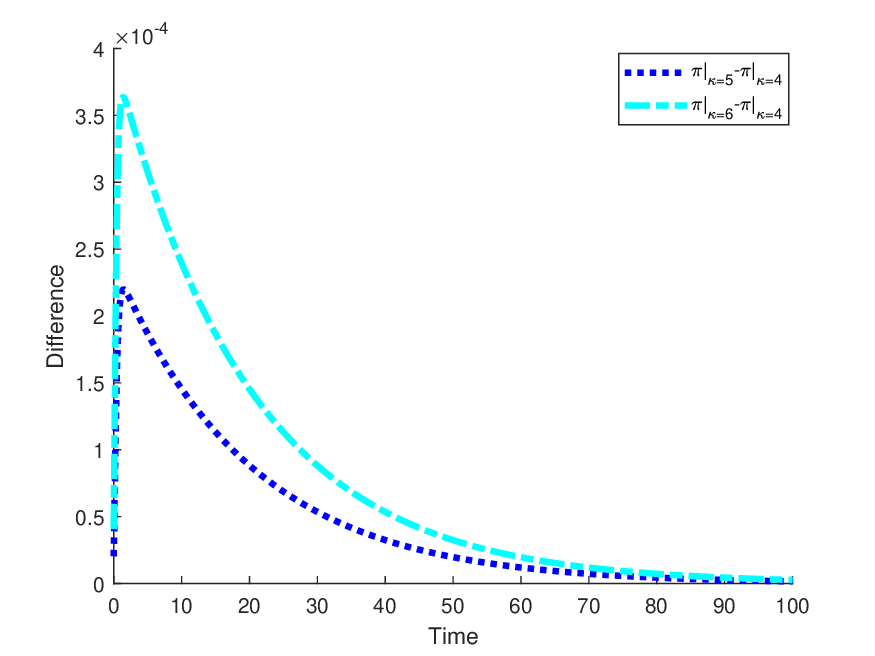}
      \caption*{(d) Case (II) when $T=100$}
    \end{minipage}
    \caption{The differences between the equilibrium investment strategies w.r.t. $\kappa$ when $\rho=-0.5$.}
\label{fig31}
\end{figure}

% Figure \ref{fig3} illustrates the impacts of the mean-reversion rate $\kappa$ on the equilibrium investment strategy when $\rho=-0.5$. As the impacts of $\kappa$ on the equilibrium investment strategy are not obvious, we provide Figure \ref{fig31} to show the sensitivity of the equilibrium investment strategy w.r.t. $\kappa$ when $\rho=-0.5$.

\begin{figure}
    \centering
    \begin{minipage}{5cm}
      \includegraphics[width=5cm]{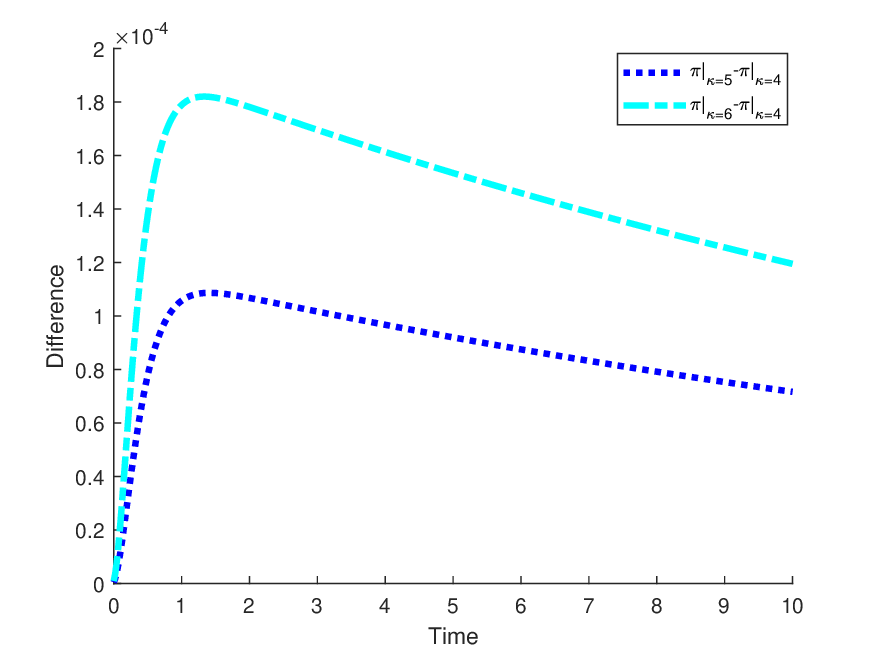}
      (a) Case (I) when $T=10$
    \end{minipage}
     \begin{minipage}{5cm}
      \includegraphics[width=5cm]{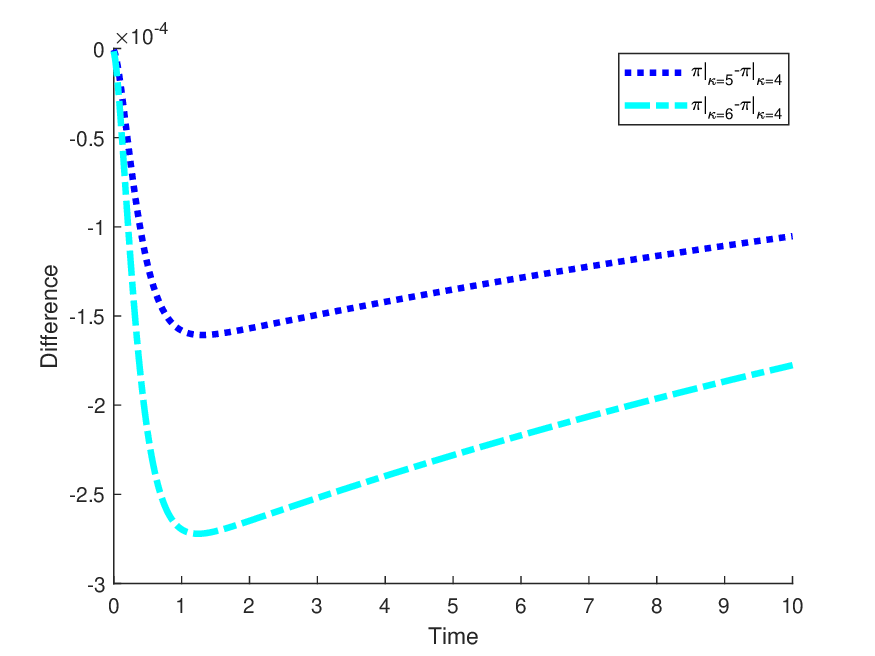}
      (b) Case (II) when $T=10$
    \end{minipage}

    \begin{minipage}{5cm}
      \includegraphics[width=5cm]{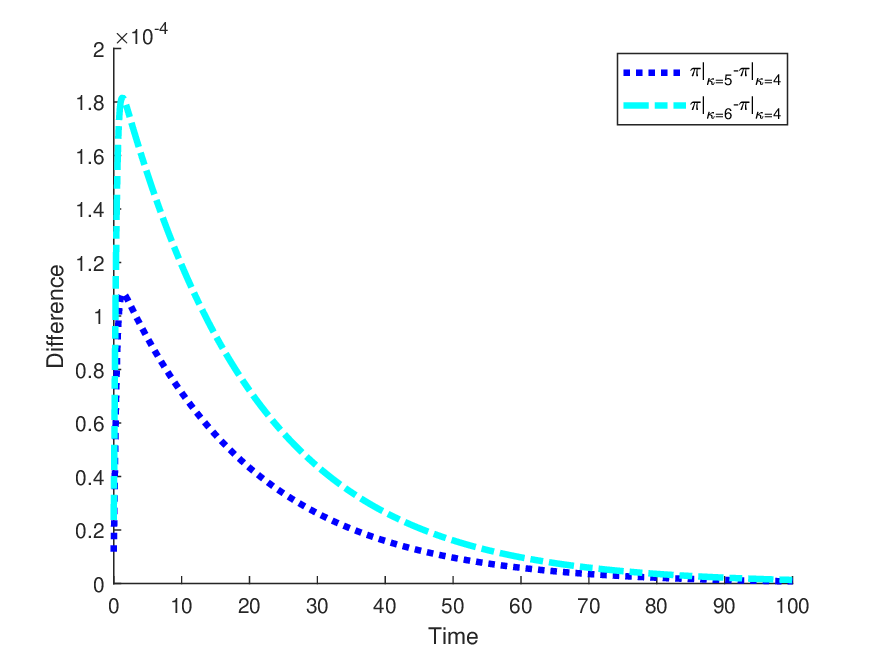}
      (c) Case (I) when $T=100$
    \end{minipage}
    \begin{minipage}{5cm}
      \includegraphics[width=5cm]{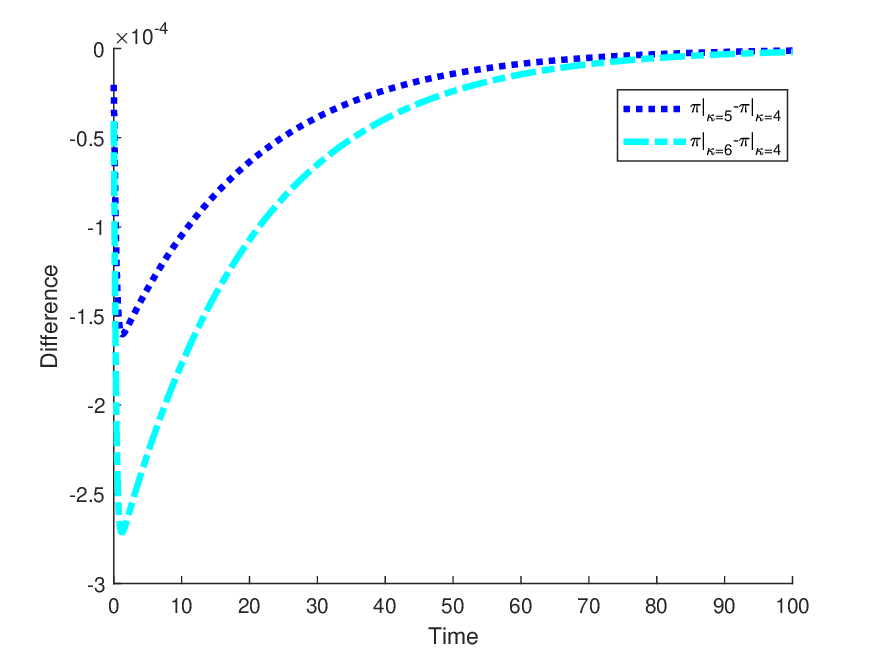}
      \caption*{(d) Case (II) when $T=100$}
    \end{minipage}
    \caption{The differences between the equilibrium investment strategies w.r.t. $\kappa$ when $\rho=0.5$.}
\label{fig32}
\end{figure}

Figures \ref{fig31} and \ref{fig32} illustrate the impacts of the mean-reversion rate $\kappa$ on the equilibrium investment strategy. It is interesting to see from Figure \ref{fig31} (a) that when $\rho<0$, the differences between equilibrium investment strategies under different mean-reversion rates generate $U$-shape lines with the increase of $t$. Overall, the differences between equilibrium investment strategies under different mean-reversion rates are negative, which implies that when $\mathbb{E}[\gamma]=2.25$, the insurer who has a higher mean-reversion rate will allocate less wealth to the risky asset for investment. Furthermore, the higher the mean-reversion rate is, the smaller the amount of investment in the risky asset will be. In contrast, Figure \ref{fig31} (b) reveals an opposite situation. As a whole the differences between equilibrium investment strategies under different mean-reversion rates are positive, which shows that when $\mathbb{E}[\gamma]=1.2$, the insurer with higher mean-reversion rate will invest a greater amount of wealth in the risky asset. Specially, the greater the mean-reversion rate, the more the amount of money invested in the risky asset. What is common is that, Figure \ref{fig31} (c) and (d) disclose that as time $t$ approaches 100, the differences between equilibrium investment strategies under different mean-reversion rates tend to zero. Since $\rho<0$,  the fluctuations in the risky asset price and the changes in the volatility of the risky asset occur in different ways. If $\kappa$ goes up, then $V(t)$ will become more stable. Due to the negative correlation between $V(t)$ and $S(t)$, there exists an elevated probability of a decline in the price of the risky asset. Therefore, the insurer having higher expected value of risk aversion would decrease the investment in the risky asset. However, although there is a decreased probability of a rise in the risky asset's price, the insurer with lower expected value of risk aversion is willing to take risks to boost the investment amount in the risky asset. For the purpose of reflecting the impacts of the mean-reversion rate $\kappa$ on the equilibrium investment strategy when $\rho>0$, Figure \ref{fig32} is provided, which presents opposite situations to Figure \ref{fig31}. If $\rho>0$, then the uncertainties regarding the risky asset's price and its volatility vary in a congruent manner. A larger value of $\kappa$ indicates a more stabilized volatility level of the risky asset. Therefore, the insurer with higher expected value of risk aversion is apt to enlarge the investment amount in the risky asset. However, the insurer with lower expected value of risk aversion is more inclined to accept risks and to pursue high returns on the risky asset with greater volatility, thereby decreasing the investment amount when the value of $\kappa$ increases. These diverse situations imply that it is meaningful and necessary to consider the random risk aversion for investment in the incomplete financial market.

%\begin{figure}
%   \centering
%   \begin{minipage}{5cm}
%      \includegraphics[width=5cm]{a41}
%      (a) Case (I) when $T=10$
%   \end{minipage}
%     \begin{minipage}{5cm}
%      \includegraphics[width=5cm]{a42}
%      (b) Case (II) when $T=10$
%    \end{minipage}

%    \begin{minipage}{5cm}
%      \includegraphics[width=5cm]{a43}
%      (c) Case (I) when $T=100$
%    \end{minipage}
%    \begin{minipage}{5cm}
%      \includegraphics[width=5cm]{a44}
%      \caption*{(d) Case (II) when $T=100$}
%    \end{minipage}
%    \caption{The impacts of $\sigma$ on the equilibrium investment strategy when $\rho=-0.5$.}
%\label{fig4}
%\end{figure}

\begin{figure}
    \centering
    \begin{minipage}{5cm}
      \includegraphics[width=5cm]{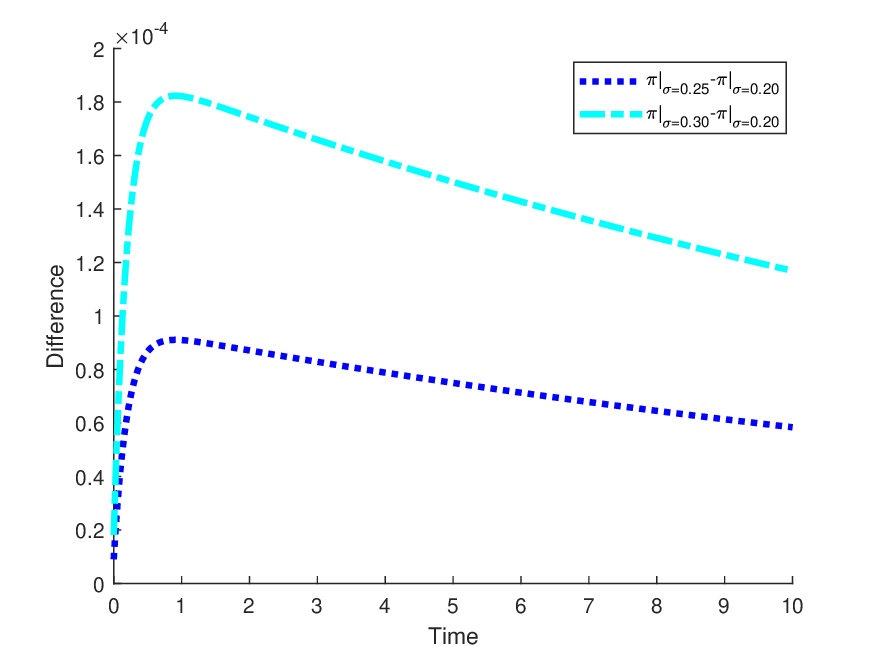}
      (a) Case (I) when $T=10$
    \end{minipage}
    \begin{minipage}{5cm}
      \includegraphics[width=5cm]{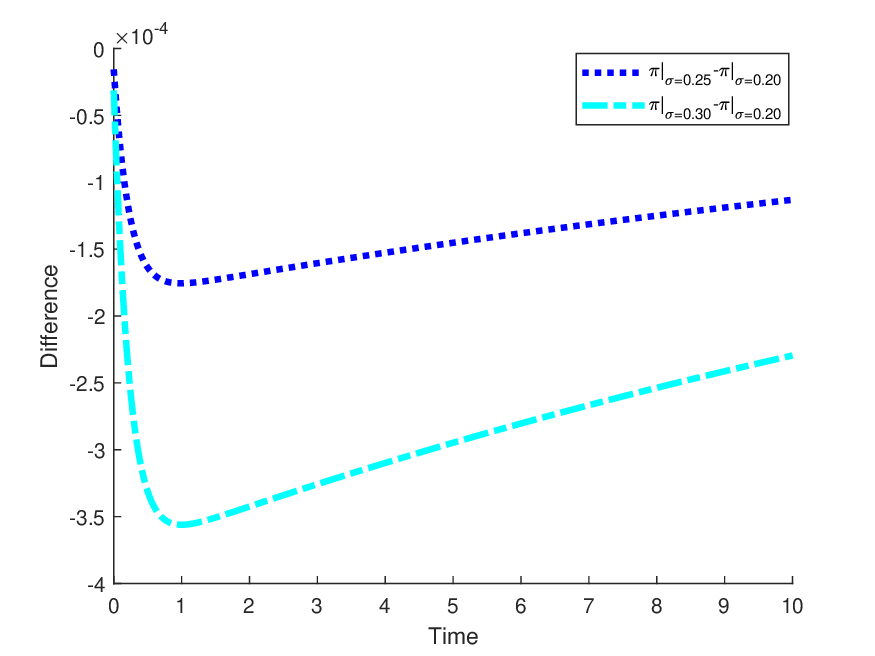}
      (b) Case (II) when $T=10$
   \end{minipage}

   \begin{minipage}{5cm}
     \includegraphics[width=5cm]{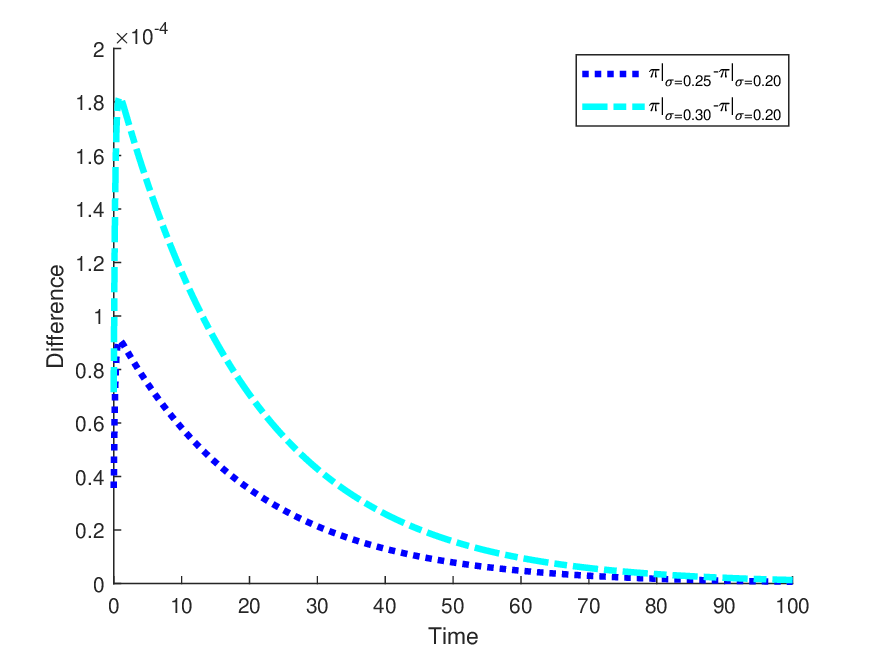}
     (c) Case (I) when $T=100$
    \end{minipage}
    \begin{minipage}{5cm}
      \includegraphics[width=5cm]{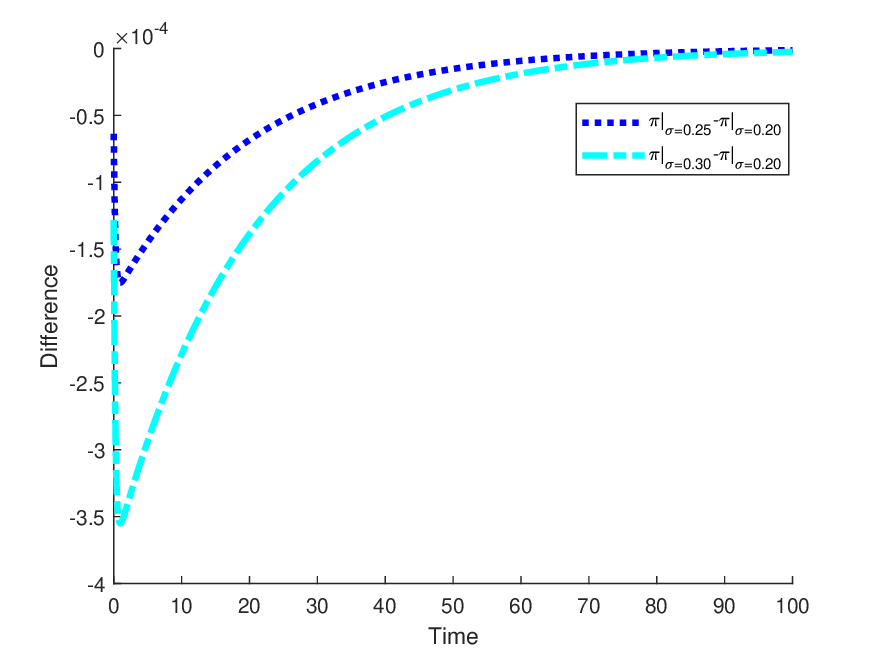}
      \caption*{(d) Case (II) when $T=100$}
    \end{minipage}
    \caption{The differences between the equilibrium investment strategies w.r.t. $\sigma$ when $\rho=-0.5$.}
\label{fig41}
\end{figure}

\begin{figure}
    \centering
    \begin{minipage}{5cm}
      \includegraphics[width=5cm]{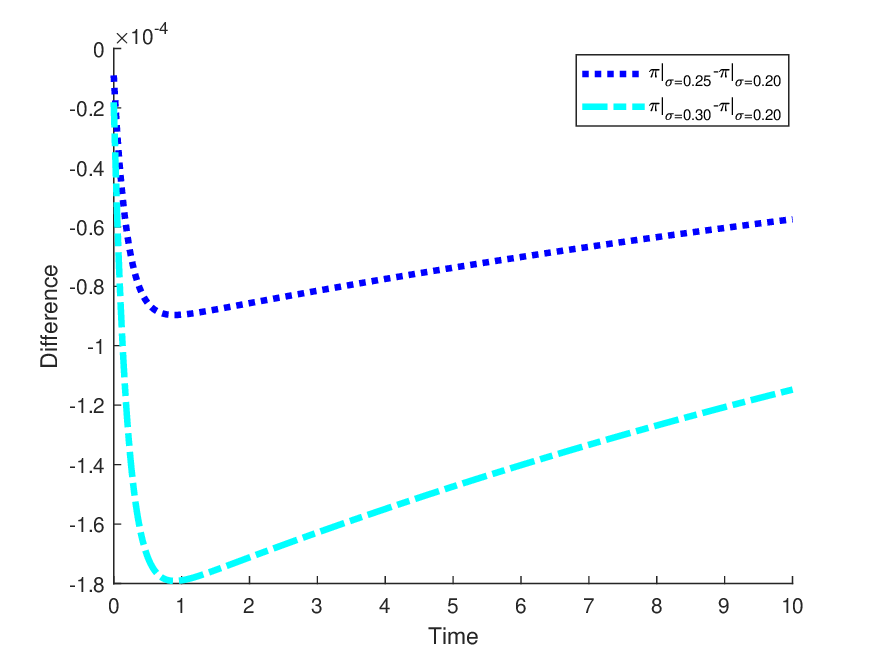}
      (a) Case (I) when $T=10$
    \end{minipage}
     \begin{minipage}{5cm}
      \includegraphics[width=5cm]{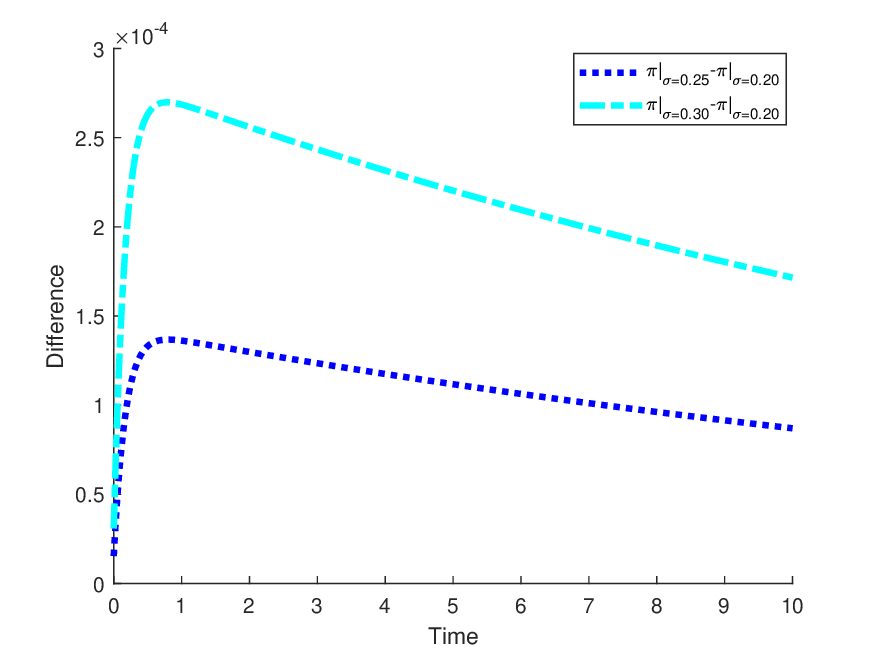}
      (b) Case (II) when $T=10$
    \end{minipage}

    \begin{minipage}{5cm}
      \includegraphics[width=5cm]{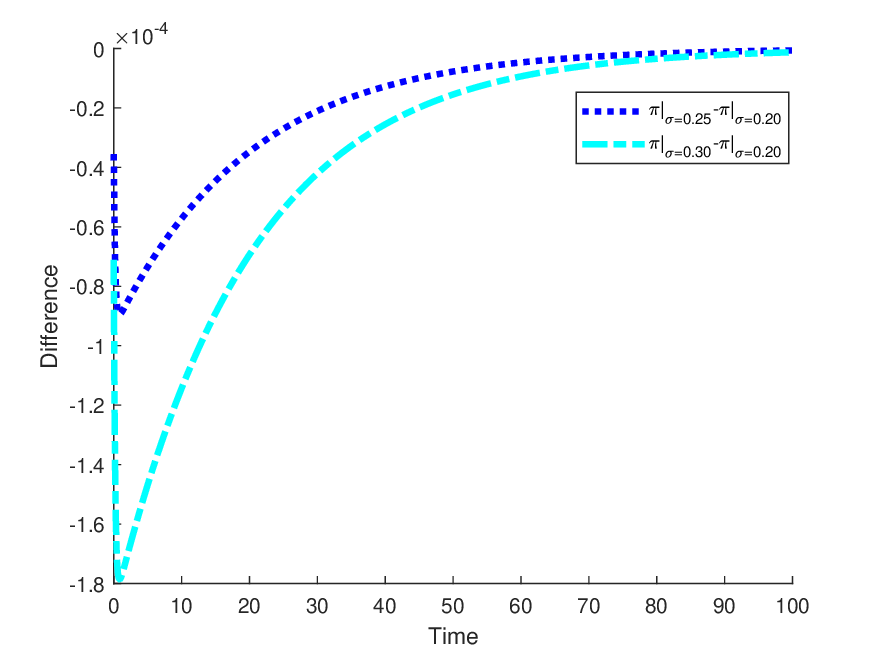}
      (c) Case (I) when $T=100$
    \end{minipage}
    \begin{minipage}{5cm}
      \includegraphics[width=5cm]{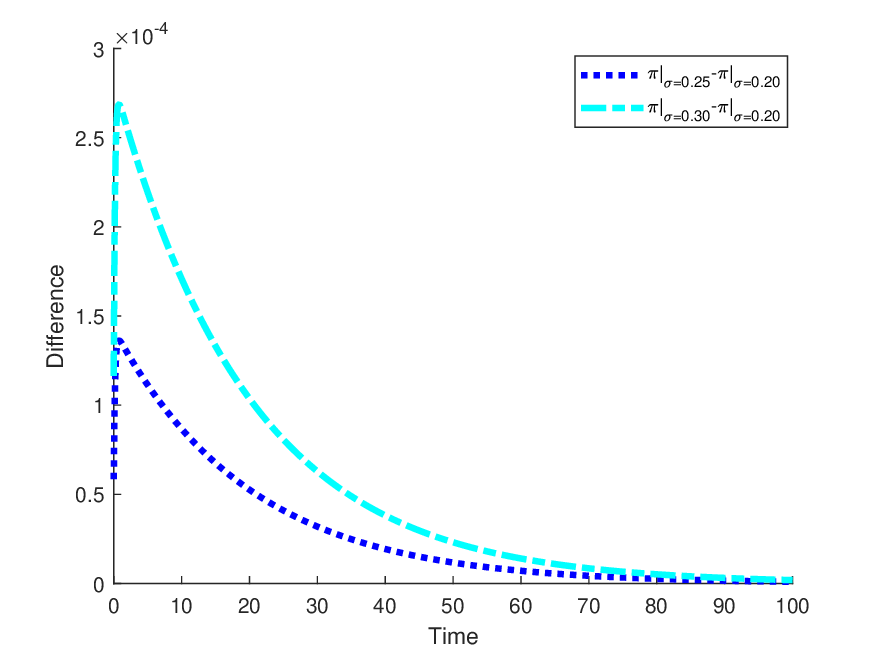}
      \caption*{(d) Case (II) when $T=100$}
    \end{minipage}
    \caption{The differences between the equilibrium investment strategies w.r.t. $\sigma$ when $\rho=0.5$.}
\label{fig42}
\end{figure}
Figures \ref{fig41} and \ref{fig42} display the impacts of $\sigma$ on the equilibrium investment strategy. As shown in Figure \ref{fig41} (a), we see that the differences between the equilibrium investment strategies corresponding to different values of $\sigma$ are positive when $\rho<0$, which means that if $\mathbb{E}[\gamma]=2.25$, then the insurer with larger values of $\sigma$ will allocate more funds to the risky asset for investment. In comparison to Figure \ref{fig41} (a), Figure \ref{fig41} (b) shows an opposite situation that the differences between the equilibrium investment strategies with various values of $\sigma$ are negative. Specially, when $\mathbb{E}[\gamma]=1.2$, the insurer with higher values of $\sigma$ will allocate a lesser amount to the risky asset for investment. Furthermore, Figure \ref{fig41} (c) and (d) illustrate that the differences tend to zero when $t$ approaches 100. On the other hand, Figure \ref{fig42} plots the differences between the equilibrium investment strategies for diverse values of $\sigma$ given that $\rho>0$. Overall, the phenomena reflected in Figure \ref{fig42} are exactly opposite to those in Figure \ref{fig41}. It is widely known that if there is an increase in $\sigma$, then the fluctuation of the risky asset's volatility will be a bit more intense. Thus, when $\rho>0$, the insurer possessing a greater expected value of risk aversion will cut down on the investment amount in the risky asset to avoid risks, while the insurer with lower expected value of risk aversion tend to pursue high returns in the risky asset with high volatility by raising the amount of investment in the risky asset. These correspond to the situations in Figure \ref{fig42}. When $\rho<0$, the insurer with higher expected value of risk aversion would put a larger amount of the fortune into the risky asset to counterbalance risks, while the insurer with lower expected value of risk aversion is willing to reduce the investment amount in the risky asset. There are in line with the phenomena reflected in Figure \ref{fig41}. Thus, it is quite necessary to study the random risk aversion for the portfolio selection problems in incomplete financial markets.

%\begin{figure}
%    \centering
%    \begin{minipage}{5cm}
%      \includegraphics[width=5cm]{a51}
%      (a) Case (I) when $T=10$
%    \end{minipage}
%     \begin{minipage}{5cm}
%      \includegraphics[width=5cm]{a52}
%      (b) Case (II) when $T=10$
%    \end{minipage}

%    \begin{minipage}{5cm}
%      \includegraphics[width=5cm]{a53}
%      (c) Case (I) when $T=100$
%    \end{minipage}
%    \begin{minipage}{5cm}
%      \includegraphics[width=5cm]{a54}
%      \caption*{(d) Case (II) when $T=100$}
%   \end{minipage}
%   \caption{The impacts of $\rho$ on the equilibrium investment strategy.}
%\label{fig5}
%\end{figure}

\begin{figure}
    \centering
    \begin{minipage}{5cm}
      \includegraphics[width=5cm]{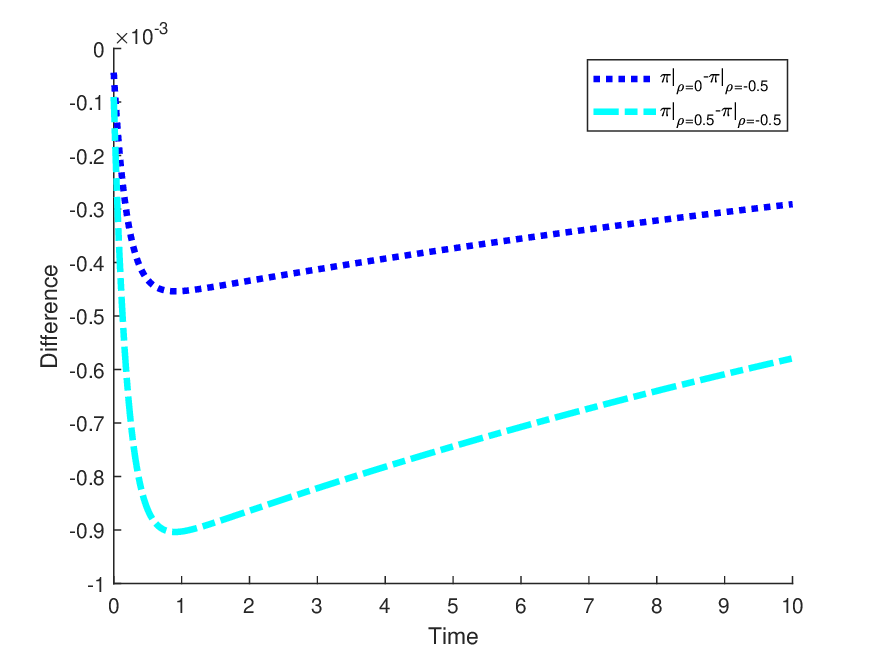}
     (a) Case (I) when $T=10$
    \end{minipage}
     \begin{minipage}{5cm}
      \includegraphics[width=5cm]{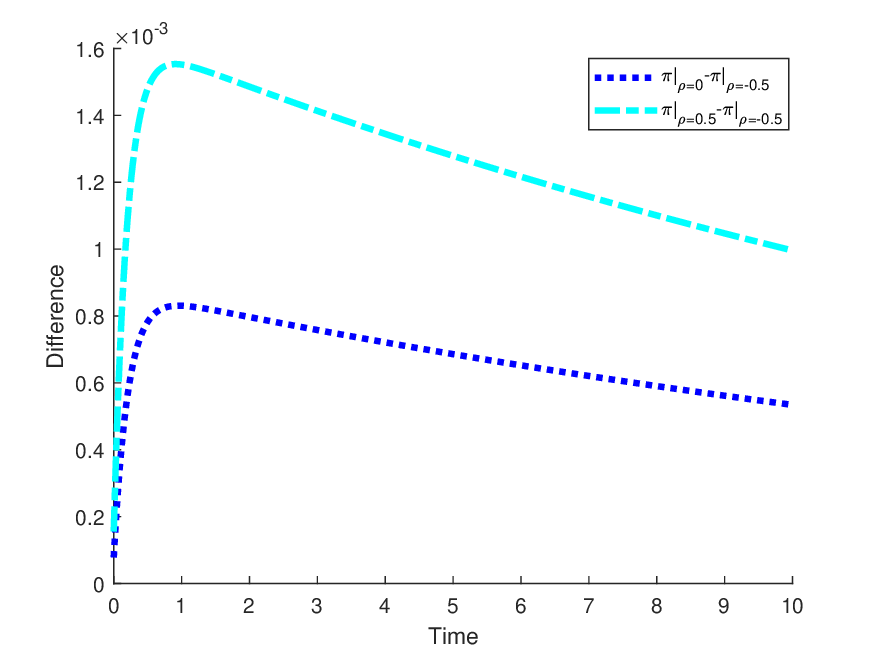}
      (b) Case (II) when $T=10$
    \end{minipage}

    \begin{minipage}{5cm}
      \includegraphics[width=5cm]{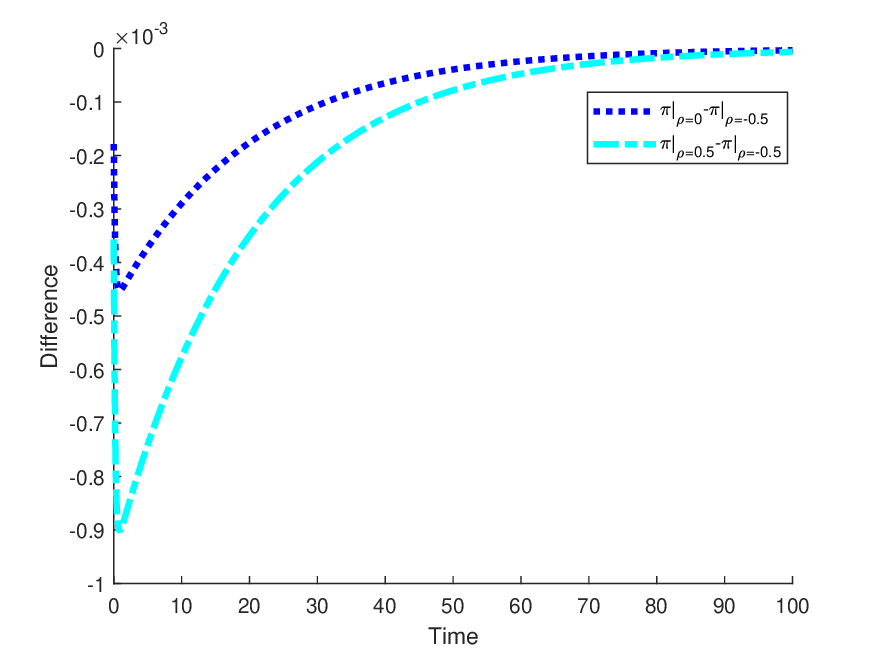}
      (c) Case (I) when $T=100$
    \end{minipage}
    \begin{minipage}{5cm}
      \includegraphics[width=5cm]{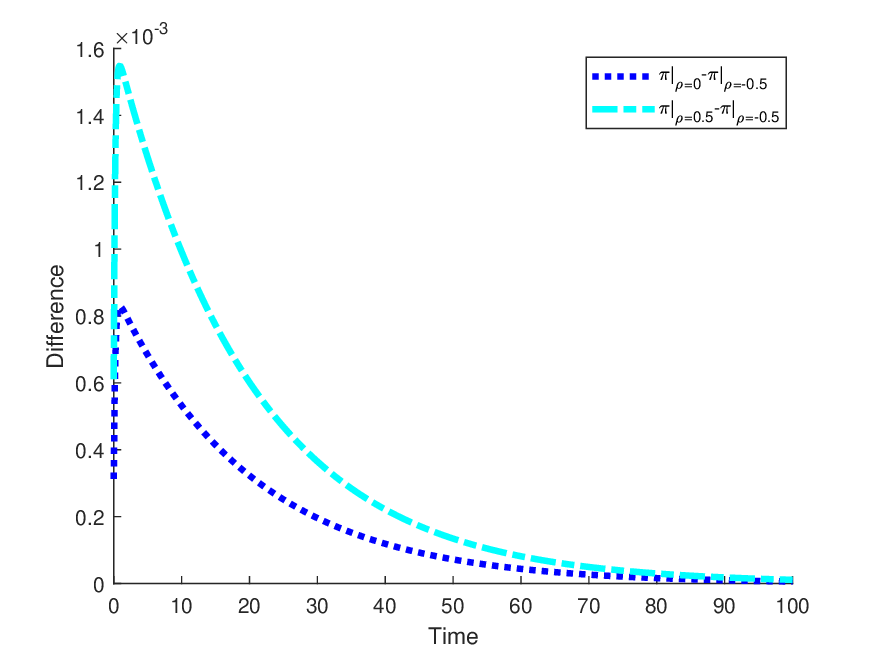}
      \caption*{(d) Case (II) when $T=100$}
    \end{minipage}
    \caption{The differences between the equilibrium investment strategies w.r.t. $\rho$.}
\label{fig51}
\end{figure}

Figure \ref{fig51} shows how the correlation coefficient $\rho$ impacts the equilibrium investment strategy. It can be concluded from Figure \ref{fig51} (a) and (b) that the differences between the equilibrium investment tactics corresponding to different values of $\rho$ show opposite trends for the cases (I) and (II). Figure \ref{fig51} (a) implies that the equilibrium investment strategy diminishes in tandem with the growth of  $\rho$, while Figure \ref{fig51} (b) reflects that the equilibrium investment strategy increases with $\rho$. Moreover, \ref{fig51} (c) and (d) illustrate that the differences will approach zero as $t$ tends to 100. Since $\rho$ denotes the correlation between $S(t)$ and $V(t)$, the fluctuations in $S(t)$ and $V(t)$ occur in an identical manner when $\rho>0$. Therefore, with the increase of $\rho$, the insurer possessing a greater expected value of risk aversion will curtail the investment sum in the risky asset to avert risks. Conversely, when $\rho$ takes a negative value, the instabilities of $S(t)$ and $V(t)$ alter in disparate ways. Then the insurer whose expected value of risk aversion is on the higher side would augment the investment sum in the risky asset for the purpose of hedging risks. Compared with the insurer having higher expected value of risk aversion, the insurer having lower expected value of risk aversion is willing to adopt opposite investment strategies to accept risks for the purposes of obtaining high returns.

\subsection{Impacts of model parameters on the equilibrium reinsurance strategy}
In this subsection, we will present some sensitivity analysis on the impacts of parameters upon the equilibrium reinsurance strategy. Meanwhile, we also provide some numerical examples to illustrate the impacts.

It follows from the equation \eqref{eq43} that there is an analytical expression for the equilibrium reinsurance strategy. Therefore, according to the equation \eqref{eq43}, we can derive that $\frac{\partial \hat{q}(t)}{\partial r}<0$, $\frac{\partial \hat{q}(t)}{\partial \eta_{2}}>0$, $\frac{\partial \hat{q}(t)}{\partial \lambda_{1}}=0$, $\frac{\partial \hat{q}(t)}{\partial \mu_{1}}>0$ and $\frac{\partial \hat{q}(t)}{\partial \mu_{2}}<0$, which implies that the equilibrium reinsurance strategy $\hat{q}(t)$ shows a decreasing trend in relation to the risk-free interest rate $r$ and the second-order moment of the claim sizes $\mu_{2}$, exhibits an increasing tendency regarding the safety loading of the reinsurer $\eta_{2}$ and the first-order moment of the claim sizes $\mu_{1}$, and remains unchanged with regard to the intensity of the number of claims $\lambda_{1}$. Then, we give some numerical illustrations in Figures \ref{fig7}-\ref{fig11}.

\begin{figure}
    \centering
    \begin{minipage}{5cm}
      \includegraphics[width=5cm]{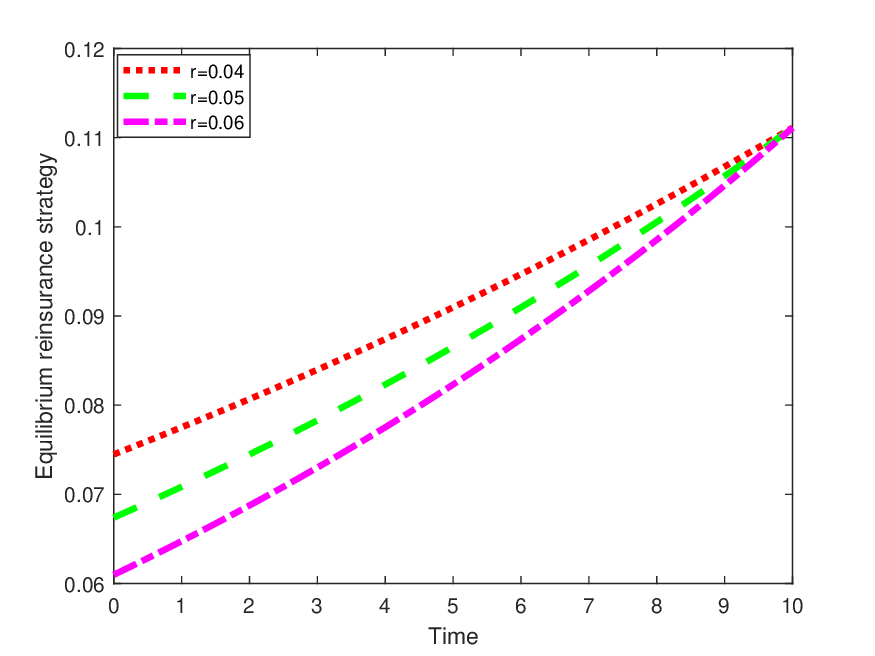}
      (a) Case (I) when $T=10$
    \end{minipage}
     \begin{minipage}{5cm}
      \includegraphics[width=5cm]{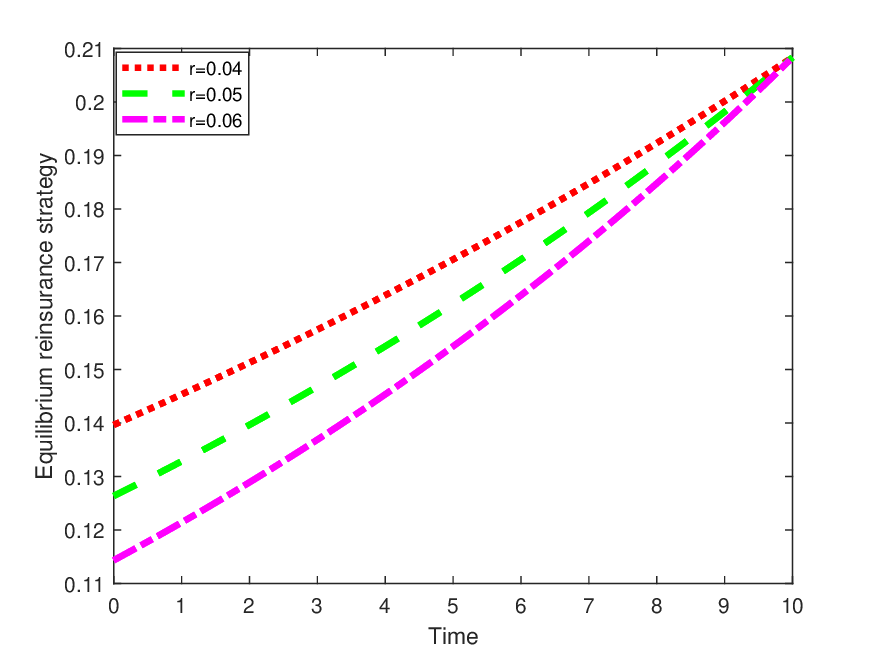}
      (b) Case (II) when $T=10$
    \end{minipage}

    \begin{minipage}{5cm}
      \includegraphics[width=5cm]{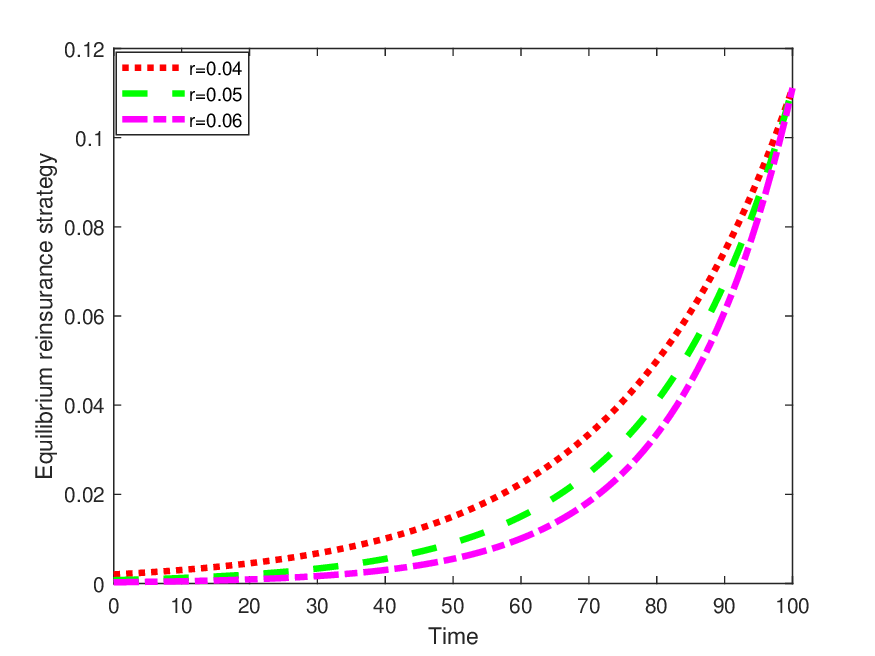}
      (c) Case (I) when $T=100$
    \end{minipage}
    \begin{minipage}{5cm}
      \includegraphics[width=5cm]{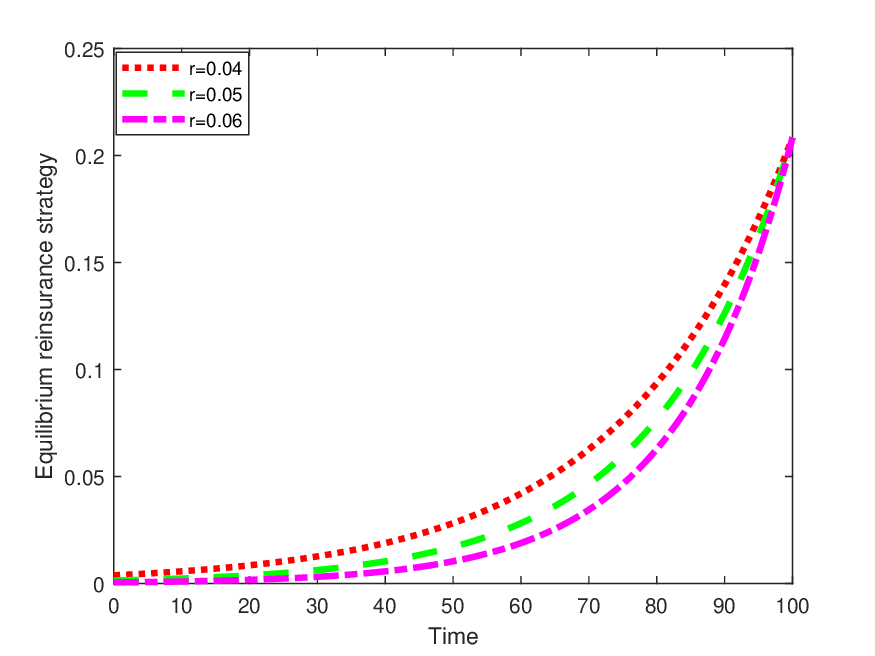}
      \caption*{(d) Case (II) when $T=100$}
    \end{minipage}
    \caption{The impacts of $r$ on the equilibrium reinsurance strategy.}
\label{fig7}

\end{figure}

\begin{figure}
    \centering
    \begin{minipage}{5cm}
      \includegraphics[width=5cm]{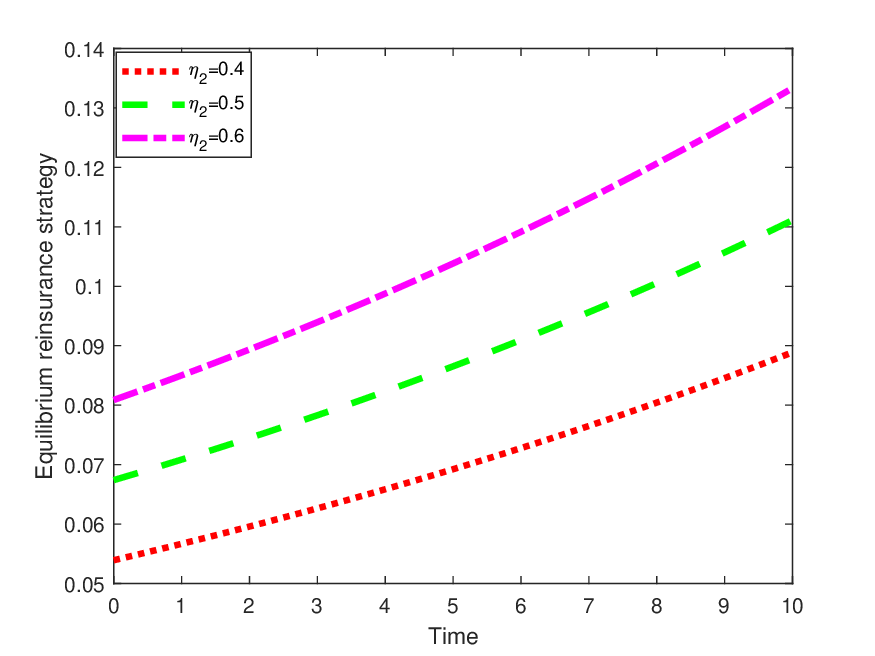}
      (a) Case (I) when $T=10$
    \end{minipage}
     \begin{minipage}{5cm}
      \includegraphics[width=5cm]{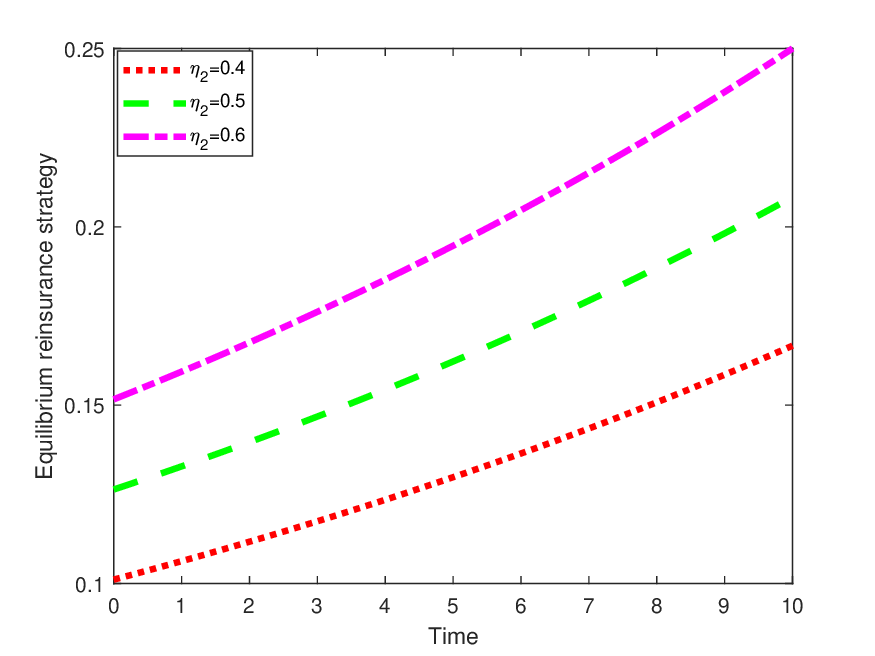}
      (b) Case (II) when $T=10$
    \end{minipage}

    \begin{minipage}{5cm}
      \includegraphics[width=5cm]{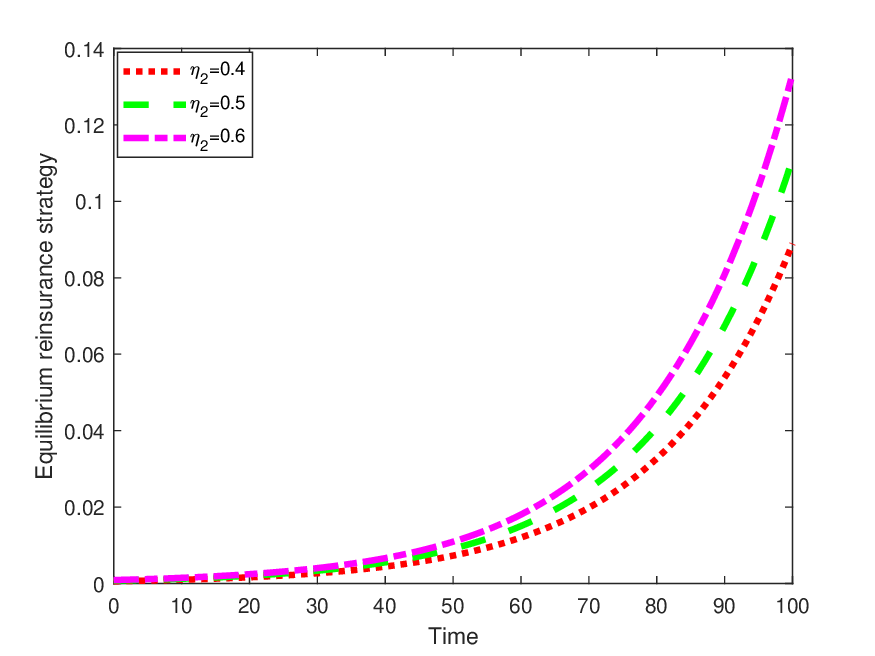}
      (c) Case (I) when $T=100$
    \end{minipage}
    \begin{minipage}{5cm}
      \includegraphics[width=5cm]{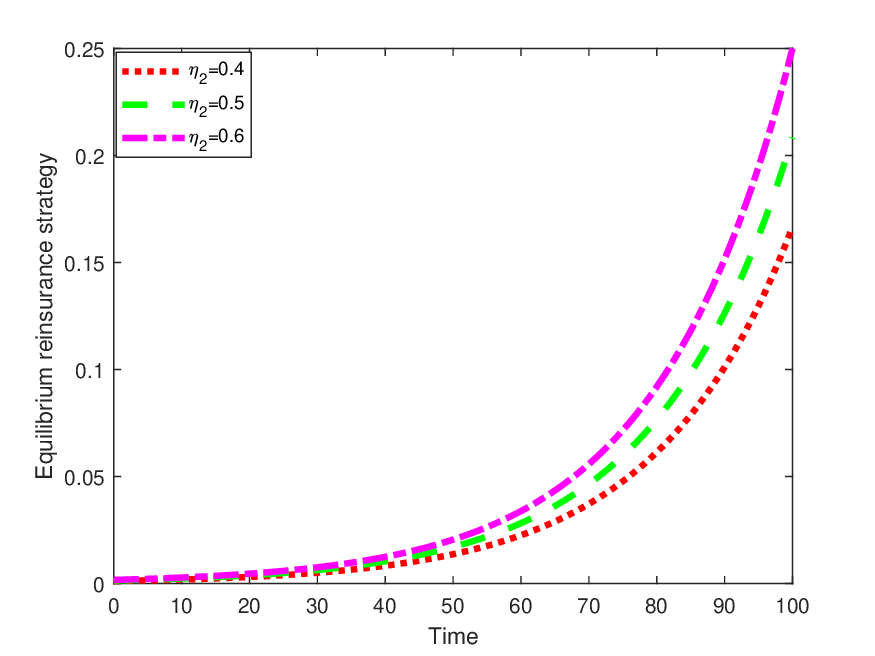}
      \caption*{(d) Case (II) when $T=100$}
    \end{minipage}
    \caption{The impacts of $\eta_{2}$ on the equilibrium reinsurance strategy.}
\label{fig8}

\end{figure}

\begin{figure}
    \centering
    \begin{minipage}{5cm}
      \includegraphics[width=5cm]{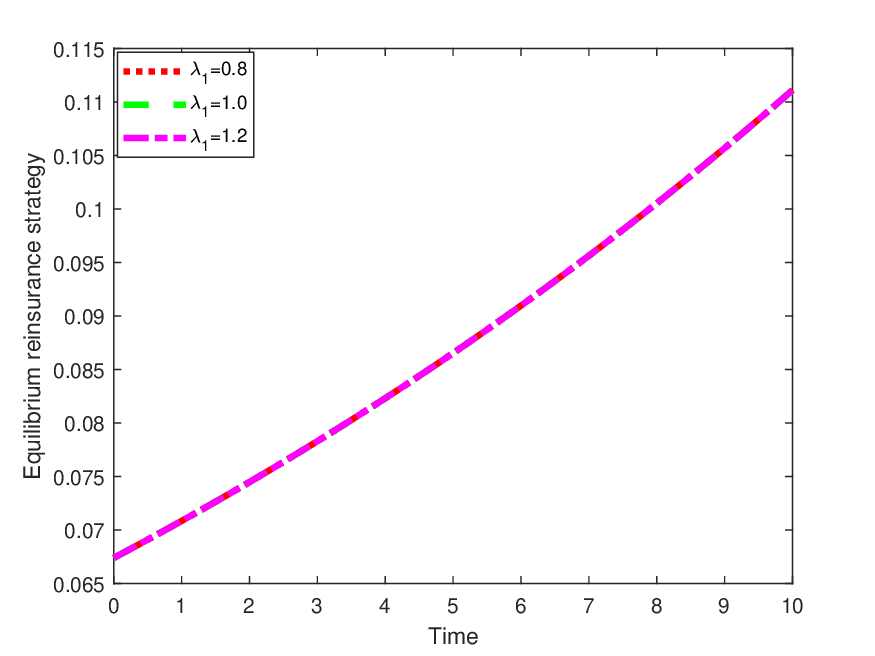}
      (a) Case (I) when $T=10$
    \end{minipage}
     \begin{minipage}{5cm}
      \includegraphics[width=5cm]{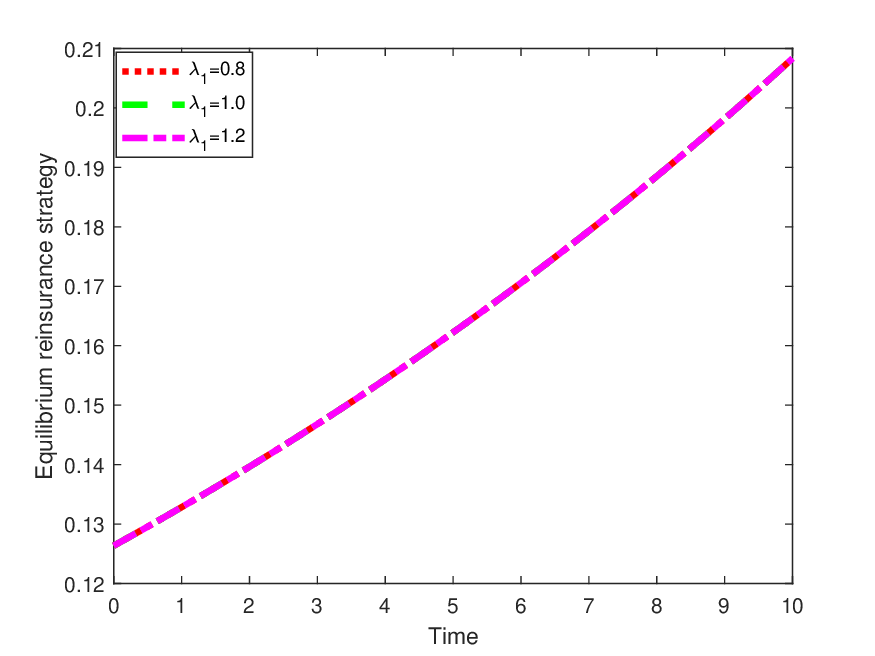}
      (b) Case (II) when $T=10$
    \end{minipage}

    \begin{minipage}{5cm}
      \includegraphics[width=5cm]{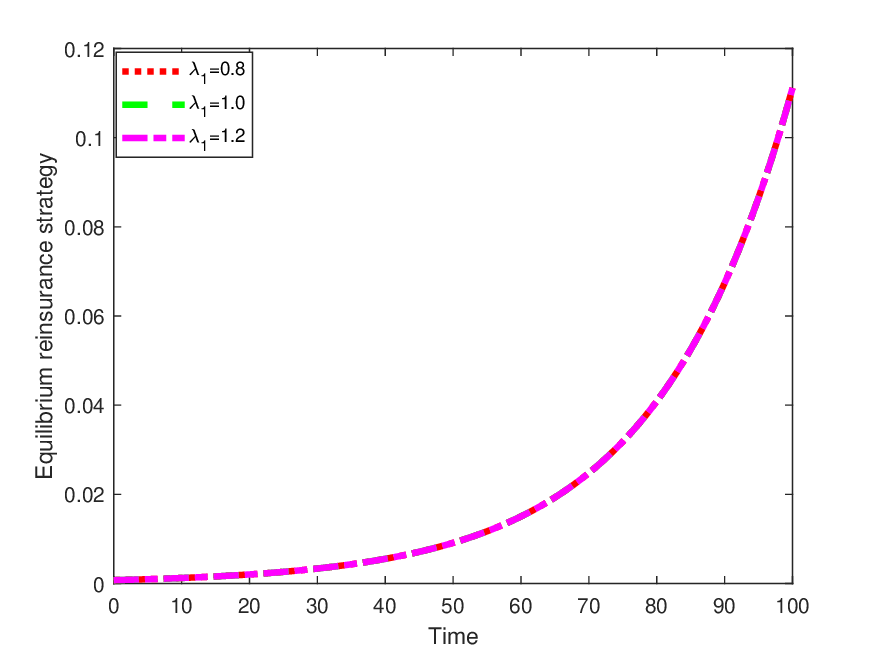}
      (c) Case (I) when $T=100$
    \end{minipage}
    \begin{minipage}{5cm}
      \includegraphics[width=5cm]{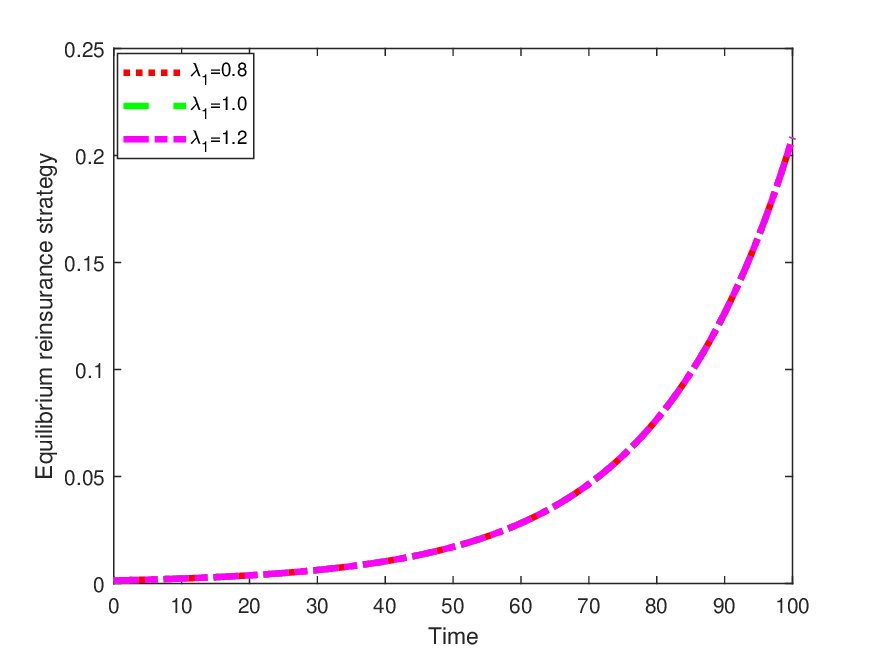}
      \caption*{(d) Case (II) when $T=100$}
    \end{minipage}
    \caption{The impacts of $\lambda_{1}$ on the equilibrium reinsurance strategy.}
\label{fig9}

\end{figure}

\begin{figure}
    \centering
    \begin{minipage}{5cm}
      \includegraphics[width=5cm]{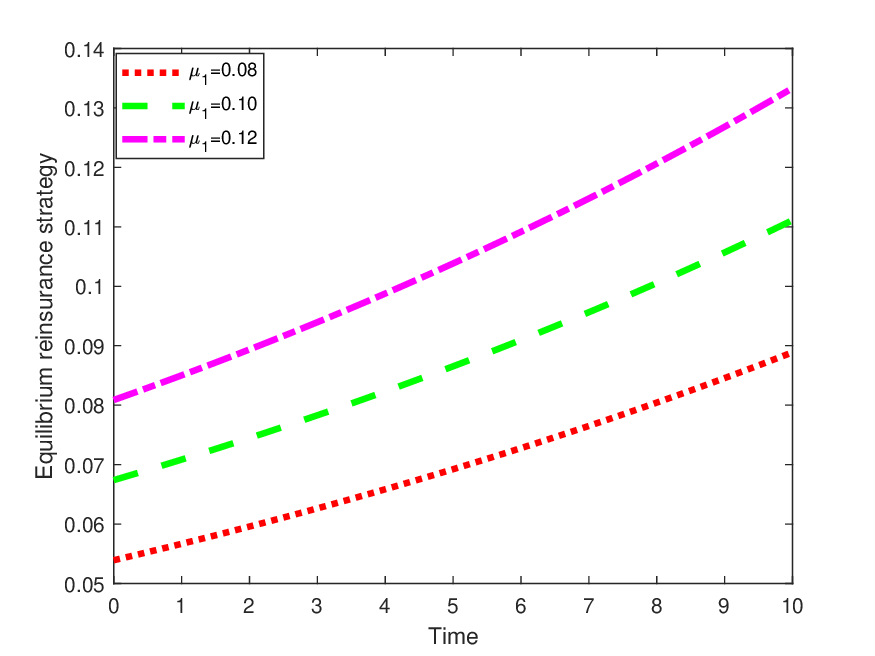}
      (a) Case (I) when $T=10$
    \end{minipage}
     \begin{minipage}{5cm}
      \includegraphics[width=5cm]{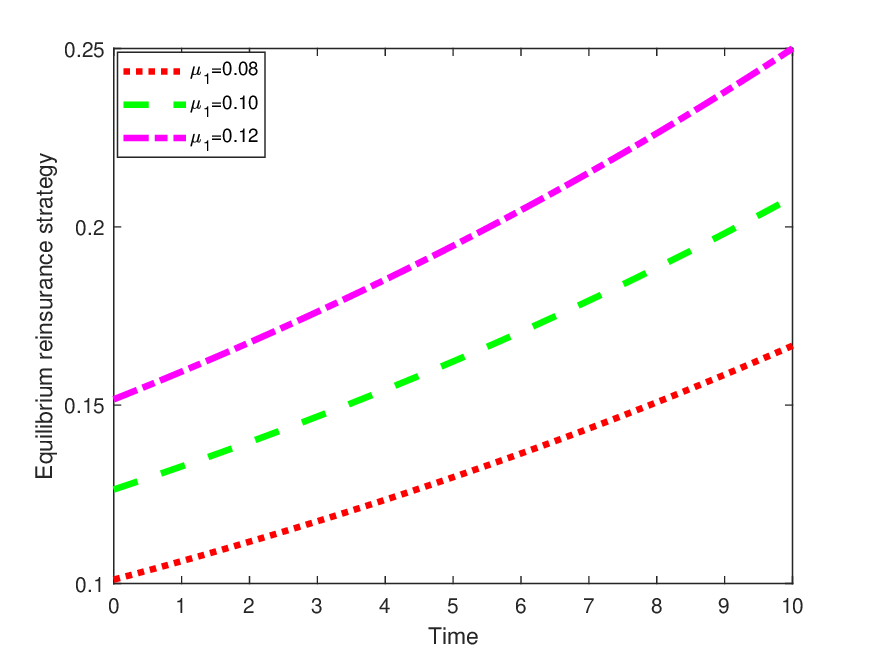}
      (b) Case (II) when $T=10$
    \end{minipage}

    \begin{minipage}{5cm}
      \includegraphics[width=5cm]{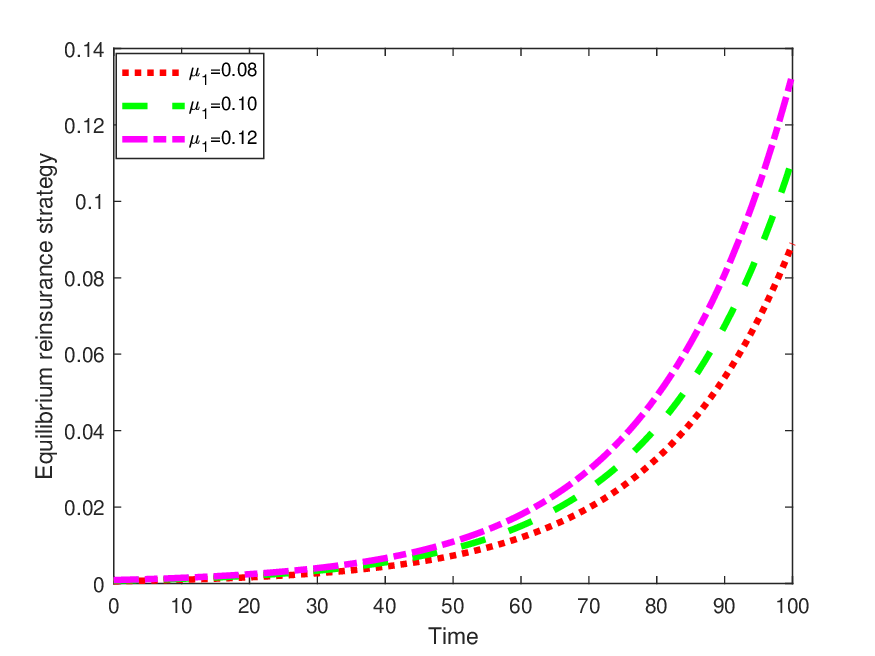}
      (c) Case (I) when $T=100$
    \end{minipage}
    \begin{minipage}{5cm}
      \includegraphics[width=5cm]{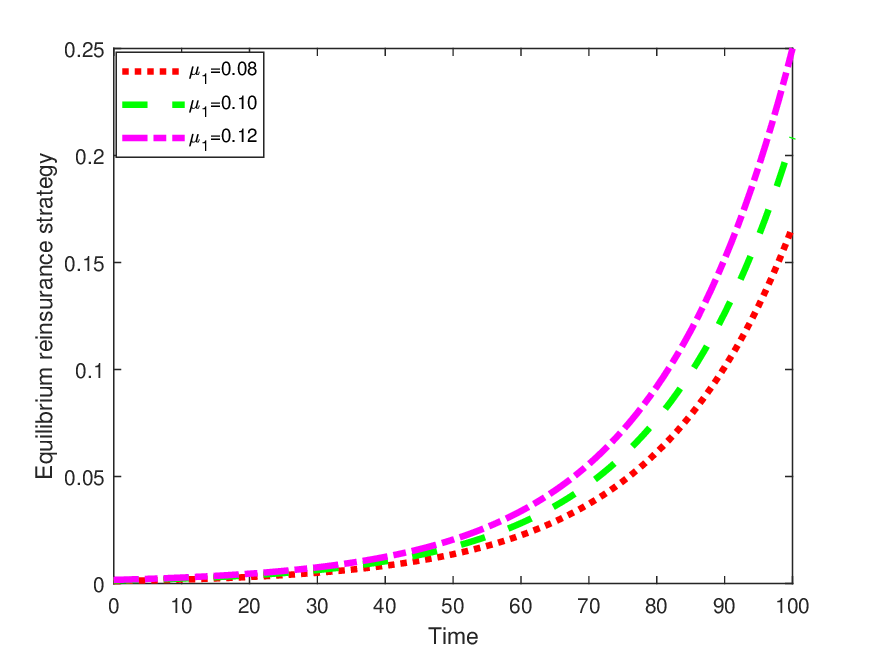}
      \caption*{(d) Case (II) when $T=100$}
    \end{minipage}
    \caption{The impacts of $\mu_{1}$ on the equilibrium reinsurance strategy.}
\label{fig10}

\end{figure}

\begin{figure}
    \centering
    \begin{minipage}{5cm}
      \includegraphics[width=5cm]{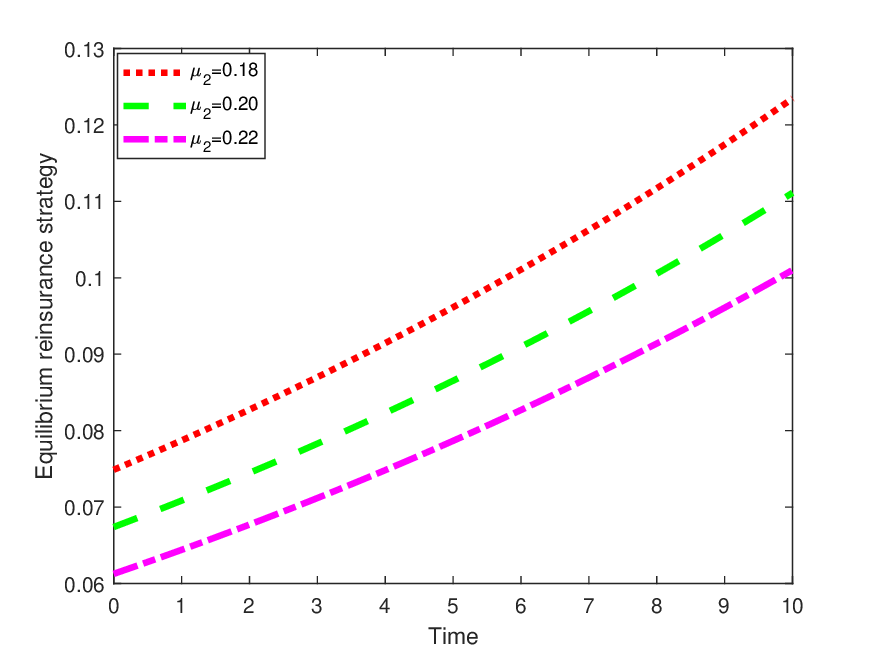}
      (a) Case (I) when $T=10$
    \end{minipage}
     \begin{minipage}{5cm}
      \includegraphics[width=5cm]{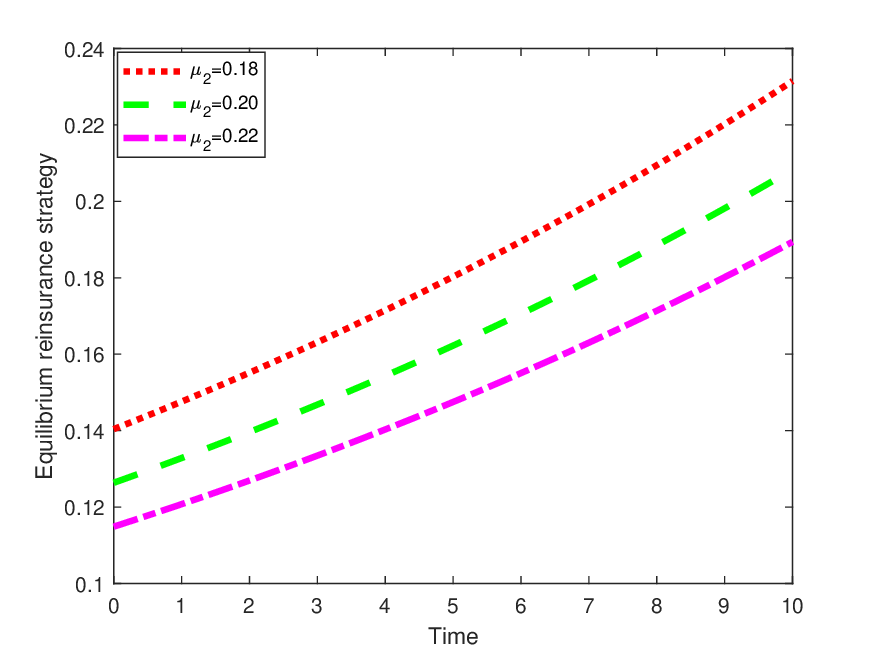}
      (b) Case (II) when $T=10$
    \end{minipage}

    \begin{minipage}{5cm}
      \includegraphics[width=5cm]{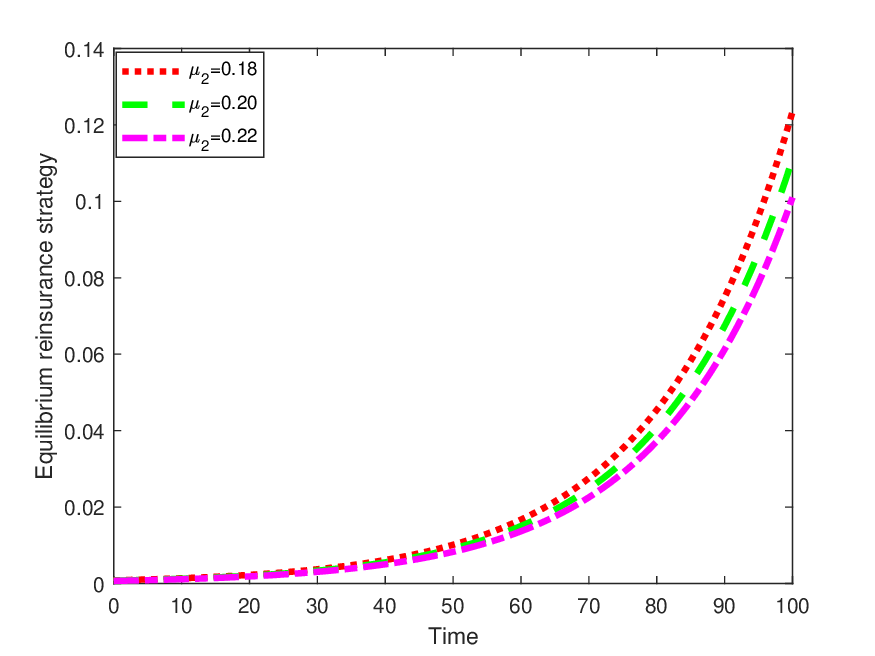}
      (c) Case (I) when $T=100$
    \end{minipage}
    \begin{minipage}{5cm}
      \includegraphics[width=5cm]{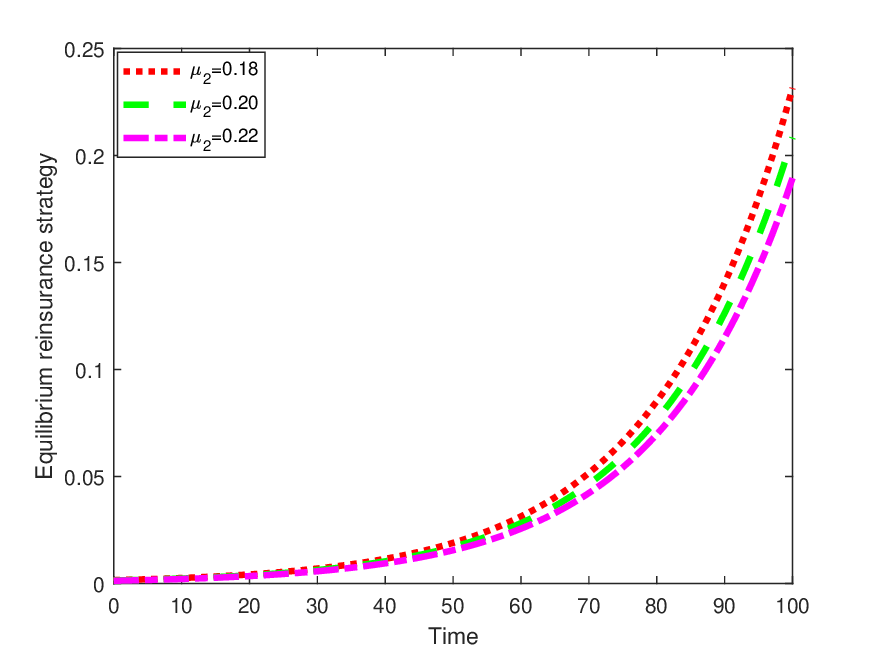}
      \caption*{(d) Case (II) when $T=100$}
    \end{minipage}
    \caption{The impacts of $\mu_{2}$ on the equilibrium reinsurance strategy.}
\label{fig11}

\end{figure}

From Figure \ref{fig7}, we can see that $0<\hat{q}(t)<1$ and the equilibrium reinsurance strategy $\hat{q}(t)$ decreases concerning $r$, which indicates that the insurer should obtain a greater amount of reinsurance if $r$ becomes larger. Moreover, Figure \ref{fig7} shows that the insurer having a greater expected value of risk aversion will acquire additional reinsurance, which is also reflected in the following Figures \ref{fig8}-\ref{fig11}. In addition, it is interesting to see from Figure \ref{fig7} that as the time $t$ becomes greater, the equilibrium reinsurance strategy $\hat{q}(t)$ also becomes larger. However, this indicates that the insurer will reduce the amount of reinsurance purchased when $t$ rises. This phenomenon also exists in the following Figures \ref{fig8}-\ref{fig11}.

Figure \ref{fig8} characterizes the impacts of the safety loading of the reinsurer $\eta_{2}$ upon the equilibrium reinsurance strategy. To be more specific, the equilibrium reinsurance strategy grows in relation to $\eta_{2}$. This conforms to our intuition. Since a greater $\eta_{2}$ signifies an elevated reinsurance expense, the insurer prefers to keep more risks in-house. Therefore, the insurer will opt for less reinsurance

Figure \ref{fig9} demonstrates the impacts of the intensity of the number of claims $\lambda_{1}$ upon the equilibrium reinsurance strategy. From Figure \ref{fig9}, we discover that the equilibrium reinsurance strategy remains unchanged regardless of the values of $\lambda_{1}$.

Figure \ref{fig10} illustrates the impacts of the first-order moment of the claim sizes $\mu_{1}$ upon the equilibrium reinsurance strategy, which shows a positive relationship between $\hat{q}(t)$ and $\mu_{1}$. Since the growth of $\mu_{1}$, which implies a greater premium for the insurer to remit to the reinsurer, prompts the insurer to hold onto more risks and consequently cut down on the amount of reinsurance purchased.

By observing Figure \ref{fig11}, a negative correlation between $\hat{q}(t)$ and $\mu_{2}$ can be discerned. This implies that the insurer will procure more reinsurance when $\mu_{2}$ rises. If there is an increase in $\mu_{2}$, then the insurer will face greater uncertainties. Consequently, the insurer tends to buy additional reinsurance to lower the risks.

\section{Conclusions}
This paper focused on the study of the reinsurance and investment problem regarding the insurer with random risk aversion under Heston's SV model. The insurer's surplus process is supposed to satisfy a diffusion approximation model, and the insurer has the option to transfer a portion of the risks to a reinsurer via proportional reinsurance or to expand by acquiring new business. Furthermore, the insurer can invest in the financial market which comprises one risk-free asset and one risky asset with its price depicted by Heston's SV model. Firstly, we formulated the general reinsurance and investment problem via expected certainty equivalents and provided the verification theorem within the game theoretic framework. Secondly, we solved the general reinsurance and investment problem. Within the exponential utility, we derived semi-analytic expressions for the equilibrium reinsurance and investment strategies and the value function under the cases of $n$ possible risk aversions and one risk aversion, respectively. Finally, we provided some numerical simulations to present the influences of model parameters. The principal discoveries are presented below: (i) the value function is properly defined under some mild assumptions, and the equilibrium reinsurance-investment strategies are deterministic and state-independent; (ii) the parameters related to the risky asset have no bearing on the equilibrium reinsurance strategy, while the parameters of the insurance market bear no relation to the equilibrium investment strategy; (iii) the equilibrium reinsurance and investment strategies are affected by the random risk aversion, and the insurer having a greater expected value of risk aversion is inclined to purchase a larger amount of reinsurance and make a smaller investment in the risky asset; (iv) the equilibrium reinsurance strategies are typically increasing over time while the equilibrium investment strategies are representatively declining with time, which backs up the idea of diminishing life-cycle investment profiles for the exponential utility.

There are some interesting directions that deserve further investigation. One is to adopt other SV models to capture the financial market such as constant elasticity of variance models \cite{Li2014}, and to consider other varieties of reinsurance like excess-of-loss reinsurance \cite{Bai2010}. Moreover, many researchers pointed out that due to the occurrence of the financial crisis in 2008, the investor/insurer should attach more importance to default risk management \cite{Deng2018}. In addition, for an insurer, an optimal reinsurance contract might not be so for a reinsurer because the reinsurer is also bent on achieving its own profit and controlling risks. Consequently, the problem of optimal reinsurance and investment between the two parties has captured the interest of some scholars in the context of game theory \cite{Deng2018, Li2014}. Therefore, another direction is to consider random risk aversion and default risk for the reinsurance-investment problems and the reinsurance-investment games. We hope to discuss these problems in future research.

\section*{Appendices}
\appendix
\renewcommand{\appendixname}{Appendix~\Alph{section}}
\section{Proof of Theorem \ref{theorem3.1}}
\begin{proof}
This proof mainly utilizes the proof approaches of Theorem 15.1 in \cite{Bjork2021} and Theorem 3.5 in \cite{Desmettre2023}. For the sake of the paper's integrity, we provide the subsequent details that are comprised of three steps.

Step 1. As by assumption for all $\gamma$, $U, Y^{\gamma}, H\in C^{1,2,2}(\mathcal{Q})$, we can obtain the continuity of $\hat{u}$ and then derive the continuity of the drift and diffusion coefficients of $X^{\hat{u}}$. Moreover, since $Y^{\gamma}$ satisfies the condition (A) for all $\gamma$, then it follows from the Feynman-Kac therorem \cite{Yong1999} and the equation \eqref{eq16} that
\begin{equation}\label{eq19}
Y^{\gamma}(t,x,v)=\mathbb{E}_{t}\left[\varphi^{\gamma}(X^{\hat{u}}(T))\right]= y^{\hat{u},\gamma}(t,x,v)\; \text{for all}\;\gamma.
\end{equation}

Step 2. Substituting $u=\hat{u}$ into the pseudo HJB \eqref{eq15} yields
\begin{align*}
\mathcal{A}^{\hat{u}}U(t,x,v)-\mathcal{A}^{\hat{u}}H(t,x,v)
+\int\iota^{\gamma}(Y^{\gamma}(t,x,v))\mathcal{A}^{\hat{u}}Y^{\gamma}(t,x,v)\mathrm{d}\Gamma(\gamma)=0.
\end{align*}
Plugging in the equation \eqref{eq16} then reads
\begin{align}\label{eq20}
\mathcal{A}^{\hat{u}}U(t,x,v)-\mathcal{A}^{\hat{u}}H(t,x,v)=0.
\end{align}
For $U\in C^{1,2,2}(\mathcal{Q})$ satisfying the condition (A), by using the It\^{o} formula, we can derive that
\begin{align}\label{eq21}
\mathbb{E}_{t}\left[U(T,X^{\hat{u}}(T),V(T))\right]=&U(t,x,v)+\mathbb{E}_{t}\left[\int^{T}_{t}\bigg(U_{s}(s,X^{\hat{u}}(s),V(s))+(rX^{\hat{u}}(s)+a\eta+a\eta_{2}q(s)+\xi V(s)\pi(s))\right.\nonumber\\
&\times U_{x}(s,X^{\hat{u}}(s),V(s))+0.5(b^{2}q^{2}(s)+\pi^{2}(s)V(s))U_{xx}(s,X^{\hat{u}}(s),V(s))\nonumber\\
&+\kappa(\theta-V(s))U_{v}(s,X^{\hat{u}}(s),V(s))+0.5\sigma^{2}V(s)U_{vv}(s,X^{\hat{u}}(s),V(s))\nonumber\\
&+\rho\pi(s)\sigma V(s)U_{xv}(s,X^{\hat{u}}(s),V(s))\bigg)\mathrm{d}s\bigg].
\end{align}
Analogously, for $H\in C^{1,2,2}(\mathcal{Q})$ meeting the condition (A), we have
\begin{align}\label{eq22}
\mathbb{E}_{t}\left[H(T,X^{\hat{u}}(T),V(T))\right]=&H(t,x,v)+\mathbb{E}_{t}\left[\int^{T}_{t}\bigg(H_{s}(s,X^{\hat{u}}(s),V(s))+(rX^{\hat{u}}(s)+a\eta+a\eta_{2}q(s)+\xi V(s)\pi(s))\right.\nonumber\\
&\times H_{x}(s,X^{\hat{u}}(s),V(s))+0.5(b^{2}q^{2}(s)+\pi^{2}(s)V(s))H_{xx}(s,X^{\hat{u}}(s),V(s))\nonumber\\
&+\kappa(\theta-V(s))H_{v}(s,X^{\hat{u}}(s),V(s))+0.5\sigma^{2}V(s)H_{vv}(s,X^{\hat{u}}(s),V(s))\nonumber\\
&+\rho\pi(s)\sigma V(s)H_{xv}(s,X^{\hat{u}}(s),V(s))\bigg)\mathrm{d}s\bigg].
\end{align}
Note that equations \eqref{eq21} and \eqref{eq22} hold in general for $U, H\in C^{1,2,2}(\mathcal{Q})$ satisfying the condition (A) and rely on neither the equation \eqref{eq15} nor the equation \eqref{eq16}. Since $U(T,x,v)=H(T,x,v)$, it follows from equations \eqref{eq12}, \eqref{eq17} and \eqref{eq19}-\eqref{eq22} that
\begin{align}\label{eq23}
U(t,x,v)=H(t,x,v)=\int(\varphi^{\gamma})^{-1}\left(Y^{\gamma}(t,x,v)\right)\mathrm{d}\Gamma(\gamma)
=\int(\varphi^{\gamma})^{-1}\left(y^{\hat{u},\gamma}(t,x,v)\right)\mathrm{d}\Gamma(\gamma)=J^{\hat{u}}(t,x,v).
\end{align}

Step 3. Similar to Lemma 3.8 of \cite{Bjork2014}, one can obtain that
\begin{align*}
J^{u_{h}}(t,x,v)=&\mathbb{E}_{t}\left[J^{u_{h}}(t+h,X^{u_{h}}(t+h),V(t+h))\right]
-\mathbb{E}_{t}\left[\int(\varphi^{\gamma})^{-1}\left(y^{u_{h},\gamma}(t+h,X^{u_{h}}(t+h),V(t+h))\right)\mathrm{d}\Gamma(\gamma)\right]\\
&+\int(\varphi^{\gamma})^{-1}\left(\mathbb{E}_{t}\left[y^{u_{h},\gamma}(t+h,X^{u_{h}}(t+h),V(t+h))\right]\right)\mathrm{d}\Gamma(\gamma).
\end{align*}
Since $u_{h}=u$ on $[t,t+h)$, one has by continuity of $X^{u_{h}}(s)$ for $s\leq t$ that $X^{u_{h}}(t+h)=X^{u}(t+h)$. Moreover, since $u_{h}=\hat{u}$ on $[t+h, T]$, one also has by \eqref{eq23} that
$$
J^{u_{h}}(t+h,X^{u_{h}}(t+h),V(t+h))=U(t+h,X^{u}(t+h),V(t+h))
$$
and
$$
y^{u_{h},\gamma}(t+h,X^{u_{h}}(t+h),V(t+h))=y^{\hat{u},\gamma}(t+h,X^{u}(t+h),V(t+h)).
$$
Therefore, we can derive that
\begin{align}\label{eq29}
J^{u_{h}}(t,x,v)=&\mathbb{E}_{t}\left[U(t+h,X^{u}(t+h),V(t+h))\right]
-\mathbb{E}_{t}\left[\int(\varphi^{\gamma})^{-1}\left(y^{\hat{u},\gamma}(t+h,X^{u}(t+h),V(t+h))\right)\mathrm{d}\Gamma(\gamma)\right]\nonumber\\
&+\int(\varphi^{\gamma})^{-1}\left(\mathbb{E}_{t}\left[y^{\hat{u},\gamma}(t+h,X^{u}(t+h),V(t+h))\right]\right)\mathrm{d}\Gamma(\gamma).
\end{align}
In addition, by the pseudo HJB \eqref{eq15} and the equation \eqref{eq19}, we have
\begin{align}\label{eq24}
\mathcal{A}^{u}U(t,x,v)-\mathcal{A}^{u}H(t,x,v)
+\int\iota^{\gamma}(y^{\hat{u},\gamma}(t,x,v))\mathcal{A}^{u}y^{\hat{u},\gamma}(t,x,v)\mathrm{d}\Gamma(\gamma)\leq0.
\end{align}
Since the admissible reinsurance and investment strategy $u(t,x,v)$ is continuous, by first applying the It\^{o} formula for a fixed $h>0$ and then taking the limit $h\rightarrow0$, we can deduce the following identities
\begin{align}\label{eq25}
&\mathbb{E}_{t}\left[U(t+h,X^{u}(t+h),V(t+h))\right]-U(t,x,v)\nonumber\\
=&h\left\{U_{t}(t,x,v)+(rx+a\eta+a\eta_{2}q(t,x,v)+\xi v\pi(t,x,v))U_{x}(t,x,v)+0.5(b^{2}q^{2}(t,x,v)+\pi^{2}(t,x,v)v)U_{xx}(t,x,v)\right.\nonumber\\
&\left.\quad+\kappa(\theta-v)U_{v}(t,x,v)+0.5\sigma^{2}vU_{vv}(t,x,v)+\rho\pi(t,x,v)\sigma vU_{xv}(t,x,v)\right\}+o(h)
\end{align}
and
\begin{align}\label{eq26}
&\mathbb{E}_{t}\left[H(t+h,X^{u}(t+h),V(t+h))\right]-H(t,x,v)\nonumber\\
=&h\left\{H_{t}(t,x,v)+(rx+a\eta+a\eta_{2}q(t,x,v)+\xi v\pi(t,x,v))H_{x}(t,x,v)+0.5(b^{2}q^{2}(t,x,v)+\pi^{2}(t,x,v)v)H_{xx}(t,x,v)\right.\nonumber\\
&\left.\quad+\kappa(\theta-v)H_{v}(t,x,v)+0.5\sigma^{2}vH_{vv}(t,x,v)+\rho\pi(t,x,v)\sigma vH_{xv}(t,x,v)\right\}+o(h)
\end{align}
as well as
\begin{align}\label{eq27}
&\mathbb{E}_{t}\left[y^{\hat{u},\gamma}(t+h,X^{u}(t+h),V(t+h))\right]-y^{\hat{u},\gamma}(t,x,v)\nonumber\\
=&h\left\{y^{\hat{u},\gamma}_{t}(t,x,v)+(rx+a\eta+a\eta_{2}q(t,x,v)+\xi v\pi(t,x,v))y^{\hat{u},\gamma}_{x}(t,x,v)+0.5(b^{2}q^{2}(t,x,v)+\pi^{2}(t,x,v)v)y^{\hat{u},\gamma}_{xx}(t,x,v)\right.\nonumber\\
&\left.\quad+\kappa(\theta-v)y^{\hat{u},\gamma}_{v}(t,x,v)+0.5\sigma^{2}vy^{\hat{u},\gamma}_{vv}(t,x,v)+\rho\pi(t,x,v)\sigma vy^{\hat{u},\gamma}_{xv}(t,x,v)\right\}+o(h),
\end{align}
where $o(h)$ denotes a higher-order infinitesimal. Furthermore, by the equation \eqref{eq27}, we can get that
\begin{align}\label{eq28}
&\int(\varphi^{\gamma})^{-1}\left(\mathbb{E}_{t}\left[y^{\hat{u},\gamma}(t+h,X^{u}(t+h),V(t+h))\right]\right)\mathrm{d}\Gamma(\gamma)
-\int(\varphi^{\gamma})^{-1}\left(y^{\hat{u},\gamma}(t,x,v)\right)\mathrm{d}\Gamma(\gamma)
\nonumber\\
=&h\left\{\int\iota^{\gamma}(y^{\hat{u},\gamma}(t,x,v))\bigg(y^{\hat{u},\gamma}_{t}(t,x,v)+(rx+a\eta+a\eta_{2}q(t,x,v)+\xi v\pi(t,x,v))y^{\hat{u},\gamma}_{x}(t,x,v)\right.\nonumber\\
&\quad+0.5(b^{2}q^{2}(t,x,v)+\pi^{2}(t,x,v)v)y^{\hat{u},\gamma}_{xx}(t,x,v)+\kappa(\theta-v)y^{\hat{u},\gamma}_{v}(t,x,v)\nonumber\\
&\quad+0.5\sigma^{2}vy^{\hat{u},\gamma}_{vv}(t,x,v)+\rho\pi(t,x,v)\sigma vy^{\hat{u},\gamma}_{xv}(t,x,v)\bigg)\mathrm{d}\Gamma(\gamma)\bigg\}+o(h).
\end{align}
Combining equations \eqref{eq25}, \eqref{eq26} and \eqref{eq28} with the equation \eqref{eq24} yields
\begin{align*}
\mathbb{E}_{t}&\left[U(t+h,X^{u}(t+h),V(t+h))\right]-U(t,x,v)-\mathbb{E}_{t}\left[H(t+h,X^{u}(t+h),V(t+h))\right]+H(t,x,v)\\
&+\int(\varphi^{\gamma})^{-1}\left(\mathbb{E}_{t}\left[y^{\hat{u},\gamma}(t+h,X^{u}(t+h),V(t+h))\right]\right)\mathrm{d}\Gamma(\gamma)
-\int(\varphi^{\gamma})^{-1}\left(y^{\hat{u},\gamma}(t,x,v)\right)\mathrm{d}\Gamma(\gamma)\leq o(h).
\end{align*}
Then, it follows from equations \eqref{eq17}, \eqref{eq23} and \eqref{eq29} that
\begin{align*}
U(t,x,v)\geq&\mathbb{E}_{t}\left[U(t+h,X^{u}(t+h),V(t+h))\right]-\mathbb{E}_{t}\left[H(t+h,X^{u}(t+h),V(t+h))\right]\\
&+\int(\varphi^{\gamma})^{-1}\left(\mathbb{E}_{t}\left[y^{\hat{u},\gamma}(t+h,X^{u}(t+h),V(t+h))\right]\right)\mathrm{d}\Gamma(\gamma)
+o(h)\\
=&J^{u_{h}}(t,x,v)+o(h)
\end{align*}
and so
$$
J^{\hat{u}}(t,x,v)\geq J^{u_{h}}(t,x,v)+o(h).
$$
Therefore,
$$
\underset{h\rightarrow0}\liminf\frac{J^{\hat{u}}(t,x,v)-J^{u_{h}}(t,x,v)}{h}\geq0,
$$
which means that $\hat{u}$ is an equilibrium reinsurance and investment strategy. Thus, we have $V(t,x,v)=U(t,x,v)$, which finalizes the proof.
\end{proof}

\section{Proof of Proposition \ref{proposition4.1}}
\begin{proof}
For $\gamma_{i}>0\;(i=1,\cdot\cdot\cdot,n)$, we have  $Y^{\gamma_{i}}(t,x,v)=-\frac{1}{\gamma_{i}}e^{g_{1}^{\gamma_{i}}(t)x+g_{2}^{\gamma_{i}}(t)v+g_{3}^{\gamma_{i}}(t)}$. Note that for $\gamma_{i}$ with $i=1,\cdot\cdot\cdot,n$, $g_{1}^{\gamma_{i}}(t)=-\gamma_{i} e^{r(T-t)}$ and $g_{2}^{\gamma_{i}}(t)$ and $g_{3}^{\gamma_{i}}(t)$ have continuous and non-explosive solutions. Thus, one has $Y^{\gamma_{i}}\in C^{1,2,2}(\mathcal{Q})$. Moreover, it follows from \eqref{eq43} that the candidate equilibrium reinsurance and investment strategy $\hat{q}(t)$ and $\hat{\pi}(t)$ are deterministic, continuous and state-independent and $\hat{q}(t)>0$ for $\forall \;t\in[0,T]$. Then, the condition (ii) in Definition \ref{definition3.1} is met. Substituting the equation \eqref{eq43} into the equation \eqref{eq6} yields
\begin{align}\label{eq49}
    X^{\hat{u}}(t)=&x_{0}e^{rt}+\int^{t}_{0}e^{r(t-s)}\left[a\eta+a\eta_{2}\hat{q}(s)+\xi V(s)\hat{\pi}(s)\right]\mathrm{d}s\nonumber\\
   &+\int^{t}_{0}e^{r(t-s)}\left[b\hat{q}(s)\mathrm{d}W_{0}(s)+\hat{\pi}(s)\sqrt{V(s)}\mathrm{d}W_{1}(s)\right],
\end{align}
which means that the condition (i) in Definition \ref{definition3.1} is satisfied. By the proof process of Lemma 1 in \cite{Li2012}, we know that $\forall \varrho\in[1,\infty)$,
\begin{equation}\label{eq51}
\mathbb{E}_{t}\left[\sup_{s\in[t,T]}|V(s)|^{\varrho}\right]<\infty.
\end{equation}
Applying the Burkh\"{o}lder-Davis-Gundy inequality, we have for some constants $K>0$ that
\begin{align}\label{eq52}
\mathbb{E}\left[\sup_{t\in[0,T]}|X^{\hat{u}}(t)|^{2}\right]\leq&K\mathbb{E}\left[\sup_{t\in[0,T]}\left|\int^{t}_{0}\left(a\eta+a\eta_{2}\hat{q}(s)+\xi V(s)\hat{\pi}(s)\right)\mathrm{d}s\right|^{2}\right.\nonumber\\
   &+\sup_{t\in[0,T]}\left|\int^{t}_{0}b\hat{q}(s)\mathrm{d}W_{0}(s)\right|^{2}
+\left.\sup_{t\in[0,T]}\left|\int^{t}_{0}\hat{\pi}(s)\sqrt{V(s)}\mathrm{d}W_{1}(s)\right|^{2}\right]\nonumber\\
\leq&K\mathbb{E}\left[\sup_{t\in[0,T]}\left|\int^{t}_{0}\left(a\eta+a\eta_{2}\hat{q}(s)\right)\mathrm{d}s\right|^{2}
+\sup_{t\in[0,T]}\left|V(t)\right|^{2}\times\sup_{t\in[0,T]}\left|\int^{t}_{0}\xi \hat{\pi}(s)\mathrm{d}s\right|^{2}\right.\nonumber\\
   &+\left.\int^{T}_{0}b^{2}\hat{q}^{2}(s)\mathrm{d}s
+\int^{T}_{0}\hat{\pi}^{2}(s)V(s)\mathrm{d}s\right]
\end{align}
with
\begin{align}\label{eq53}
\mathbb{E}\left[\int^{T}_{0}\hat{\pi}^{2}(s)V(s)\mathrm{d}s\right]\leq&
\mathbb{E}\left[\sup_{t\in[0,T]}\left|V(t)\right|\times\int^{T}_{0}\hat{\pi}^{2}(s)\mathrm{d}s\right].
\end{align}
By the equation \eqref{eq51} and the boundness of $\hat{q}(t)$ and $\hat{\pi}(t)$, one can obtain that
\begin{equation}\label{eq54}
\mathbb{E}\left[\sup_{t\in[0,T]}|X^{\hat{u}}(t)|^{2}\right]<\infty
\end{equation}
and then
\begin{align}\label{eq63}
\mathbb{E}\left[\sup_{t\in[0,T]}|X^{\hat{u}}(t)|\right]\leq\left(\mathbb{E}\left[\left(\sup_{t\in[0,T]}|X^{\hat{u}}(t)|\right)^{2}\right]\right)^{\frac{1}{2}}
\leq\left(\mathbb{E}\left[\sup_{t\in[0,T]}|X^{\hat{u}}(t)|^{2}\right]\right)^{\frac{1}{2}}<\infty.
\end{align}

Inserting the equation \eqref{eq49} into $Y^{\gamma_{i}}(t,x,v)$, we derive the following estimate
\begin{align}\label{eq55}
\left|Y^{\gamma_{i}}(t,x,v)\right|^{4}=&\left|\frac{1}{\gamma_{i}^{4}}e^{4g_{1}^{\gamma_{i}}(t)x+4g_{2}^{\gamma_{i}}(t)v+4g_{3}^{\gamma_{i}}(t)}\right|\nonumber\\
\leq&Ke^{4g_{1}^{\gamma_{i}}(t)X^{\hat{u}}(t)}\nonumber\\
=&Ke^{4g_{1}^{\gamma_{i}}(t)\left(x_{0}e^{rt}+\int^{t}_{0}e^{r(t-s)}\left[a\eta+a\eta_{2}\hat{q}(s)+\xi V(s)\hat{\pi}(s)\right]\mathrm{d}s
   +\int^{t}_{0}e^{r(t-s)}\left[b\hat{q}(s)\mathrm{d}W_{0}(s)+\hat{\pi}(s)\sqrt{V(s)}\mathrm{d}W_{1}(s)\right]\right)}\nonumber\\
   \leq&Ke^{-4\gamma_{i}\int^{t}_{0}\left[\xi V(s)\bar{\pi}(s)\mathrm{d}s
   +b\bar{q}(s)\mathrm{d}W_{0}(s)+\bar{\pi}(s)\sqrt{V(s)}\mathrm{d}W_{1}(s)\right]},
\end{align}
where
\begin{equation}\label{eq58}
  \bar{\pi}(s)=\frac{\xi+\rho\sigma \sum\limits_{i=1}^{n}g_{2}^{\gamma_{i}}(s) p_{i}}{\sum\limits_{i=1}^{n}\gamma_{i} p_{i}},\;\;
  \bar{q}(s)=\frac{a \eta_{2}}{b^{2}\sum\limits_{i=1}^{n}\gamma_{i} p_{i}}.
\end{equation}
The first inequality in the equation \eqref{eq55} is valid because $g_{2}^{\gamma_{i}}(t)\leq0$ and $g_{3}^{\gamma_{i}}(t)$ are deterministic and bounded, and the second inequality follows from the fact that $g_{1}^{\gamma_{i}}(t)$, $x_{0}e^{rt}$ and $\int^{t}_{0}e^{r(t-s)}\left[a\eta+a\eta_{2}\hat{q}(s)\right]\mathrm{d}s$ are deterministic and bounded. Now, we consider the integral $e^{-4\gamma_{i}\int^{t}_{0}b\bar{q}(s)\mathrm{d}W_{0}(s)}$. Since
$$
e^{-4\gamma_{i}\int^{t}_{0}b\bar{q}(s)\mathrm{d}W_{0}(s)}=\underbrace{e^{\int^{t}_{0}8\gamma_{i}^{2}b^{2}\bar{q}^{2}(s)\mathrm{d}s}}_{constant}\times
\underbrace{e^{-\int^{t}_{0}8\gamma_{i}^{2}b^{2}\bar{q}^{2}(s)\mathrm{d}s-4\gamma_{i}\int^{t}_{0}b\bar{q}(s)\mathrm{d}W_{0}(s)}}_{martingale},
$$
one can obtain that
\begin{equation}\label{eq60}
\mathbb{E}\left[e^{-4\gamma_{i}\int^{t}_{0}b\bar{q}(s)\mathrm{d}W_{0}(s)}\right]<\infty.
\end{equation}
Next, we attempt to find an estimate for $e^{-4\gamma_{i}\int^{t}_{0}\left(\xi V(s)\bar{\pi}(s)\mathrm{d}s
  +\bar{\pi}(s)\sqrt{V(s)}\mathrm{d}W_{1}(s)\right)}$. We find that
\begin{align*}
e^{-4\gamma_{i}\int^{t}_{0}\left(\xi V(s)\bar{\pi}(s)\mathrm{d}s
  +\bar{\pi}(s)\sqrt{V(s)}\mathrm{d}W_{1}(s)\right)}
=&\underbrace{e^{\int^{t}_{0}\left(-4\gamma_{i}\xi\bar{\pi}(s)+16\gamma_{i}^{2}\bar{\pi}^{2}(s)\right)V(s)\mathrm{d}s}}_{B}\\
&\times\underbrace{e^{-\int^{t}_{0}16\gamma_{i}^{2}\bar{\pi}^{2}(s)V(s)\mathrm{d}s
  -\int^{t}_{0}4\gamma_{i}\bar{\pi}(s)\sqrt{V(s)}\mathrm{d}W_{1}(s)}}_{C}.
\end{align*}
For the term $C$, since $\gamma_{i}\bar{\pi}(s)$ is  deterministic and bounded on $[0,T]$, it follows from Lemma 4.3 in \cite{ZengX2013} that
\begin{equation}\label{eq56}
  \mathbb{E}[C^{2}]=\mathbb{E}\left[e^{-\int^{t}_{0}32\gamma_{i}^{2}\bar{\pi}^{2}(s)V(s)\mathrm{d}s
  -\int^{t}_{0}8\gamma_{i}\bar{\pi}(s)\sqrt{V(s)}\mathrm{d}W_{1}(s)}\right]<\infty.
\end{equation}
For the term $B$, one has
\begin{equation}\label{eq57}
  \mathbb{E}[B^{2}]=\mathbb{E}\left[e^{\int^{t}_{0}\left(-8\gamma_{i}\xi\bar{\pi}(s)+32\gamma_{i}^{2}\bar{\pi}^{2}(s)\right)V(s)\mathrm{d}s}\right].
\end{equation}
By Theorem 5.1 in \cite{ZengX2013}, if the assumption
\begin{equation}\label{eq59}
  -8\gamma_{i}\xi\bar{\pi}(s)+32\gamma_{i}^{2}\bar{\pi}^{2}(s)\leq\frac{\kappa^{2}}{2\sigma^{2}}
\end{equation}
is satisfied, one can obtain that $\mathbb{E}[B^{2}]<\infty$. Hence, applying equations \eqref{eq55}, \eqref{eq60} and \eqref{eq56} and $\mathbb{E}[B^{2}]<\infty$, we can arrive at
\begin{align*}
\mathbb{E}\left[\left|Y^{\gamma_{i}}(t,x,v)\right|^{4}\right]&\leq K\mathbb{E}\left[e^{-4\gamma_{i}\int^{t}_{0}\left[\xi V(s)\bar{\pi}(s)\mathrm{d}s
   +b\bar{q}(s)\mathrm{d}W_{0}(s)+\bar{\pi}(s)\sqrt{V(s)}\mathrm{d}W_{1}(s)\right]}\right]\\
&=K\mathbb{E}\left[e^{-4\gamma_{i}\int^{t}_{0}b\bar{q}(s)\mathrm{d}W_{0}(s)}\right]\times\mathbb{E}\left[e^{-4\gamma_{i}\int^{t}_{0}\left(\xi V(s)\bar{\pi}(s)\mathrm{d}s
  +\bar{\pi}(s)\sqrt{V(s)}\mathrm{d}W_{1}(s)\right)}\right]\\
  &\leq K \mathbb{E}\left[BC\right] \leq K \left(\mathbb{E}\left[B^{2}\right] \mathbb{E}\left[C^{2}\right]\right)^{\frac{1}{2}}<\infty
\end{align*}
for any $(t,x,v)\in\mathcal{Q}$, where the second equation holds because $W_{0}$ is independent of $W_{1}$ and $W_{2}$. Thus,
\begin{equation}\label{eq61}
  \mathbb{E}\left[\sup_{t\in[0,T]}\left|Y^{\gamma_{i}}(t,x,v)\right|^{4}\right]<\infty
\end{equation}
and then
\begin{equation}\label{eq62}
  \mathbb{E}\left[\sup_{t\in[0,T]}\left|Y^{\gamma_{i}}(t,x,v)\right|^{2}\right]\leq \left(\mathbb{E}\left[\left(\sup_{t\in[0,T]}\left|Y^{\gamma_{i}}(t,x,v)\right|^{2}\right)^{2}\right]\right)^{\frac{1}{2}}
  \leq\left(\mathbb{E}\left[\sup_{t\in[0,T]}\left|Y^{\gamma_{i}}(t,x,v)\right|^{4}\right]\right)^{\frac{1}{2}}<\infty.
\end{equation}

Subsequently, we will show that $Y^{\gamma_{i}}$ satisfies the condition (A). By the form of $Y^{\gamma_{i}}$, we have
\begin{align*}
\mathbb{E}\left[\int^{T}_{0}\left|Y^{\gamma_{i}}_{t}(t,x,v)\right|\mathrm{d}t\right]&=\mathbb{E}\left[\int^{T}_{0}\left|-\frac{1}{\gamma_{i}}e^{g_{1}^{\gamma_{i}}(t)x
+g_{2}^{\gamma_{i}}(t)v+g_{3}^{\gamma_{i}}(t)}\left[\frac{\partial g_{1}^{\gamma_{i}}(t)}{\partial t}x+\frac{\partial g_{2}^{\gamma_{i}}(t)}{\partial t}v+\frac{\partial g_{3}^{\gamma_{i}}(t)}{\partial t}\right]\right|\mathrm{d}t\right]\\
&\leq K\mathbb{E}\left[\sup_{t\in[0,T]}\left|Y^{\gamma_{i}}(t,x,v)\right|\times \sup_{t\in[0,T]}\left|X^{\hat{u}}(t)+V(t)+1\right|\right],
\end{align*}
where the inequality follows from the fact that $\frac{\partial g_{1}^{\gamma_{i}}(t)}{\partial t}$, $\frac{\partial g_{2}^{\gamma_{i}}(t)}{\partial t}$ and $\frac{\partial g_{3}^{\gamma_{i}}(t)}{\partial t}$ are deterministic and bounded. Moreover, by equations \eqref{eq54} and \eqref{eq62}, we can derive that
\begin{align*}
\mathbb{E}\left[\sup_{t\in[0,T]}\left|Y^{\gamma_{i}}(t,x,v)\right|\times \sup_{t\in[0,T]}\left|X^{\hat{u}}(t)\right|\right]\leq&
\left(\mathbb{E}\left[\left(\sup_{t\in[0,T]}\left|Y^{\gamma_{i}}(t,x,v)\right|\right)^{2}\right]\times
\mathbb{E}\left[\left(\sup_{t\in[0,T]}\left|X^{\hat{u}}(t)\right|\right)^{2}\right]\right)^{\frac{1}{2}}\\
\leq&
\left(\mathbb{E}\left[\sup_{t\in[0,T]}\left|Y^{\gamma_{i}}(t,x,v)\right|^{2}\right]\times
\mathbb{E}\left[\sup_{t\in[0,T]}\left|X^{\hat{u}}(t)\right|^{2}\right]\right)^{\frac{1}{2}}\\
<&\infty.
\end{align*}
Similarly, by equations \eqref{eq51} and \eqref{eq62}, we can also obtain
$$\mathbb{E}\left[\sup_{t\in[0,T]}\left|Y^{\gamma_{i}}(t,x,v)\right|\times \sup_{t\in[0,T]}\left|V(t)\right|\right]
<\infty, \quad \mathbb{E}\left[\sup_{t\in[0,T]}\left|Y^{\gamma_{i}}(t,x,v)\right|\right]
<\infty.$$
Therefore, one has $\mathbb{E}\left[\int^{T}_{0}\left|Y^{\gamma_{i}}_{t}(t,x,v)\right|\mathrm{d}t\right]<\infty$, which means $Y^{\gamma_{i}}_{t}(t,x,v)\in\mathbb{L}^{1}_{\mathcal{F}}(0,T;\mathbb{R})$. In the same way, one can prove that $(rx+a\eta+a\eta_{2}\hat{q}+\xi v\hat{\pi})Y^{\gamma_{i}}_{x},\;\kappa(\theta-v)Y^{\gamma_{i}}_{v},\;(b^{2}\hat{q}^{2}+\hat{\pi}^{2}v)Y^{\gamma_{i}}_{xx},\;\sigma^{2}vY^{\gamma_{i}}_{vv},\;
\rho\hat{\pi}\sigma vY^{\gamma_{i}}_{xv}\in\mathbb{L}^{1}_{\mathcal{F}}(0,T;\mathbb{R})$. Furthermore, since $g_{1}^{\gamma_{i}}(t)$ and $\hat{\pi}$ are deterministic and bounded, we have
\begin{align*}
\mathbb{E}\left[\int^{T}_{0}\left|\hat{\pi}\sqrt{v}Y^{\gamma_{i}}_{x}(t,x,v)\right|^{2}\mathrm{d}t\right]\leq& K\mathbb{E}\left[\sup_{t\in[0,T]}\left|Y^{\gamma_{i}}(t,x,v)\right|^{2}\times \sup_{t\in[0,T]}\left|V(t)\right|\right]\\
\leq&K
\left(\mathbb{E}\left[\left(\sup_{t\in[0,T]}\left|Y^{\gamma_{i}}(t,x,v)\right|^{2}\right)^{2}\right]\times
\mathbb{E}\left[\left(\sup_{t\in[0,T]}\left|V(t)\right|\right)^{2}\right]\right)^{\frac{1}{2}}\\
\leq&K
\left(\mathbb{E}\left[\sup_{t\in[0,T]}\left|Y^{\gamma_{i}}(t,x,v)\right|^{4}\right]\times
\mathbb{E}\left[\sup_{t\in[0,T]}\left|V(t)\right|^{2}\right]\right)^{\frac{1}{2}}\\
<&\infty,
\end{align*}
where the last inequality follows from equations \eqref{eq51} and \eqref{eq61}. Hence, $\hat{\pi}\sqrt{v}Y^{\gamma_{i}}_{x}(t,x,v)\in\mathbb{L}^{2}_{\mathcal{F}}(0,T;\mathbb{R})$. In a similar manner, we can show that $b\hat{q}Y^{\gamma_{i}}_{x}(t,x,v),\;\sigma\sqrt{v}Y^{\gamma_{i}}_{v}(t,x,v)\in\mathbb{L}^{2}_{\mathcal{F}}(0,T;\mathbb{R})$.

Note that
\begin{align*}
U(t,x,v)=H(t,x,v)=&\int(\varphi^{\gamma})^{-1}\left(Y^{\gamma}(t,x,v)\right)\mathrm{d}\Gamma(\gamma)\\
=&\int -\frac{1}{\gamma}\left[g_{1}^{\gamma}(t)x+g_{2}^{\gamma}(t)v+g_{3}^{\gamma}(t)\right]\mathrm{d}\Gamma(\gamma)\\
=&-\sum\limits_{i=1}^{n}\frac{1}{\gamma_{i}}\left[g_{1}^{\gamma_{i}}(t)x+g_{2}^{\gamma_{i}}(t)v+g_{3}^{\gamma_{i}}(t)\right]p_{i}.
\end{align*}
Since $g_{1}^{\gamma_{i}}(t)$, $g_{2}^{\gamma_{i}}(t)$ and $g_{3}^{\gamma_{i}}(t)$ are deterministic and bounded on $[0,T]$, one has $U,H\in C^{1,2,2}(\mathcal{Q})$ and $\int\left|(\varphi^{\gamma})^{-1}\left(Y^{\gamma}(t,x,v)\right)\right|\mathrm{d}\Gamma(\gamma)<\infty$ by equations \eqref{eq51} and \eqref{eq63}, which means that the condition (iii) in Definition \ref{definition3.1} is satisfied. In addition, because $\frac{\partial g_{1}^{\gamma_{i}}(t)}{\partial t}$, $\frac{\partial g_{2}^{\gamma_{i}}(t)}{\partial t}$ and $\frac{\partial g_{3}^{\gamma_{i}}(t)}{\partial t}$ are deterministic and bounded, we can derive that
\begin{align*}
\mathbb{E}\left[\int^{T}_{0}\left|U_{t}(t,x,v)\right|\mathrm{d}t\right]&=\mathbb{E}\left[\int^{T}_{0}\left|-\sum\limits_{i=1}^{n}\frac{1}{\gamma_{i}}\left[\frac{\partial g_{1}^{\gamma_{i}}(t)}{\partial t}x+\frac{\partial g_{2}^{\gamma_{i}}(t)}{\partial t}v+\frac{\partial g_{3}^{\gamma_{i}}(t)}{\partial t}\right]\right|\mathrm{d}t\right]\\
&\leq K\mathbb{E}\left[\sup_{t\in[0,T]}\left|X^{\hat{u}}(t)\right|+\sup_{t\in[0,T]}\left|V(t)\right|\right]+K\\
&<\infty,
\end{align*}
where the last inequality follows from equations \eqref{eq51} and \eqref{eq63}. Likewise, we can show that $(rx+a\eta+a\eta_{2}\hat{q}+\xi v\hat{\pi})U_{x},\;\kappa(\theta-v)U_{v},\;(b^{2}\hat{q}^{2}+\hat{\pi}^{2}v)U_{xx},\;\sigma^{2}vU_{vv},\;
\rho\hat{\pi}\sigma vU_{xv}\in\mathbb{L}^{1}_{\mathcal{F}}(0,T;\mathbb{R})$. By the boundedness of functions $g_{1}^{\gamma_{i}}(t)$ for all $\gamma_{i}$ and the equation \eqref{eq51}, one can obtain that
\begin{align*}
\mathbb{E}\left[\int^{T}_{0}\left|\hat{\pi}\sqrt{v}U_{x}(t,x,v)\right|^{2}\mathrm{d}t\right]\leq& K\mathbb{E}\left[\sup_{t\in[0,T]}\left|U_{x}(t,x,v)\right|^{2}\times \sup_{t\in[0,T]}\left|V(t)\right|\right]\\
\leq&K\mathbb{E}\left[\sup_{t\in[0,T]}\left|\sum\limits_{i=1}^{n}\frac{1}{\gamma_{i}}g_{1}^{\gamma_{i}}(t)p_{i}\right|^{2}\times \sup_{t\in[0,T]}\left|V(t)\right|\right]\\
\leq&K\mathbb{E}\left[ \sup_{t\in[0,T]}\left|V(t)\right|\right]\\
<&\infty,
\end{align*}
which implies that $\hat{\pi}\sqrt{v}U_{x}(t,x,v)\in\mathbb{L}^{2}_{\mathcal{F}}(0,T;\mathbb{R})$. Similarly, one can verify that $b\hat{q}U_{x}(t,x,v),\;\sigma\sqrt{v}U_{v}(t,x,v)\in\mathbb{L}^{2}_{\mathcal{F}}(0,T;\mathbb{R})$. Thus, $U$ satisfies the condition (A). Since $U=H$, $H$ also meets the condition (A).
\end{proof}

\end{document}